\documentclass[12pt,a4paper,reqno]{amsart}

\usepackage{tikz}
\usetikzlibrary{matrix,arrows,calc,snakes,patterns,decorations.markings}

\usepackage[mathscr]{eucal}
\usepackage{graphics,epic}
\usepackage{amsfonts}
\usepackage{amscd}
\usepackage{latexsym}
\usepackage{amsmath,amssymb, amsthm, stmaryrd, bm, bbm}
\usepackage[all,2cell]{xy}
\usepackage{mathrsfs}
\usepackage{setspace}

\theoremstyle{definition}
\newtheorem{theorem}{Theorem}[section]
\newtheorem*{theorem*}{Theorem}

\newtheorem{lemma}[theorem]{Lemma}
\newtheorem{proposition}[theorem]{Proposition}
\newtheorem{corollary}[theorem]{Corollary}

\newtheorem*{conjecture*}{Conjecture}
\newtheorem{question}[theorem]{Question}

\newtheorem{remark}[theorem]{Remark}

\newcommand{\ie}{{\em i.e.}\ }

\newcommand{\eg}{{\em e.g.}\ }
\newcommand{\ko}{\: , \;}

\newcommand{\opname}[1]{\operatorname{\mathsf{#1}}}

\renewcommand{\mod}{\opname{mod}\nolimits}

\newcommand{\proj}{\opname{proj}\nolimits}

\newcommand{\Mod}{\opname{Mod}\nolimits}

\newcommand{\per}{\opname{per}\nolimits}
\newcommand{\add}{\opname{add}\nolimits}

\renewcommand{\Im}{\opname{Im}\nolimits}
\newcommand{\Ker}{\opname{Ker}\nolimits}
\newcommand{\Fac}{\opname{Fac}\nolimits}

\newcommand{\thick}{\opname{thick}\nolimits}

\newcommand{\ind}{\opname{ind}}

\def\Mut{{\mathrm{Mut}}}
\def\fl{{\longrightarrow}\,}
\def\ens#1{\left\{ #1 \right\}}
\def\ccc{\mathbf{c}}
\def\gg{\mathbf{g}}

\renewcommand{\ker}{\opname{ker}\nolimits}

\newcommand{\Hom}{\opname{Hom}}

\newcommand{\Ext}{\opname{Ext}}

\newcommand{\End}{\opname{End}}

\newcommand{\cone}{\opname{cone}}

\newcommand{\std}{\mathrm{std}}

\newcommand{\ca}{{\mathcal A}}
\newcommand{\cb}{{\mathcal B}}
\newcommand{\cc}{{\mathcal C}}
\newcommand{\cd}{{\mathcal D}}
\newcommand{\ce}{{\mathcal E}}
\newcommand{\cf}{{\mathcal F}}

\newcommand{\ch}{{\mathcal H}}
\newcommand{\ci}{{\mathcal I}}

\newcommand{\cm}{{\mathcal M}}

\newcommand{\cp}{{\mathcal P}}

\newcommand{\cs}{{\mathcal S}}
\newcommand{\ct}{{\mathcal T}}
\newcommand{\cu}{{\mathcal U}}

\newcommand{\cx}{{\mathcal X}}

\renewcommand{\hat}[1]{\widehat{#1}}

\newcommand{\del}{\partial}

\newcommand{\tilt}{\opname{tilt}\nolimits}
\newcommand{\silt}{\opname{silt}\nolimits}
\newcommand{\twosilt}{\opname{2-silt}\nolimits}
\newcommand{\rtwosilt}{\opname{r.2-silt}\nolimits}
\newcommand{\tstr}{\opname{t-str}\nolimits}
\newcommand{\intertstr}{\opname{int-t-str}\nolimits}
\newcommand{\rintertstr}{\opname{r.int-t-str}\nolimits}
\newcommand{\cotstr}{\opname{co-t-str}\nolimits}
\newcommand{\intercotstr}{\opname{int-co-t-str}\nolimits}
\newcommand{\rintercotstr}{\opname{r.int-co-t-str}\nolimits}
\newcommand{\smc}{\opname{smc}\nolimits}
\newcommand{\intersmc}{\opname{2-smc}\nolimits}
\newcommand{\rintersmc}{\opname{r.2-smc}\nolimits}
\newcommand{\sttilt}{\opname{s\text{$\tau$}-tilt}\nolimits}
\newcommand{\rsttilt}{\opname{r.s\text{$\tau$}-tilt}\nolimits}
\newcommand{\fftor}{\opname{f-tors}\nolimits}
\newcommand{\rfftor}{\opname{r.f-tors}\nolimits}
\newcommand{\cto}{\opname{c-tilt}\nolimits}
\newcommand{\rcto}{\opname{r.c-tilt}\nolimits}
\newcommand{\mut}{\opname{mut}\nolimits}
\newcommand{\cmat}{\opname{c-mat}\nolimits}
\newcommand{\Cl}{\opname{Cl}\nolimits}
\newcommand{\gmat}{\opname{g-mat}\nolimits}

\newcommand{\sbullet}{\scriptstyle\bullet}
\newcommand{\scirc}{\scriptstyle\circ}

\numberwithin{equation}{section}

\begin{document}
\date{\today}

\title[Ordered Exchange Graphs]{Ordered Exchange Graphs}
\dedicatory{Dedicated to the memory of Dieter Happel}
\thanks{TB was supported by NSERC, Bishop's University and Universit\'e de Sherbrooke. DY was supported by DFG - SPP Darstellungstheorie KO1281/9-1 and JSPS.}
\author{Thomas Br\"ustle}
\address{Thomas Br\"ustle, D\'epartement de Math\'ematiques, Universit\'e de Sherbrooke, Sherbrooke, Canada, J1K 2R1 and Department of Mathematics, Bishops University, Sherbrooke, Canada, J1M 1Z7}
\email{thomas.brustle@usherbrooke.ca and tbruestl@ubishops.ca}

\author{Dong Yang}
\address{Dong Yang, Department of Mathematics, Nanjing University,
22 Hankou Road, Nanjing 210093, P.R.China}

\email{yangdong@nju.edu.cn}

\begin{abstract} The exchange graph of a  cluster algebra encodes the combinatorics of mutations of clusters. Through the recent "categorifications" of cluster algebras using representation theory one obtains a whole variety of exchange graphs associated with objects such as a finite-dimensional algebra or a differential graded algebra concentrated in non-positive degrees. These constructions often come from variations of the concept of tilting, the vertices of the exchange graph being torsion pairs, $t$-structures, silting objects, support $\tau$-tilting modules and so on. All these exchange graphs stemming from representation theory have the additional feature that they are the Hasse quiver of a partial order which is naturally defined for the objects. In this sense, the exchange graphs studied in this article can be considered as a generalization or as a completion of the poset of tilting modules which has been studied by Happel and Unger.
The goal of this article is to axiomatize the thus obtained structure of an ordered exchange graph, to present the various constructions of ordered exchange graphs and to relate them among each other.
\end{abstract}

\maketitle

\tableofcontents

\section{The goal} 

\begin{figure}
\[\hspace{7pt}\begin{xy} 0;<0.25pt,0pt>:<0pt,-0.25pt>::
(50,0) *+{\framebox(73,33){\parbox{70pt}{\tiny \begin{spacing}{0.7}reachable isoclasses of basic $2$-term silting objects in $\per(\Gamma)$\end{spacing}}}} ="0",
(700,0) *+{\framebox(73,33){\parbox{71pt}{\tiny \begin{spacing}{0.7} reachable intermediate bounded
co-$t$-structures in $\per(\Gamma)$\end{spacing}}}}="1",
(300,250) *+{\framebox(73,40){\parbox{70pt}{\tiny \begin{spacing}{0.7} reachable intermediate bounded
$t$-structures with length heart in $\cd_{fd}(\Gamma)$\end{spacing}}}}="2",
(950,250) *+{\framebox(80,40){\parbox{75pt}{\tiny \begin{spacing}{0.7} reachable isoclasses of intermediate simple-minded collections in $\cd_{fd}(\Gamma)$\end{spacing}}}}="3",
(40,500) *+{\framebox(80,35){\parbox{75pt}{\tiny \begin{spacing}{0.7} reachable isoclasses of basic $2$-term silting objects of $\ch^b(\proj J)$\end{spacing}}}}="4",
(700,500) *+{\framebox(73,33){\parbox{71pt}{\tiny \begin{spacing}{0.7} reachable intermediate bounded
co-$t$-structures in $\ch^b(\proj J)$\end{spacing}}}}="5",
(300,750) *+{\framebox(73,40){\parbox{71pt}{\tiny \begin{spacing}{0.7} reachable intermediate bounded
$t$-structures with length heart in $\cd^b(\mod J)$\end{spacing}}}}="6",
(940,750) *+{\framebox(80,40){\parbox{75pt}{\tiny \begin{spacing}{0.7} reachable isoclasses of intermediate simple-minded collections in $\cd^b(\mod J)$\end{spacing}}}}="7",
(250,1000) *+{\framebox(73,35){\parbox{70pt}{\tiny \begin{spacing}{0.7} reachable isoclasses of basic support $\tau$-tilting modules over $J$\end{spacing}}}}="8",
(750,1000) *+{\framebox(73,30){\parbox{70pt}{\tiny \begin{spacing}{0.7} reachable functorially finite torsion pairs of $\mod J$\end{spacing}}}}="9",
(500,1250) *+{\framebox(85,30){\parbox{82pt}{\tiny \begin{spacing}{0.7} reachable isoclasses of basic cluster-tilting objects in $\cc_{(Q,W)}$\end{spacing}}}}="10",
(150,1700) *+{\framebox(60,20){\parbox{57pt}{\tiny $\gg$-matrices of $Q$}}}="11",
(500,1500) *+{\framebox(60,20){\parbox{55pt}{\tiny {clusters in $\ca_Q$}}}}="12",
(900,1700) *+{\framebox(60,20){\parbox{57pt}{\tiny $\ccc$-matrices of $Q$}}}="13",
(600,750) *+{}="14",
(800,1500) *+{\framebox(45,20){\parbox{40pt}{\tiny {mutation class of $\hat{Q}$}}}}="15",
"0", {\ar "1"}, {\ar "2"}, {\ar "3"}, {\ar "4"}, {\ar@/_109pt/ "10"}, {\ar@/_77pt/ "11"}, 
"1", {\ar "0"},  {\ar@{-->} "5"}, 
"2",  {\ar "3"}, 
"3", {\ar "1"}, {\ar "2"}, {\ar@/^53pt/ "13"},
"4", {\ar@{-->} "5"}, {\ar "6"}, {\ar@{-->} "7"}, {\ar@/_20pt/ "8"},
"5", {\ar@{-->} "4"}, 
"6", {\ar "7"}, {\ar "2"}, 
"7", {\ar@{-->} "5"}, {\ar "6"}, {\ar "3"},
"8", {\ar "9"}, 
"9", {\ar "6"}, {\ar@{-} "14"}, "14", {\ar@{-->} "2"},
"10", {\ar "8"}, {\ar "12"}, {\ar "11"},
"11", {\ar "13"},
"12", {\ar "11"}, {\ar "15"},
"13", {\ar "11"},
"15", {\ar "13"},
 \end{xy}\]
\caption{The diagram}
\label{f:the-map}
\end{figure}
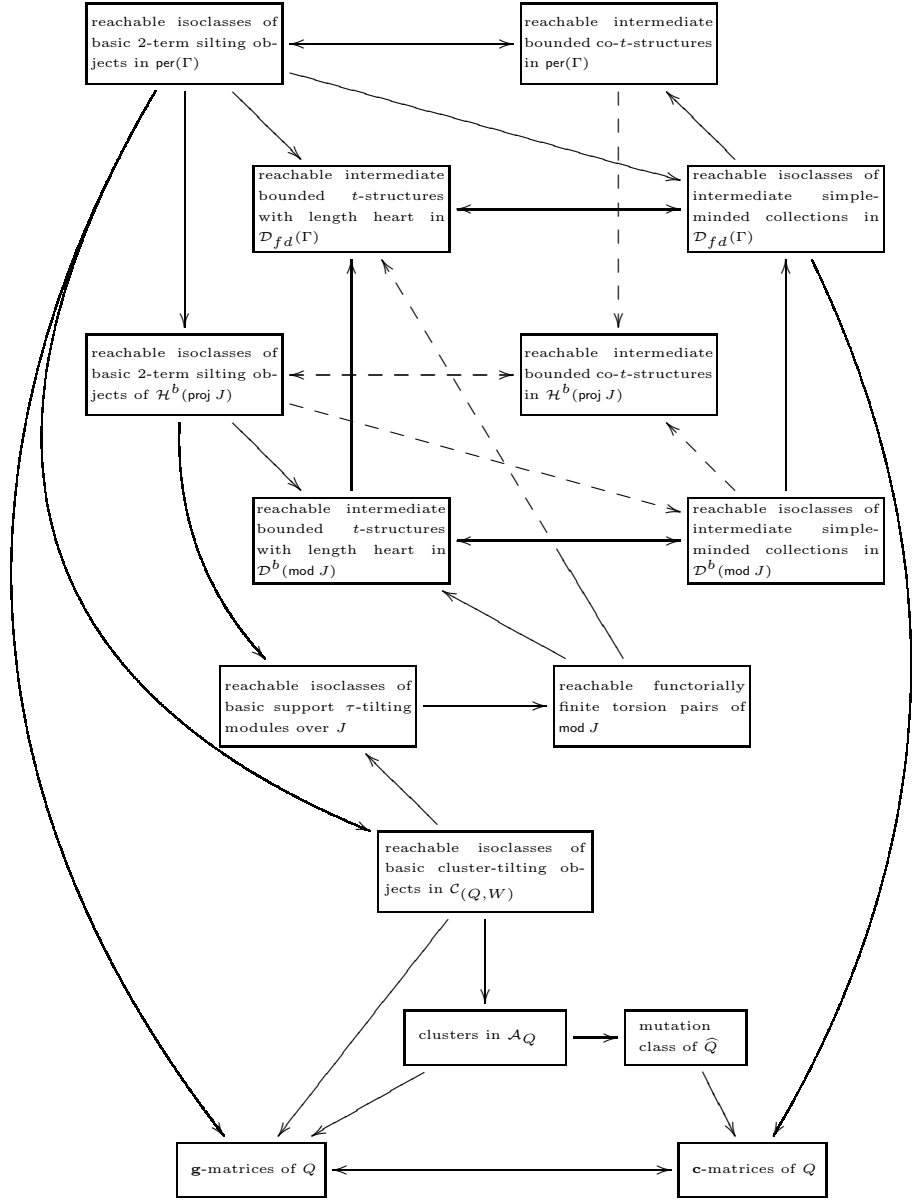

The exchange graph of a quiver $Q$ without loops or oriented 2-cycles encodes the combinatorics of mutations of clusters in the cluster algebra defined by $Q$.
With the interpretation of cluster algebras using representation theory of algebras, a whole variety of ordered exchange graphs has been associated with objects such as a finite-dimensional algebra $J$ or a differential graded (=dg) algebra $\Gamma$ concentrated in non-positive degrees. These constructions often come from variations of the concept of tilting, the vertices of the exchange graph being  torsion pairs, $t$-structures, silting objects, support $\tau$-tilting modules and so on. 
The goal of this article is to present these various constructions of ordered exchange graphs, relate them among each other and also to the combinatorial versions stemming from cluster theory. 
The above "atlas"  in Figure~\ref{f:the-map} shows the main examples and their interdependence.  All the maps in this diagram are bijections. They preserve partial orders and commute with mutations, and hence induce isomorphisms of ordered exchange graphs. 
We use the following notations which are explained in detail in Sections \ref{s:preliminaries}--\ref{s:the-bijection}:
\begin{itemize}
\item $(Q,W)$: a quiver with potential such that $Q$ has no loops or oriented 2-cycles, $W$ is non-degenerate and the algebra $J$ below is finite-dimensional, (Section~\ref{ss:qp})
\item $\Gamma=\hat{\Gamma}(Q,W)$: the (complete) Ginzburg dg algebra of $(Q,W)$, (Section~\ref{ss:ginzburg-algebra})
\item $J=\hat{J}(Q,W)$: the (complete) Jacobian algebra of $(Q,W)$, (Section~\ref{ss:ginzburg-algebra})
\item $\per(\Gamma)$: the perfect derived category of $\Gamma$,
\item $\cd_{fd}(\Gamma)$: the finite-dimensional derived category of $\Gamma$,
\item $\cc_{(Q,W)}$: the Amiot cluster category of $(Q,W)$, (Section~\ref{ss:amiot-cluster-category})
\item $\ch^b(\proj J)$: the bounded homotopy category of finitely generated projective $J$-modules,
\item $\cd^b(\mod J)$: the bounded derived category of finite-dimensional $J$-modules,
\item $\ca_Q$: the cluster algebra with principal coefficients defined by $Q$, (Section~\ref{ss:cluster})
\item $\hat{Q}$: the framed quiver associated to $Q$ (Section~\ref{ss:mut}).
\end{itemize}

The diagram naturally falls into five layers, which, from the top to the bottom, respectively involve derived categories of 3-Calabi--Yau dg algebras, derived categories of finite-dimensional algebras, module categories of finite-dimensional algebras, 2-Calabi--Yau triangulated categories and cluster combinatorics. The first two layers should be viewed as forming a cube (the "dashing" has no mathematical meaning).

This article is written in handbook style. The objects  and maps appearing in Figure~\ref{f:the-map} are defined in Section \ref{s:the-graph} and Section \ref{s:the-bijection}, respectively. Most notions and results presented here are not new, however we decide to give proofs of some results which are only known to the experts and difficult to find in the literature. The literature on cluster theory is quickly expanding, and we do not aim to cover everything in full generality. The reader is encouraged to read the original references. More detailed surveys on some of the objects and maps include \cite{Zelevinsky07,Keller10c,Fomin10,Leclerc10,Reiten10,Keller:derivedcluster}. In the appendix we provide some results on the derived category of a non-positive dg algebra which are used in Sections \ref{s:the-graph} and \ref{s:the-bijection}.

Throughout this article, $k$ denotes an uncountable algebraically closed field and $D=\Hom_k(?,k)$ denotes the $k$-dual. All categories and algebras are assumed to be over $k$. The condition that $k$ is uncountable is needed to guarantee the existence of a non-degenerate potential (see Section~\ref{ss:qp}) and is not necessary for other results. The suspension functor of a triangulated category is denoted by $\Sigma$.

\medskip
\noindent{\it Acknowledgement.} The authors would like to thank Christof Geiss, Osamu Iyama, Bernhard Keller, Yuya Mizuno and Pierre-Guy Plamondon for answering their questions and for very helpful comments on preliminary versions. The second-named author is indebted to Bernhard Keller for sharing his insight on cluster theory.

\section{Preliminaries}\label{s:preliminaries}

In this section, we give the definition of an ordered exchange graph and introduce some basic notions on general categories, derived categories, quivers with potential and their Ginzburg dg algebras as well as Jacobian algebras.

\subsection{Ordered exchange graphs}
A central notion in cluster theory is that of an ordered exchange graph, which we formalize for the purpose of this article as follows:

\subsubsection{Exchange graphs}\label{ss:exchange-graph} Consider a set $V$ with a \emph{compatibility relation} $R$, that is, $R$ is reflexive and symmetric. We say that two elements $x$ and $y$ of $V$ are \emph{compatible} if $(x,y)\in R$.
Assume the following conditions:
\begin{itemize}
\item[(1)] All maximal subsets of pairwise compatible elements, the \emph{clusters}, are finite and have the same cardinality, say $n$;
\item[(2)] Any subset of $n-1$ pairwise compatible elements is contained in precisely two clusters. 
\end{itemize} 
We then define an \emph{exchange graph} to be the graph whose vertices are the clusters and where two clusters are joined by an edge precisely when their intersection has cardinality $n-1$. We refer to the edges of an exchange graph as \emph{mutations}, and we use the same terminology for any graph which is isomorphic to an exchange graph. Note that all those graphs are $n-$regular.

The conditions on the compatibility relation $R$ can be rephrased as follows: consider the (abstract) simplicial complex $\Delta$ whose $l$-simplices are the subsets of $l+1$ pairwise compatible elements of $V$. A simplex of codimension $1$ is called a wall. We assume that
\begin{itemize}
\item[(1)] $\Delta$ is a pure simplicial complex, \ie all maximal simplices are of the same dimension;
\item[(2)] every wall is contained in precisely two maximal simplices.
\end{itemize}
Then the exchange graph is the dual graph of $\Delta$. If in addition the exchange graph is connected, then $\Delta$ is a pseudo-manifold. See~\cite[Section 2.1]{cluster2}.


\subsubsection{Ordered exchange graphs}\label{ss:ordered-exchange-graph} We define an \emph{ordered exchange graph} to be an exchange graph $C$ endowed with a partial order $\le $ on the set of clusters such that:
\begin{itemize}
\item[(i)] (the underlying graph of) the Hasse quiver of the partial order $\le$ coincides with the graph $C$:  a predecessor $x$ of $x'$ with respect to $\le$ is an immediate predecessor precisely when $x'$ and $x$ are related by a mutation,
\item[(ii)] the Hasse quiver has a unique source and at most one sink.
\end{itemize}

The orientation on an ordered exchange graph $C$ induces a colouring of the elements of each cluster $x$:  Since $C$ is an $n-$regular graph, there will be $g$ arrows starting in $x$ and $r=n-g$ arrows ending in $x$. Since edges correspond to mutations of the elements of $x$, we can write the set $x$ as a union of $g $ "green" elements and $r$ "red" elements, where the condition (2) in the definition of an exchange graph ensures that these two sets are disjoint. The source in $C$ is the unique cluster with all of its elements green, and the sink (if it exists) has all elements red.

A maximal path in an ordered exchange graph $C$ can thus be interpreted as a maximal sequence of mutations at green elements. These sequences, called \emph{maximal green sequences} by B. Keller~\cite{Keller11b}, play an important role in finding quantum dilogarithm identities and non-commutative Donaldson--Thomas invariants \cite{Keller11b} and in calculating the complete spectrum of a BPS (Bogomol'nyi--Prasad--Sommerfield) particle in string theory \cite{ACCERV:BPS,CCV,Xie} (this also appears  implicitly in \cite{GMN:WKB}).
\bigskip
 
The main feature of an ordered exchange graph  is the property (i): the arrows in the Hasse quiver of the poset are given by mutations. This property has been established by Happel and Unger  \cite{HappelUnger05} in the context of tilting modules (see Section~\ref{ss:tilting-module}), and the concepts presented in this article can be seen as a way to enlarge the poset of tilting modules such that the Hasse quiver becomes an $n-$regular graph (see Section~\ref{ss:support-tau-tilting}).

Note that not all the structures of an ordered exchange graph (the compatibility relation, the partial order, the mutations) are explicitly visible or known for some of the examples. The purpose of this article is to point out that all examples we discuss carry all those structures, via the bijections we provide. Some structures are more naturally defined for some cases: The compatibility relation is given by mutual vanishing of extension groups in the context of (cluster-)tilting objects, the partial order is given most naturally by inclusion in the context of torsion classes, $t$-structures and co-$t$-structures, and the mutation is the main ingredient in the definition of $\ccc$-matrices, $\gg$-matrices and clusters. However, the partial order on $\ccc$-matrices seems not to be known explicitly.

\subsection{Some notions on categories}

Let $\cc$ be a $k$-linear category. We denote by $\Hom_\cc(M,N)$ or simply $\Hom(M,N)$ the morphism space from $M$ to $N$ in $\cc$. For a subcategory or a set of objects $\cs$ of $\cc$, denote by ${}^\perp\cs$ (respectively, $\cs^\perp$) the \emph{left (respectively, right) orthogonal category} of $\cs$, \ie
\[{}^\perp\cs:=\{M\in \cc\mid \Hom_{\cc}(M,N)=0 \text{ for any } N\in\cs\}\]
(respectively, 
\[\cs^\perp:=\{M\in \cc\mid \Hom_{\cc}(N,M)=0 \text{ for any } N\in\cs\}\text{)}.\]
The \emph{split Grothendieck group} of $\cc$, denoted by $K^{\rm split}_0(\cc)$, is defined as the quotient of the free abelian group generated by isomorphism classes of objects of $\cc$ by the subgroup generated by elements of the form $[M]-[N]-[L]$, where $M\cong N\oplus L$.
The \emph{Grothendieck group} of an abelian (respectively, triangulated) category $\cc$, denoted by $K_0(\cc)$, is the quotient of the free abelian group generated by isomorphism classes of objects of $\cc$ by the subgroup generated by elements of the form $[M]-[N]-[L]$, whenever there is an exact sequence $0\rightarrow N\rightarrow M\rightarrow L\rightarrow 0$ (respectively, a triangle $N\rightarrow M\rightarrow L\rightarrow\Sigma N$).

Assume that $\cc$ is Krull--Schmidt. For an object $M$ in $\cc$, we denote by $|M|$ the number of pairwise non-isomorphic indecomposable direct summands of $M$, and  by $\add(M)=\add_{\cc}(M)$ the smallest full subcategory of $\cc$ which contains $M$ and which is closed under taking finite direct sums and direct summands. An object $M$ of $\cc$ is said to be \emph{basic} if each indecomposable direct summand of $M$ occurs with multiplicity $1$ in a decomposition of $M$ into the direct sum of indecomposable objects.

Let $\ca$ be an abelian category. A full subcategory $\cb$ of $\ca$ is \emph{functorially finite} for any object $X$ of $\ca$ there are objects $L_X$ and $R_X$ of $\cb$ together with morphisms $X\rightarrow L_X$ and $R_X\rightarrow X$ such that the induced morphisms of functors 
\[\Hom_{\cb}(L_X,?)\rightarrow \Hom_{\ca}(X,?)|_\cb~~\text{ and } \Hom_{\cb}(?,R_X)\rightarrow\Hom_{\ca}(?,X)|_\cb\]
are surjective.
The abelian category $\ca$  is called a \emph{length category} if every object has finite length, \ie every object admits a finite filtration such that all the subfactors are simple.

For a subcategory or a set of objects $\cs$ of  a triangulated category $\cc$, we denote by $\thick(\cs)$ the \emph{thick subcategory of $\cc$ generated by $\cs$}, \ie the smallest triangulated subcategory of $\cc$ which contains $\cs$ and which is closed under taking isomorphisms and direct summands. We say that $\cs$ \emph{generates} $\cc$ if $\cc=\thick(\cs)$ holds. For $d\in\mathbb{Z}$, the triangulated category $\cc$ is \emph{$d$-Calabi--Yau} if there is a bifunctorial isomorphism
\[D\Hom(M,N)\stackrel{\sim}{\longrightarrow} \Hom(N,\Sigma^d M)\]
for any $M$ and $N$ in $\cc$.

\subsection{Derived categories}
For an algebra $A$, we will denote by $\Mod A$ the category of (right) $A$-modules, by $\mod A$ the category of finite-dimensional $A$-modules and by $\proj A$ the category of finitely generated projective $A$-modules. Denote by $\ch^b(\proj A)$ the homotopy category of bounded complexes of $\proj A$, by $\cd^b(\mod A)$ the derived category of bounded complexes of $\mod A$ and by $\cd(\Mod A)$ the derived category of complexes of $\Mod A$.

Let $A$ be a dg algebra, \ie a graded algebra endowed with a differential $d$ such that $(A,d)$ is a complex of vector spaces and the following graded Leibniz rule holds
for all homogeneous elements $a$ of degree $p$ and all elements $b$:
\[
d(ab)=d(a)b+(-1)^pad(b).
\]
Consider the derived category $\cd(A)$ of (right) dg $A$-modules, see~\cite{Keller94,Keller06d}. This is a triangulated category. For  a dg $A$-module $M$, we have
\[
\Hom_{\cd(A)}(A,\Sigma^m M)=H^m(M).
\]
This formula will be used without further reference.

We are interested in the following two triangulated subcategories of $\cd(A)$: 

(1) the \emph{perfect derived category} $\per(A)=\thick(A)$, the thick subcategory of $\cd(A)$ generated by $A_A$, the free dg $A$-module of rank $1$; 

(2) the \emph{finite-dimensional derived category} $\cd_{fd}(A)$, which consists of those dg $A$-modules whose total cohomology is finite-dimensional over $k$.

If $A$ is a finite-dimensional algebra, we can view it as a dg algebra concentrated in degree $0$. In this case, we have $\cd(A)=\cd(\Mod A)$, $\cd_{fd}(A)\cong\cd^b(\mod A)$ and $\per(A)\cong \ch^b(\proj A)$.

A dg algebra $A$ is said to be \emph{non-positive} if its degree $i$ component vanishes for all $i>0$.

\subsection{Quivers with potential}\label{ss:qp}

For a finite quiver $Q$, we denote its set of vertices by $Q_0$ and its set of
arrows by $Q_1$. For an arrow $\alpha$, we denote by $s(\alpha)$ its source and by $t(\alpha)$ its target. The trivial path corresponding to a vertex $i$ will
be denoted by $e_i$. The path algebra of $kQ$ is by definition the vector space with basis the set of paths of $Q$ and with multiplication given by concatenation of paths, \ie for two paths $p$ and $q$, we have
\[p\cdot q=\begin{cases} pq & \text{ if } s(p)=t(q)\\
0 &\text{ otherwise.}\end{cases}\]

If $Q$ is a graded quiver, \ie there is an integer associated with each arrow of $Q$, then $kQ$ is naturally a graded algebra. In this case, the {\em complete path algebra $\widehat{kQ}$} is the completion of
the path algebra $kQ$ in the category
of graded vector spaces with respect to the ideal generated by the
arrows of $Q$. Thus, the $n$-th component of
$\widehat{kQ}$ consists of formal combinations
$\sum_{p}\lambda_p p$ of all paths $p$ of degree $n$. Below we will reserve the terminology \emph{quiver} for ungraded quivers and we will sometimes consider them as graded quivers concentrated in degree $0$.

\smallskip

Let $Q$ be a finite quiver.  A {\em potential} on $Q$ is an element of the closure
of the subspace of the complete path algebra $\widehat{kQ}$ generated by all non-trivial cycles of $Q$, \ie an element of the form $\sum_c \lambda_c c$, where $c$ runs over all non-trivial cycles of $Q$. For a potential $W$ on $Q$, the pair $(Q,W)$ is called a \emph{quiver with potential}. For an
arrow $\rho$ and a cycle $c$ of $Q$, we define
$\del_\rho(c)=\sum_{c=u \rho v} vu $, where the sum is taken over all decompositions of
the cycle $c$ (where $u$ and $v$ are possibly trivial paths).
Writing $W=\sum_{c:cycle}\lambda_c c$, we define
$\del_\rho(W)=\sum_{c:cycle}\lambda_c \del_\rho(c)$.

Thanks to the work of Derksen, Weyman and Zelevinsky~\cite{DerksenWeymanZelevinsky08}, the quiver mutation (Section~\ref{ss:mut}) extends to a mutation of quivers with potential. Given a quiver with potential $(Q,W)$ such that $Q$ has no loops or oriented 2-cycles and a vertex $i$ of $Q$ , the mutation at $i$ produces a new quiver with potential $\mu_i(Q,W)$. Unlike the quiver case, however, the new quiver with potential $\mu_i(Q,W)$ may contain oriented 2-cycles. A potential $W$ is said to be \emph{non-degenerate} if for any sequence $(i_1,\ldots,i_k)$ of vertices of $Q$ the quiver of the multi-mutated quiver with potential $\mu_{i_k}\cdots\mu_{i_1}(Q,W)$ does not contain oriented 2-cycles. Recall that the base field $k$ is uncountable. It follows from \cite[Corollary 7.4]{DerksenWeymanZelevinsky08} that non-degenerate potentials exist for any quiver without loops or oriented 2-cycles.

\subsection{The Ginzburg dg algebra and the Jacobian algebra}\label{ss:ginzburg-algebra}
Let $(Q,W)$ be a quiver with potential.  The {\em (complete) Ginzburg dg
algebra $\hat{\Gamma}(Q,W)$} of $(Q,W)$ is constructed as follows \cite{Ginzburg06}: Let $\tilde{Q}$
be the graded quiver with the same vertices as $Q$ and whose arrows
are
\begin{itemize}
\item the arrows of $Q$ (they all have degree~$0$),
\item an arrow $\rho^*: j \to i$ of degree $-1$ for each arrow $\rho:i\to j$ of $Q$,
\item a loop $t_i : i \to i$ of degree $-2$ for each vertex $i$
of $Q$.
\end{itemize}
The underlying graded algebra of $\hat{\Gamma}(Q,W)$ is the
complete path algebra $\widehat{k\tilde{Q}}$. It is endowed with the $\mathfrak{m}$-adic topology, where $\mathfrak{m}$ is the ideal of $\widehat{k\tilde{Q}}$ generated by all arrows of $\tilde{Q}$.
The differential of $\hat{\Gamma}(Q,W)$ is the unique continuous
linear endomorphism homogeneous of degree~$1$ which satisfies the
graded Leibniz rule
\[
d(uv)= (du) v + (-1)^p u dv \ko
\]
for all homogeneous $u$ of degree $p$ and all $v$, and takes the
following values on the arrows of $\tilde{Q}$:
\begin{itemize}
\item $d(\rho)=0$ for each arrow $\rho$ of $Q$,
\item $d(\rho^*) = \del_\rho W$ for each arrow $\rho$ of $Q$,
\item $d(t_i) = e_i \sum_{\rho\in Q_1} (\rho\rho^*-\rho^*\rho) e_i$ for each vertex $i$ of $Q$.
\end{itemize}

\emph{The (complete) Jacobian algebra} $\hat{J}(Q,W)$ of the quiver with
potential $(Q,W)$ is by definition the $0$-th cohomology of the
Ginzburg dg algebra $\hat\Gamma(Q,W)$. Concretely we have
\[\hat{J}(Q,W)=\widehat{kQ}/\overline{\langle \del_\rho W,\rho\in Q_1\rangle},\]
where for an ideal $I$ of $\widehat{kQ}$ we denote by $\bar{I}$ its closure. The quiver with potential $(Q,W)$ is \emph{Jacobi-finite} if the Jacobian algebra $\hat{J}(Q,W)$
is finite-dimensional.

\begin{remark} Let $(Q,W)$ be a quiver with potential and $\Gamma=\hat{\Gamma}(Q,W)$ be its Ginzburg dg algebra.
\begin{itemize}
\item[(a)] It follows from the definition that $\Gamma$ is non-positive.
\item[(b)] According to \cite[Theorem A.17]{KellerYang11}, the dg algebra $\Gamma$ is (topologically) homologically smooth and bimodule $3$-Calabi--Yau. It then follows from~\cite[Lemma 4.1]{Keller10} that the perfect derived category $\per(\Gamma)$ contains the finite-dimensional derived category $\cd_{fd}(\Gamma)$ as a triangulated subcategory; moreover, $\cd_{fd}(\Gamma)$ is $3$-Calabi--Yau as a triangulated category.
\end{itemize}
\end{remark}
\smallskip
\noindent
{\bf Example.} Let $Q$ be the quiver $\xymatrix{1\ar[r]^\alpha & 2}$ of type $\mathbb{A}_2$. Then the above construction yields
\[
\tilde{Q}=\xymatrix@C=3pc{1\ar@<.7ex>[r]^\alpha \ar@(dl,ul)^{t_1}& 2\ar@<.7ex>[l]^{\alpha^*}\ar@(dr,ur)_{t_2}}
\]
There are no non-trivial cycles in $Q$, so the only potential on $Q$ is $0$. The Ginzburg dg algebra $\hat{\Gamma}(Q,0)$ is the complete path algebra $\widehat{k\tilde{Q}}$ endowed with the differential which takes $t_1$ to $-\alpha^*\alpha$, $t_2$ to $\alpha\alpha^*$ and takes $\alpha$ and $\alpha^*$ to $0$. The Jacobian algebra $\hat{J}(Q,0)$ is the path algebra $kQ$.


\section{The ordered exchange graphs}\label{s:the-graph}
We introduce in this section the following sets, thus introducing the various ordered exchange graphs presented in Figure \ref{f:the-map}:
\begin{align*}
&\fftor(A), \sttilt(A), \intertstr(A), \intercotstr(A), \twosilt(A), &\\
&\intersmc(A),\cto(\cc), \mut(Q), \Cl(Q),\cmat(Q), \gmat(Q).\hspace{9pt}& 
\end{align*}
We warn the reader that we will use the notation $\mu^-$ for left mutation with $-$ meaning a decrease with respect to the partial order. This convention is the same as the one used in \cite{Keller:derivedcluster} and opposite to the one used in \cite{AiharaIyama12,KoenigYang12}. We remark that for most of these sets the structure of an ordered exchange graph are defined only after having the bijections in Section~\ref{s:the-bijection}.


\subsection{Tilting modules, $\tilt(A)$}\label{ss:tilting-module}
In this subsection, we recall some results of Happel--Unger and of Riedtmann--Schofield on the "ordered exchange graph" of tilting modules.
 Already in 1987 Ringel observed that the set of tilting modules over a finite-dimensional algebra A carries the structure of a simplicial complex. The study of this complex and its poset structure was initiated in \cite{RiedtmannSchofield91} and further carried out by Happel and Unger \cite{Unger96a,Unger96b,HappelUnger05,HappelUnger1,HappelUnger2}. See also the contributions of Ringel and of Unger in the Handbook of tilting theory \cite{Ringel07,Unger07}.

Let $A$ be a finite-dimensional algebra. An $A$-module $T$ is a \emph{pretilting module} if $\Ext_A^1(T,T)=0$ and the projective dimension of $T$ is at most $1$. It is a \emph{tilting module} if in addition there is a short exact sequence of $A$-modules
\[
\xymatrix{0\ar[r] & A\ar[r] & T^0\ar[r] & T^1\ar[r] & 0}
\]
with $T^0$ and $T^1$ in $\add(T)$.

The following theorem of Bongartz shows that any pretilting module can be completed to a tilting module.

\begin{theorem}{(\cite[Lemma 2.1 and Theorem 2.1]{Bongartz81})}\label{t:bongartz-complement}
Let $T$ be a pretilting $A$-module. Then there is an $A$-module $X$ such that $T\oplus X$ is a tilting module. Consequently, $T$ is a tilting module if and only if $|T|=|A|$.
\end{theorem}

In general $X$ is not basic and it may contain summands from $\add(T)$. But if $T$ is basic, then we can choose a suitable summand $Y$ of $X$ (called the \emph{Bongartz complement} of $T$) such that $T\oplus Y$ becomes a basic tilting module. The next result enables us to define mutation for tilting modules.

\begin{theorem}{(\cite[Theorem 1.1]{HappelUnger89} and \cite[Proposition 1.3]{RiedtmannSchofield91})}\label{t:almost-complete-tilting-has-at-most-two-complements}
Let $T=T'\oplus X$ be a basic tilting $A$-module with $X$ indecomposable. Then there exists at most one indecomposable $A$-module $Y$ such that $Y\not\cong X$ and $T'\oplus Y$ is a tilting module. If such a $Y$ exists, then there is a short exact sequence either of the form
\[
\xymatrix{0\ar[r] & X\ar[r] & E\ar[r] & Y\ar[r] & 0}
\]
with $E$ in $\add(T')$ or of the form
\[
\xymatrix{0\ar[r] & Y\ar[r] & E'\ar[r] & X\ar[r] & 0}
\] 
with $E'$ in $\add(T')$.
\end{theorem}

For a finite-dimensional algebra $A$, put 
\begin{align*}
\tilt(A) &= \text{the set of isomorphism classes of basic tilting $A$-modules}.
\end{align*}
When there is no confusion, we will identify a module $T$ with its isomorphism class $[T]$. Let $T=T'\oplus X$ be an element of $\tilt(A)$ with $X$ indecomposable. The \emph{left mutation} $\mu^-_X(T)$ of $T$ at $X$ is defined as $T'\oplus Y$, provided that such a $Y$ as in Theorem~\ref{t:almost-complete-tilting-has-at-most-two-complements} exists and that there is a short exact sequence
\[
\xymatrix{0\ar[r] & X\ar[r] & E\ar[r] & Y\ar[r] & 0}
\]
with $E$ in $\add(T')$. In this case, $X$ is the Bongartz complement of $T'$. The quiver whose vertices are elements of $\tilt(A)$ and whose arrows are given by left mutations will be denoted by  $\mathcal{E}_A$. The underlying graph of $\mathcal{E}_A$ is in general not an exchange graph as in Section~\ref{ss:exchange-graph}, because it is not regular, see~\cite[Corollary 1.3]{HappelUnger89}, \cite[Remark 1.3]{RiedtmannSchofield91} and~\cite[Lemma 4.3]{HappelUnger2}. Thanks to the work~\cite{AdachiIyamaReiten12} of Adachi, Iyama and Reiten, it can be completed into an ordered exchange graph by considering support $\tau$-tilting modules (see Section~\ref{ss:support-tau-tilting}).

For an $A$-module $T$, define
\[T^{\perp_1}:=\{X\in\mod A\mid \Ext_A^1(T,X)=0\}.\]
On $\tilt(A)$ define a partial ordering by
\[T\leq T':\Leftrightarrow T^{\perp_1}\subseteq T'^{\perp_1},\]
see~\cite[Remark 2.2 (b)]{RiedtmannSchofield91}. The following was posed as a question by Riedtmann and Schofield and proved by Happel and Unger.

\begin{theorem}{(\cite[Theorem 2.1]{HappelUnger05})}
The Hasse quiver of $(\tilt(A),\leq)$ coincides with $\mathcal{E}_A$.
\end{theorem}

The quiver $\mathcal{E}_A$ is in general not connected, see~\cite[Sections 6 and 7]{HappelUnger2}. However, we have

\begin{proposition}{(\cite[Corollary 2.2]{HappelUnger05})}\label{p:finiteness-and-connectedness}
If $\mathcal{E}_A$ has a finite connected component, then it is connected.
\end{proposition}

\noindent
{\bf Example.} Let $A$ be the path algebra of the quiver $\xymatrix{1\ar[r] & 2}$ of type $\mathbb{A}_2$ and $\mod A$ be the category of finite-dimensional right $A$-modules. Up to isomorphism, there are precisely three indecomposable objects in $\mod A$: $S_1$, $P_2$ and $S_2$, which respectively correspond to the following representations
\[\xymatrix{k & 0\ar[l]}, \xymatrix{k & k\ar[l]_{id}}, \xymatrix{0 & k\ar[l]}.\]
The Auslander--Reiten quiver of $\mod A$ is
\[\xymatrix@C=1pc@R=1pc{&P_2\ar[dr]&\\ S_1\ar[ur] && S_2}\]
There are precisely two isomorphism classes of basic tilting modules: $A=P_1\oplus P_2$ and $P_2\oplus S_2$. The graph $\ce_A$ is
\[
\xymatrix{A\ar[r] & P_2\oplus S_2}.
\]

\subsection{Torsion pairs, $\fftor(A)$}\label{ss:torsion-pair}

Let $\ca$ be an abelian category. A \emph{torsion pair} $(\ct,\cf)$ of $\ca$ is a pair of full subcategories such that
\begin{itemize}
\item[$\cdot$] $\Hom_{\ca}(M,N)=0$ for all $M\in\ct$ and $N\in\cf$,
\item[$\cdot$] for any $M\in\ca$ there is a short exact sequence
\begin{align}
\xymatrix{0\ar[r] & M'\ar[r] & M\ar[r] & M''\ar[r] & 0}\label{eq:torsion-pair}
\end{align}
with $M'\in\ct$ and $M''\in \cf$.
\end{itemize}
This notion was introduced by Dickson in~\cite{Dickson66}. The subcategories $\ct$ and $\cf$ are respectively called the \emph{torsion class} and the \emph{torsion-free class}. For example, let $\ca$ be the category of fintely generated abelian groups; then the torsion groups form a torsion class, and the torsion-free groups form the corresponding torsion-free class.
 Note that a simple object $M$ in $\ca$ lies either in the torsion class or in the torsion-free class, because it does not admit any non-trivial short exact sequence of the form (\ref{eq:torsion-pair}).
 One observes that $\cf=\ct^\perp$ and $\ct={}^\perp\cf$. Thus a torsion pair is determined by its torsion class (respectively, its torsion-free class). 
 
Let $A$ be a finite-dimensional algebra. A torsion pair of $\mod A$ is said to be \emph{functorially finite} if its torsion class is functorially finite, or equivalently, its torsion-free class is functorially finite (see~\cite[Proposition 1.1]{AdachiIyamaReiten12}). When $A$ is representation-finite, all torsion pairs of $\mod A$ are functorially finite. This is not true for representation-infinite algebras. For example, for the path algebra $A$ of the Kronecker quiver $\xymatrix{\cdot\ar@<.3ex>[r]\ar@<-.3ex>[r]&\cdot}$, the torsion class formed by the preinjective modules is not functorially finite.
Set
\begin{align*}
\fftor(A) &= \text{the set of functorially finite torsion pairs of $\mod A$}.
\end{align*}
On this set, there is a natural partial order
\[(\ct,\cf)\leq (\ct',\cf'):\Leftrightarrow \ct\subseteq \ct', \text{ equivalently},~~ \cf\supseteq\cf'.\]
Clearly, with respect to this order $(\mod A,0)$ is the unique maximal element and $(0,\mod A)$ is the unique minimal element.

\medskip
\noindent{\bf Example.} Let $A$ be the hereditary algebra of type $\mathbb{A}_2$ as in Section~\ref{ss:tilting-module}. The category $\mod A$ has precisely five torsion classes, all of which are functorially finite: $\add(S_1\oplus P_2\oplus S_2)=\mod A$, $\add(P_2\oplus S_2)$, $\add(S_2)$, $\add(S_1)$ and $0$. Thus the Hasse quiver of $\fftor(A)$ is (we depict the torsion pairs in the Auslander--Reiten quiver by marking in black the indecomposable objects in the torsion classes):
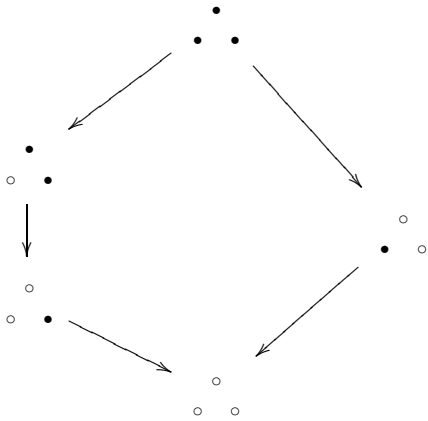
\[\begin{xy} 0;<0.35pt,0pt>:<0pt,-0.35pt>::
(300,0) *+{\begin{array}{*{3}{c@{\hspace{3pt}}}}
&\sbullet&\\[-4pt] \sbullet & &\sbullet
\end{array}} ="0",
(100,150) *+{\begin{array}{*{3}{c@{\hspace{3pt}}}}
&\sbullet&\\[-4pt] \scirc & &\sbullet
\end{array}}="1",
(100,300) *+{\begin{array}{*{3}{c@{\hspace{3pt}}}}
&\scirc&\\[-4pt] \scirc & &\sbullet
\end{array}}="2",
(500,225) *+{\begin{array}{*{3}{c@{\hspace{3pt}}}}
&\scirc&\\[-4pt] \sbullet & &\scirc
\end{array}}="3",
(300,400) *+{\begin{array}{*{3}{c@{\hspace{3pt}}}}
&\scirc&\\[-4pt] \scirc & &\scirc
\end{array}}="4",
"0", {\ar "1"}, "0", {\ar "3"},
"1", {\ar "2"},
"2", {\ar "4"},
"3", {\ar "4"},  
\end{xy}\]

\subsection{Support $\tau$-tilting modules, $\sttilt(A)$}\label{ss:support-tau-tilting}

Let $A$ be a finite-dimensional algebra. A finite-dimensional $A$-module $T$ is a \emph{$\tau$-rigid module} if $\Hom_A(T,\tau T)=0$ and a \emph{$\tau$-tilting module} if
in addition $|T|=|A|$ holds. Here $\tau$ denotes the Auslander--Reiten translation of $\mod A$ (see~\cite{AssemSimsonSkowronski06}). An $A$-module $T$ is a \emph{support $\tau$-tilting module} if there is an idempotent $e$ of $A$ such that $T$ is a $\tau$-tilting module over $A/A eA$. Note that $T$ is rigid over $A$ and the idempotent $e$ is unique in the sense that if $e'$ is another such idempotent then $\add(eA)=\add(e'A)$ holds. 
These notions were already studied by Auslander and Smal{\o}~\cite{AuslanderSmalo81} in the 1980's, but we are using here the terminology adopted by Adachi, Iyama and Reiten in \cite{AdachiIyamaReiten12}, which generalises the work of Ingalls and Thomas~\cite{IngallsThomas09} on hereditary algebras.

The notion of $\tau$-tilting modules generalises that of tilting modules (Section~\ref{ss:tilting-module}). 

\begin{theorem}{(\cite{AdachiIyamaReiten12})}\label{t:tilting-and-support-tau-tilting}
Let $T$ be an $A$-module.
\begin{itemize}
\item[(a)] $T$ is a tilting module if and only if $T$ is a faithful $\tau$-tilting module.
\item[(b)] If $A$ is hereditary, then $T$ is a partial tilting module if and only if $T$ is a $\tau$-rigid module and $T$ is a tilting module if and only if $T$ is a $\tau$-tilting module. 
\end{itemize}
\end{theorem}

Adachi, Iyama and Reiten  introduce in~\cite{AdachiIyamaReiten12} the mutation of support $\tau$-tilting modules (for the case when $A$ is hereditary, see also~\cite{Hubery06}), which generalises the mutation of a tilting module as in Section~\ref{ss:tilting-module}. Let $T=T_1\oplus\ldots\oplus T_r$ be a basic support $\tau$-tilting $A$-module with each $T_i$ indecomposable. For $1\leq i\leq r$ such that $T_i\not\in \Fac(\bigoplus_{j\neq i}T_j)$ (the notation $\Fac$ is defined below), the \emph{left mutation of $T$ at $T_i$}, denoted $\mu_{T_i}^-(T)$, is defined as $T_i^*\oplus\bigoplus_{j\neq i}T_j$, where $T_i^*$ is the cokernel of a minimal left $\add(\bigoplus_{j\neq i}T_j)$-approximation of $T_i$
\[\xymatrix{T_i\ar[r]& T'}.\]

For a basic finite-dimensional algebra $A$, we set
\begin{align*}
\sttilt(A) &= \text{the set of isomorphism classes of basic support $\tau$-tilting $A$-modules}.
\end{align*}
On this set there is a partial order
\[T\leq T':\Leftrightarrow \Fac(T)\subseteq \Fac(T'),\]
see \cite[Section 2.4]{AdachiIyamaReiten12}. Here  $\Fac(T)$ is the full subcategory of $\mod A$ consisting of modules $M$ such that there is an epimorphism $T^{\oplus m}\rightarrow M$ for some $m\in\mathbb{N}$.
Clearly, with respect to this order, $A$ is the unique maximal element and $0$ is the unique minimal element. The following theorem was essentially already established in \cite{AdachiIyamaReiten12}.

\begin{theorem}\label{t:support-tau-tilting-is-an-ordered-exchange-graph} Let $A$ be a basic finite-dimensional algebra.
The set $\sttilt(A)$ has the structure of an ordered exchange graph with a sink.
\end{theorem}
\begin{proof} Put $n:=|A|$ and let $P_1,\ldots,P_n$ be a complete set of pairwise non-isomorphic indecomposable projective $A$-modules. 
Let $V$ be the union of $\{1,\ldots,n\}$ and the set of indecomposable $\tau$-rigid $A$-modules. Consider the following compatibility relation $R$ on $V$: 
\begin{itemize}
\item[$\cdot$] for two indecomposable $\tau$-rigid modules $M$ and $N$, the pair $(M,N)$ belongs to $R$ if and only if $\Hom_A(M, \tau N)=0=\Hom_A(N,\tau M)$;
\item[$\cdot$] for an indecomposable $\tau$-rigid module $M$ and $1\leq i\leq n$, the pair $(M,i)$ belongs to $R$ if and only if $\Hom_A(P_i,M)=0$ if and only if the pair $(i,M)$ belongs to $R$;
\item[$\cdot$] for any $1\leq i,j\leq n$, the pair $(i,j)$ belongs to $R$.
\end{itemize}

Now we are ready to prove the desired statement.

First, let $T=T_1\oplus\ldots \oplus T_r$ be a basic support $\tau$-tilting $A$-module with $T_i$ indecomposable for all $i$. Let $e$ be an idempotent of $A$ such that $T$ is a $\tau$-tilting module over $A/AeA$ and that $eA$ is basic. Then $(T,eA)$ is a basic support $\tau$-tilting pair in the sense of \cite[Definition 0.3]{AdachiIyamaReiten12}. It follows that $|eA|=|A|-|A/AeA|=n-r$. Write $eA=P_{i_1}\oplus\ldots\oplus P_{i_{n-r}}$. Then
the support $\tau$-tilting module $T$ can be identified with the $n$-element set $\{T_1,\ldots,T_r,i_1,\ldots,i_{n-r}\}$. So we have the condition (1) in the definition of an exchange graph.

Secondly, the condition (2) in the definition of an exchange graph is satisfied by \cite[Theorem 2.17]{AdachiIyamaReiten12}.

Thirdly, the condition (i) in the definition of an ordered exchange graph is satisfied by \cite[Corollary 2.31]{AdachiIyamaReiten12}.

Finally, let us show that the Hasse quiver has a unique source and a unique sink, so the condition (ii) in the definition of an ordered exchange graph is satisfied. Since $A$ is the unique maximal element of $\sttilt(A)$, it is a source in the Hasse quiver and for any support $\tau$-tilting $A$-module $T$ which is not isomorphic to $A$, we have $A> T$. It follows from \cite[Theorem 2.32 and Definition-Proposition 2.26]{AdachiIyamaReiten12} that  there exists a support $\tau$-tilting module $T'$ and an indecomposable direct summand $M$ of $T'$ such that $\mu_M^-(T')=T$. In particular, $T$ is not a source. Therefore, $A$ is the unique source of $\sttilt(A)$. Dually, one shows that $0$ is the unique sink.
\end{proof}

When $T$ is a tilting module, we have $T^{\perp_1}=\Fac(T)$. So by Theorem~\ref{t:tilting-and-support-tau-tilting} (a), the ordered exchange graph $\sttilt(A)$ completes the graph $\ce_A$ in Section~\ref{ss:tilting-module} to a regular graph. In general $\sttilt(A)$  has less connected components than $\ce_A$. For example, it is easy to check that for the path algebra $A$ of the Kronecker quiver, $\ce_A$ has two connected components while $\sttilt(A)$ is connected.
The following nice result, analogous to Proposition~\ref{p:finiteness-and-connectedness}, is a consequence of Theorem~\ref{t:support-tau-tilting-is-an-ordered-exchange-graph} because each component of a finite Hasse quiver necessarily has a source and a sink since it has no oriented cycles.

\begin{proposition}{(\cite[Corollary 2.35]{AdachiIyamaReiten12})}\label{p:finiteness-and-connectedness-for-support-tau-tilting}
Let $A$ be a basic finite-dimensional algebra. If $\sttilt(A)$  has a finite connected component, then $\sttilt(A) $  itself is connected.
\end{proposition}

Examples of such algebras include representation-finite algebras and preprojective algebras of Dynkin quivers (\cite{Mizuno13}).


\medskip
\noindent{\bf Example.} Let $A$ be the hereditary algebra of type ${\mathbb A}_2$ as in Section~\ref{ss:tilting-module}. There are precisely five isomorphism classes of basic support $\tau$-tilting $A$-modules: $S_1\oplus P_2=A$, $P_2\oplus S_2$, $S_2$, $S_1$ and $0$. The ordered exchange graph $\sttilt(A) $ of basic support $\tau$-tilting $A$-modules is
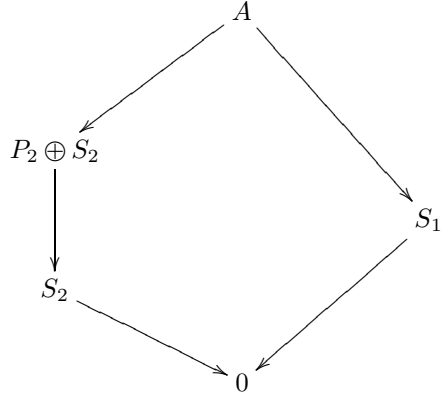
\[\begin{xy} 0;<0.35pt,0pt>:<0pt,-0.35pt>::
(300,0) *+{A} ="0",
(100,150) *+{P_2\oplus S_2}="1",
(100,300) *+{S_2}="2",
(500,225) *+{S_1}="3",
(300,400) *+{0}="4",
"0", {\ar "1"}, "0", {\ar "3"},
"1", {\ar "2"},
"2", {\ar "4"},
"3", {\ar "4"},  
\end{xy}\]

\subsection{$t$-structures, $\intertstr(A)$}\label{ss:t-str}

Let $\cc$ be a triangulated category.
A \emph{$t$-structure} on
$\cc$  is a pair $(\cc^{\leq
0},\cc^{\geq 0})$ of  strict (that is, closed under
isomorphisms) and full subcategories of $\cc$ such that
\begin{itemize}
\item[$\cdot$] $\Sigma\cc^{\leq 0}\subseteq\cc^{\leq 0}$ and
$\Sigma^{-1}\cc^{\geq 0}\subseteq\cc^{\geq 0}$;
\item[$\cdot$] $\Hom(M,\Sigma^{-1}N)=0$ for $M\in\cc^{\leq 0}$
and $N\in\cc^{\geq 0}$,
\item[$\cdot$] for each $M\in\cc$ there is a triangle in $\cc$ 
\begin{align}
\xymatrix{M'\ar[r] &
M\ar[r] & M''\ar[r] &\Sigma M'}\label{eq:triangle-t-str}
\end{align} 
with $M'\in\cc^{\leq
0}$ and $M''\in\Sigma^{-1}\cc^{\geq 0}$.
\end{itemize}
The above triangle (\ref{eq:triangle-t-str}) is canonical and yields endofunctors $\sigma^{\leq 0}$ and $\sigma^{\geq 1}$ of $\cc$ such that $\sigma^{\leq 0}M=M'$ and $\sigma^{\geq 1}M=M''$. For $i\in\mathbb{Z}$, let $\sigma^{\leq i}=\Sigma^{-i}\sigma^{\leq 0}\Sigma^i$ and $\sigma^{\geq i}=\Sigma^{-i+1}\sigma^{\geq 1}\Sigma^{i-1}$.

The notion of $t$-structures was introduced by Beilinson, Bernstein and Deligne in~\cite{BeilinsonBernsteinDeligne82} when studying  perverse sheaves over stratified topological spaces.
The two subcategories $\cc^{\leq 0}$ and $\cc^{\geq 0}$ are often
called the \emph{aisle} and the \emph{co-aisle} of the $t$-structure. 
One observes that $\cc^{\geq 0}=\Sigma (\cc^{\leq 0})^\perp$ and $\cc^{\leq 0}=\Sigma^{-1} {}^\perp\cc^{\geq 0}$. Thus a $t$-structure is determined by its aisle (respectively, by its co-aisle). 
The \emph{heart} $\ch = \cc^{\leq 0}~\cap~\cc^{\geq 0}$ of a $t$-structure is
always abelian. There is a family of cohomological functors $H^i=\sigma^{\leq i}\sigma^{\geq i}:\cc\rightarrow\cc^{\leq 0}\cap\cc^{\geq 0}$ (called the \emph{cohomology functors}). The
$t$-structure $(\cc^{\leq 0},\cc^{\geq 0})$ is said to be
\emph{bounded} if
\[\bigcup_{n\in\mathbb{Z}}\Sigma^n \cc^{\leq
0}=\cc=\bigcup_{n\in\mathbb{Z}}\Sigma^n\cc^{\geq 0}.\]
A bounded $t$-structure is one of the two ingredients of a Bridgeland stability condition~\cite{Bridgeland07}.
A typical
example of a $t$-structure is the pair $(\cd_\std^{\leq 0},\cd_\std^{\geq 0})$
for the bounded derived category $\cd^b(\mod A)$ of a finite-dimensional
algebra $A$, where $\cd_\std^{\leq 0}$ consists of complexes with
vanishing cohomologies in positive degrees, and $\cd_\std^{\geq 0}$
consists of complexes with vanishing cohomologies in negative
degrees. This is a bounded $t$-structure of
$\cd^b(\mod A)$ whose heart is $\mod A$, which is a
length category. We shall refer to this $t$-structure as the \emph{standard $t$-structure} on $\cd^b(\mod A)$. Recall that a finite-dimensional algebra $A$ can be viewed as a dg algebra concentrated in degree $0$, and in this case, the finite-dimensional derived category $\cd_{fd}(A)$ is equivalent to the bounded derived category $\cd^b(\mod A)$. More generally, for a non-positive dg algebra $A$, the finite-dimensional derived category $\cd_{fd}(A)$ admits a standard bounded $t$-structure $(\cd^{\leq 0}_\std,\cd^{\geq 0}_\std)$ whose heart is $\mod H^0(A)$, see Appendix~\ref{ss:non-pos-dg}.

A bounded $t$-structure is determined by its heart in the sense that its aisle (respectively, its co-aisle) is the smallest full subcategory of $\cc$ which contains the heart and which is closed under the shift functor (respectively, a quasi-inverse of the shift functor) and under taking extensions and direct summands. Moreover,

\begin{lemma}\label{l:grothendieck-gp-and-bbd-t-str}
Let $(\cc^{\leq 0},\cc^{\geq 0})$ be a bounded $t$-structure on $\cc$ with heart $\ca$. Then $\cc$ is idempotent complete (that is, any idempotent in $\cc$ has a kernel) and the embedding $\ca\rightarrow\cc$ induces an isomorphism of Grothendieck groups $K_0(\ca)\rightarrow K_0(\cc)$. In particular, if $\ca$ is a length category, then $K_0(\cc)$ is a free abelian group and isomorphism classes of simple objects in $\ca$ form a basis of $K_0(\cc)$.
\end{lemma}
\begin{proof} See~\cite[Remark 1.15 i)]{Huybrechts11} for a proof of the idempotent completeness of $\cc$. The other statements are known and easy to check.
\end{proof}

Fix a bounded $t$-structure $(\cc^{\leq 0},\cc^{\geq 0})$. A $t$-structure $(\cc'^{\leq 0},\cc'^{\geq 0})$ is \emph{intermediate} with respect to $(\cc^{\leq 0},\cc^{\geq 0})$ if we have
$\Sigma\cc^{\leq 0}\subseteq\cc'^{\leq 0}\subseteq\cc^{\leq 0}$, or equivalently, $\Sigma \cc^{\geq 0}\supseteq \cc'^{\geq 0}\supseteq \cc^{\geq 0}$. Intermediate $t$-structures are bounded and they can be characterised in terms of the position of their hearts.
\begin{lemma}\label{l:intermediate-t-str}
Let $(\cc^{\leq 0},\cc^{\geq 0})$ and $(\cc'^{\leq 0},\cc'^{\geq 0})$ be two bounded $t$-structures on $\cc$. Then $(\cc'^{\leq 0},\cc'^{\geq 0})$ is intermediate with respect to $(\cc^{\leq 0},\cc^{\geq 0})$ if and only if $\cc'^{\leq 0}\cap\cc'^{\geq 0}\subseteq \cc^{\leq 0}\cap\Sigma \cc^{\geq 0}$.
\end{lemma}

In~\cite{HappelReitenSmaloe96}, Happel, Reiten and Smal{\o} introduced an operation, called the \emph{Happel--Reiten--Smal{\o} tilt}, to produce new $t$-structures from given ones. Precisely, let $(\cc^{\leq 0},\cc^{\geq 0})$ be a bounded $t$-structure on $\cc$ with cohomology functors $H^i$ ($i\in\mathbb{Z}$) and let $(\ct,\cf)$ be a torsion pair of its heart.
Define two full subcategories of $\cc$ by
\begin{align*}
\cc'^{\leq 0}&=\{X\in\cc\mid H^i(X)=0 \text{ for } i>0,~ H^0(X)\in\ct\},\\
\cc'^{\geq 0}&=\{X\in\cc\mid H^i(X)=0 \text{ for } i<-1,~H^{-1}(X)\in\cf\}.
\end{align*}
By~\cite[Proposition 2.1]{HappelReitenSmaloe96} and~\cite[Proposition 2.5]{Bridgeland05}, $(\cc'^{\leq 0},\cc'^{\geq 0})$ is a $t$-structure on $\cc$. This $t$-structure is clearly intermediate with respect to $(\cc^{\leq 0},\cc^{\geq 0})$. Moreover, the heart of this new $t$-structure is the extension closure of $\ct$ and $\Sigma \cf$ in $\cc$, \ie the smallest subcategory of $\cc$ which contains $\ct$ and $\Sigma\cf$ and which is closed under extensions; moreover, $(\Sigma\cf,\ct)$ is a torsion pair of this heart.
Clearly the Happel--Reiten--Smal{\o} tilt plays the role of mutations in our context. More precisely,  assume in addition that heart $\ca$ of $(\cc^{\leq 0},\cc^{\geq 0})$ is a length category. For a  simple object  $S$ of $\ca$, let $\cs$ denote the extension closure of $S$ in $\ca$. Then the pair $(^{\perp_{\ca}}S,\cs)$  is a torsion pair of $\ca$. The \emph{left mutation} $\mu^-_S(\cd^{\leq 0},\cd^{\geq 0})$ of $(\cd^{\leq 0},\cd^{\geq 0})$ at $S$ is the Happel--Reiten--Smal{\o} tilt at the torsion pair $(^{\perp_{\ca}}S,\cs)$. We point out that this mutation is different from the mutation defined by Zhou and Zhu in~\cite{ZhouZhu11b} for the more general notion of torsion pairs of triangulated categories.

On the set of $t$-structures on $\cc$ there is a natural partial order
\[(\cc^{\leq 0},\cc^{\geq 0})\leq (\cc'^{\leq 0},\cc'^{\geq 0}):\Leftrightarrow \cc^{\leq 0}\subseteq \cc'^{\leq 0}, \text{ equivalently, } \cc^{\geq 0}\supseteq\cc'^{\geq 0}.\]
For a fixed bounded $t$-structure $(\cc^{\leq 0},\cc^{\geq 0})$, the set of intermediate $t$-structures inherits a partial order, with respect to which $(\cc^{\leq 0},\cc^{\geq 0})$ is the unique maximal element and $(\Sigma\cc^{\leq 0},\Sigma \cc^{\geq 0})$ is the unique minimal element.

For a  non-positive dg algebra $A$ we set
\begin{align*}
\intertstr(A) &= \text{the set of intermediate bounded $t$-structures on $\cd_{fd}(A)$ with length}\\
& \text{heart with respect to the standard $t$-structure $(\cd_\std^{\leq 0},\cd_\std^{\geq 0})$}.
\end{align*}
Here by "a $t$-structure with length heart", we mean a $t$-structure whose heart is a length category. 
The notion of mutation and the poset structure introduced above turn the set $\intertstr(A)$ into an ordered exchange graph.
We illustrate this structure using the same algebra $A$ from Section~\ref{ss:tilting-module}.

\medskip
\noindent{\bf Example.} Let $A$ be the hereditary algebra of type $\mathbb{A}_2$ as in Section~\ref{ss:tilting-module}.
There are five intermediate $t$-structures on $\cd^b(\mod A)$ with respect to the standard $t$-structure. We depict them in the Auslander--Reiten quiver of $\cd^b(\mod A)$, for example, the standard $t$-structure is depicted as
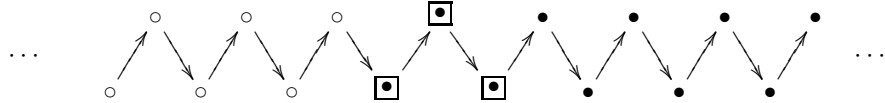
\[\xymatrix@R=0.15pc@C=0.5pc{
&&&\circ\ar[rdd] && \circ\ar[rdd] &&\circ \ar[rdd]&& 
\framebox(8,8){\parbox{5pt}{$\bullet$}}\ar[rdd] && \bullet\ar[rdd] && \bullet\ar[rdd] && \bullet\ar[rdd] &&\bullet \\
\cdots &&&&&&&&&&&&&& &&&&\cdots\\
&&\circ\ar[ruu]&&\circ\ar[ruu]&&\circ\ar[ruu]&& \framebox(8,8){\parbox{5pt}{$\bullet$}}\ar[ruu] && \framebox(8,8){\parbox{5pt}{$\bullet$}}\ar[ruu] && \bullet \ar[ruu]&&\bullet \ar[ruu]&&\bullet\ar[ruu]
}\]
where the objects in black are in the aisle and the objects in boxes belong to $\mod A$. The ordered exchange graph $\intertstr(A)$ is
\[\begin{xy} 0;<0.35pt,0pt>:<0pt,-0.35pt>::
(300,0) *+{\begin{array}{*{11}{c@{\hspace{-0pt}}}}
\scirc&&\scirc&&\framebox(4,4){\parbox{4pt}{$\sbullet$}}&&\sbullet&&\sbullet&&\sbullet\\[-8pt]
&\scirc&&\framebox(4,4){\parbox{4pt}{$\sbullet$}}&&\framebox(4,4){\parbox{4pt}{$\sbullet$}}&&\sbullet&&\sbullet
\end{array}} ="0",
(100,150) *+{\begin{array}{*{11}{c@{\hspace{-0pt}}}}
\scirc&&\scirc&&\framebox(4,4){\parbox{4pt}{$\sbullet$}}&&\sbullet&&\sbullet&&\sbullet\\[-8pt]
&\scirc&&\framebox(4,4){\parbox{4pt}{$\scirc$}}&&\framebox(4,4){\parbox{4pt}{$\sbullet$}}&&\sbullet&&\sbullet
\end{array}}="1",
(100,300) *+{\begin{array}{*{11}{c@{\hspace{-0pt}}}}
\scirc&&\scirc&&\framebox(4,4){\parbox{4pt}{$\scirc$}}&&\sbullet&&\sbullet&&\sbullet\\[-8pt]
&\scirc&&\framebox(4,4){\parbox{4pt}{$\scirc$}}&&\framebox(4,4){\parbox{4pt}{$\sbullet$}}&&\sbullet&&\sbullet
\end{array}}="2",
(500,225) *+{\begin{array}{*{11}{c@{\hspace{-0pt}}}}
\scirc&&\scirc&&\framebox(4,4){\parbox{4pt}{$\scirc$}}&&\sbullet&&\sbullet&&\sbullet\\[-8pt]
&\scirc&&\framebox(4,4){\parbox{4pt}{$\sbullet$}}&&\framebox(4,4){\parbox{4pt}{$\scirc$}}&&\sbullet&&\sbullet
\end{array}}="3",
(300,400) *+{\begin{array}{*{11}{c@{\hspace{-0pt}}}}
\scirc&&\scirc&&\framebox(4,4){\parbox{4pt}{$\scirc$}}&&\sbullet&&\sbullet&&\sbullet\\[-8pt]
&\scirc&&\framebox(4,4){\parbox{4pt}{$\scirc$}}&&\framebox(4,4){\parbox{4pt}{$\scirc$}}&&\sbullet&&\sbullet
\end{array}}="4",
"0", {\ar "1"}, "0", {\ar "3"},
"1", {\ar "2"},
"2", {\ar "4"},
"3", {\ar "4"},  
\end{xy}\]

\subsection{Co-$t$-structures, $\intercotstr(A)$}
Let $\cc$ be a triangulated category. A \emph{co-$t$-structure} on
$\cc$  is a pair $(\cc_{\geq 0},\cc_{\leq
0})$ of strict and full subcategories of $\cc$ such that
\begin{itemize}
\item[$\cdot$] both $\cc_{\geq 0}$ and $\cc_{\leq
0}$ are additive and closed under taking direct summands,
\item[$\cdot$] $\Sigma^{-1}\cc_{\geq 0}\subseteq\cc_{\geq 0}$ and
$\Sigma\cc_{\leq 0}\subseteq\cc_{\leq 0}$;
\item[$\cdot$] $\Hom(M,\Sigma N)=0$ for $M\in\cc_{\geq 0}$
and $N\in\cc_{\leq 0}$,
\item[$\cdot$] for each $M\in\cc$ there is a triangle in $\cc$
\begin{align}
\xymatrix{M'\ar[r] &
M\ar[r]& M''\ar[r] & \Sigma M'}\label{eq:triangle-co-t-str}
\end{align} with $M'\in\cc_{\geq
0}$ and $M''\in\Sigma\cc_{\leq 0}$.
\end{itemize}
The above triangle (\ref{eq:triangle-co-t-str}) is not canonical.
This notion was introduced by Pauksztello in \cite{Pauksztello08} and  independently by Bondarko as \emph{weight structures} in~\cite{Bondarko10}.
The \emph{co-heart} is defined as the intersection $\cc_{\geq
0}~\cap~\cc_{\leq 0}$. Note that the co-heart is usually not an abelian category. 
As for $t$-structures, the subcategories $\cc_{\geq 0}$ and $\cc_{\leq 0}$ are called the \emph{aisle} and \emph{co-aisle} of the co-$t$-structure.  
One easily observes that $\cc_{\geq 0}=\Sigma {}^\perp(\cc_{\leq 0})$ and $\cc_{\leq 0}=\Sigma \cc_{\geq 0}^{\perp}$, thus a co-$t$-structure is determined by its aisle (respectively, by its co-aisle).
The co-$t$-structure $(\cc_{\geq 0},\cc_{\leq
0})$ is said to be \emph{bounded}~\cite{Bondarko10} if
$$\bigcup_{n\in\mathbb{Z}}\Sigma^n \cc_{\leq
0}=\cc=\bigcup_{n\in\mathbb{Z}}\Sigma^n\cc_{\geq 0}.$$ 
A bounded co-$t$-structure is determined by its co-heart in the sense that its aisle (respectively, its co-aisle) is the smallest full subcategory of $\cc$ which contains the co-heart and which is closed under a quasi-inverse of the shift functor (respectively, the shift functor) and under taking extensions and direct summands.
A bounded co-$t$-structure is one of the two ingredients of a J{\o}rgensen--Pauksztello co-stability condition~\cite{JorgensenPauksztello11}.
A typical
example of a co-$t$-structure is the pair $(\cp^\std_{\geq 0},\cp^\std_{\leq 0})$
for the homotopy category $\ch^b(\proj A)$ of a finite-dimensional algebra $A$, where $\cp^\std_{\geq 0}$ consists of
complexes which are homotopy equivalent to a complex bounded below at
$0$, and $\cp^\std_{\leq 0}$ consists of complexes which are homotopy
equivalent to a complex bounded above at $0$. The co-heart of this
co-$t$-structure is $\proj A$. We shall refer to this co-$t$-structure as the \emph{standard} co-$t$-structure on $\ch^b(\proj A)$. Recall that viewing $A$ as a dg algebra concentrated in degree $0$, we have $\per(A)\cong\ch^b(\proj A)$. More generally, for a non-positive dg algebra $A$, the perfect derived category $\per(A)$ admits a standard co-$t$-structure $(\cp^\std_{\geq 0},\cp^\std_{\leq 0})$, see Appendix~\ref{ss:non-pos-dg}.

Fix a bounded  co-$t$-structure $(\cc_{\geq 0},\cc_{\leq 0})$. A co-$t$-structure $(\cc'_{\geq 0},\cc'_{\leq 0})$ is \emph{intermediate} with respect to $(\cc_{\geq 0},\cc_{\leq 0})$ if we have $\Sigma \cc_{\geq 0}\supseteq \cc'_{\geq 0}\supseteq \cc_{\geq 0}$, or equivalently, $\Sigma \cc_{\leq 0}\subseteq\cc'_{\leq 0}\subseteq\cc_{\leq 0}$. Intermediate co-$t$-structures are bounded and they can be characterised in terms of the position of their co-hearts.
\begin{lemma}\label{l:intermediate-co-t-str}
Let $(\cc_{\geq 0},\cc_{\leq 0})$ and $(\cc'_{\geq 0},\cc'_{\leq 0})$ be two bounded co-$t$-structures on $\cc$. Then $(\cc'_{\geq 0},\cc'_{\leq 0})$ is intermediate with respect to $(\cc_{\geq 0},\cc_{\leq 0})$ if and only if $\cc'_{\geq 0}\cap\cc'_{\leq 0}\subseteq \Sigma\cc_{\geq 0}\cap \cc_{\leq 0}$.
\end{lemma}

On the set of co-$t$-structures on $\cc$ there is a natural partial order
\[(\cc_{\geq 0},\cc_{\leq 0})\leq (\cc'_{\geq 0},\cc'_{\leq 0}):\Leftrightarrow \cc_{\geq 0}\supseteq \cc'_{\geq 0},  \text{ equivalently, } \cc_{\leq 0}\subseteq \cc'_{\leq 0}.\]
For a fixed bounded co-$t$-structure $(\cc_{\geq 0},\cc_{\leq 0})$, this order relation induces a partial order on 
the set of intermediate co-$t$-structures, with respect to which $(\cc_{\geq 0},\cc_{\leq 0})$ is the unique maximal element and $(\Sigma\cc_{\geq 0},\Sigma \cc_{\leq 0})$ is the unique minimal element. 

For a  non-positive dg algebra $A$ we set
\begin{align*}
\intercotstr(A) &= \text{the set of intermediate bounded co-$t$-structures on $\per(A)$ }\\
 &    \text{with respect to the standard co-$t$-structure $(\cp^\std_{\geq 0},\cp^\std_{\leq 0})$. }
\end{align*}

\medskip
\noindent
{\bf Example.} Let $A$ be the hereditary algebra of type $\mathbb{A}_2$ as in Section~\ref{ss:tilting-module}. There are five intermediate co-$t$-structures on $\ch^b(\proj A)$ (which, in this case, is equivalent to $\cd^b(\mod A)$) with respect to the standard one. We depict them in the Auslander--Reiten quiver of $\ch^b(\proj A)$, for example, the standard co-$t$-structure is depicted as
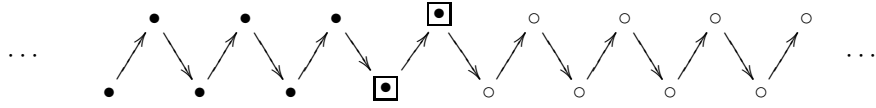
\[\xymatrix@R=0.15pc@C=0.5pc{
&&&\bullet\ar[rdd] && \bullet\ar[rdd] &&\bullet\ar[rdd] && 
\framebox(8,8){\parbox{5pt}{$\bullet$}}\ar[rdd] && \circ\ar[rdd] && \circ\ar[rdd] && \circ\ar[rdd] &&\circ \\
\cdots &&&&&&&&&&&&&& &&&&\cdots\\
&&\bullet\ar[ruu]&&\bullet\ar[ruu]&&\bullet\ar[ruu]&& \framebox(8,8){\parbox{5pt}{$\bullet$}}\ar[ruu] && \circ\ar[ruu] && \circ\ar[ruu] &&\circ\ar[ruu] &&\circ\ar[ruu]
}\]
where the objects in black are in the aisle and the objects in boxes are projective. The Hasse quiver of $\intercotstr(A)$ is
\[\begin{xy} 0;<0.35pt,0pt>:<0pt,-0.35pt>::
(300,0) *+{\begin{array}{*{11}{c@{\hspace{-0pt}}}}
\sbullet&&\sbullet&&\framebox(4,4){\parbox{4pt}{$\sbullet$}}&&\scirc&&\scirc&&\scirc\\[-8pt]
&\sbullet&&\framebox(4,4){\parbox{4pt}{$\sbullet$}}&&\scirc&&\scirc&&\scirc
\end{array}} ="0",
(100,150) *+{\begin{array}{*{11}{c@{\hspace{-0pt}}}}
\sbullet&&\sbullet&&\framebox(4,4){\parbox{4pt}{$\sbullet$}}&&\scirc&&\scirc&&\scirc\\[-8pt]
&\sbullet&&\framebox(4,4){\parbox{4pt}{$\sbullet$}}&&\sbullet&&\scirc&&\scirc
\end{array}}="1",
(100,300) *+{\begin{array}{*{11}{c@{\hspace{-0pt}}}}
\sbullet&&\sbullet&&\framebox(4,4){\parbox{4pt}{$\sbullet$}}&&\sbullet&&\scirc&&\scirc\\[-8pt]
&\sbullet&&\framebox(4,4){\parbox{4pt}{$\sbullet$}}&&\sbullet&&\scirc&&\scirc
\end{array}}="2",
(500,225) *+{\begin{array}{*{11}{c@{\hspace{-0pt}}}}
\sbullet&&\sbullet&&\framebox(4,4){\parbox{4pt}{$\sbullet$}}&&\scirc&&\scirc&&\scirc\\[-8pt]
&\sbullet&&\framebox(4,4){\parbox{4pt}{$\sbullet$}}&&\scirc&&\sbullet&&\scirc
\end{array}}="3",
(300,400) *+{\begin{array}{*{11}{c@{\hspace{-0pt}}}}
\sbullet&&\sbullet&&\framebox(4,4){\parbox{4pt}{$\sbullet$}}&&\sbullet&&\scirc&&\scirc\\[-8pt]
&\sbullet&&\framebox(4,4){\parbox{4pt}{$\sbullet$}}&&\sbullet&&\sbullet&&\scirc
\end{array}}="4",
"0", {\ar "1"}, "0", {\ar "3"},
"1", {\ar "2"},
"2", {\ar "4"},
"3", {\ar "4"},  
\end{xy}\]

\subsection{Silting objects, $\twosilt(A) $}

Let $\cc$ be a triangulated category. An object $M$ of $\cc$ is called a \emph{presilting object} if $\Hom_{\cc}(M,\Sigma^i M)$ vanishes for all $i>0$, a \emph{silting object} if
in addition $M$ generates $\cc$, \ie $\cc=\thick(M)$, and a \emph{tilting object} if further $\Hom_{\cc}(M,\Sigma^i M)$ vanishes also for all $i<0$. Tilting objects play an essential role in the Morita theory of derived categories of algebras (\cite{Happel87,ClineParshallScott86,Rickard89,Keller94}) and
the notion of silting objects, generalising that of tilting objects, was introduced by Keller and Vossieck in~\cite{KellerVossieck88} to study $t$-structures on the bounded derived category of finite-dimensional representations over a Dynkin quiver. A typical example of a silting object is the free module $A_A$ of rank $1$ in $\ch^b(\proj A)$ for a finite-dimensional algebra $A$ (more generally, the free dg module $A_A$ of rank $1$ in $\per(A)$ for a non-positive dg algebra $A$, see Appendix~\ref{ss:non-pos-dg}). In contrast, the bounded derived category $\cd^b(\mod A)$ rarely has silting objects. In fact, we have

\begin{proposition}
Let $A$ be a finite-dimensional algebra. Then $\cd^b(\mod A)$ has a silting object if and only if $A$ has finite global dimension.
\end{proposition}

Indeed, assume that $\cd^b(\mod A)$ has a silting object. Then it follows by d\'evissage that for any objects $M$ and $N$ of $\cd^b(\mod A)$, the space $\Hom(M,\Sigma^i N)$ vanishes for sufficiently large $i$. Since $A$ has only finitely many isomorphism classes of simple modules, there is an integer $d$ such that $\Hom(S,\Sigma^i S')=\Ext^i(S,S')$ vanishes for any simple modules $S$ and $S'$ and for any $i>d$. It follows that the global dimension of $A$ is no greater than $d$.

Triangulated categories with silting objects are quite special in terms of their Grothendieck groups.

\begin{theorem}{(\cite[Theorem 2.37]{AiharaIyama12})}\label{t:generators-of-grothendieckgroup-silting-case}
Let $\cc$ be a Krull--Schmidt triangulated category with a silting object $M$. Then the embedding $\add(M)\rightarrow \cc$ induces an isomorphism $K^{\rm split}_0(\add(M))\rightarrow K_0(\cc)$ of Grothendieck groups. In particular, $K_0(\cc)$ is a free abelian group of rank $|M|$.
\end{theorem}


There is the following question on whether any presilting object can be completed to a silting object.

\begin{question}\label{q:completion-of-presilting} Let $\cc$ be a Krull--Schmidt triangulated category.
Suppose that $\cc$ has a silting object $M$ and let $N$ be a presilting object of $\cc$. Is there an object $N'$ such that $N\oplus N'$ is a silting object?
\end{question}

If we replace "silting" by "tilting", this question is known to have a negative answer in general, see for example~\cite[Section 8]{Rickard89}, \cite[Section 2]{RickardSchofield89} and~\cite[Example 4.4]{LiuVitoriaYang12}. In general, Question~\ref{q:completion-of-presilting} is open. It has a positive answer for $\cc=\ch^b(\proj A)$, where $A$ is a representation-finite symmetric algebra, see~\cite{AbeHoshino06,Aihara10}, or $A$ is a piecewise hereditary algebra\footnote{The key point of a proof is that in this case (pre)silting objects are closely related to (complete) exceptional sequences, see~\cite[Section 3]{AiharaIyama12}.}. 
Another direction is to consider presilting objects $N$ which are \emph{2-term} with respect to the given silting object $M$, \ie there is a triangle
\[\xymatrix{M^{-1}\ar[r] & M^0\ar[r] & N\ar[r] &\Sigma M^{-1}}\]
with $M^{-1}$ and $M^0$ in $\add(M)$. The following result is due to Derksen--Fei~\cite[Section 5]{DerksenFei09}, Aihara~\cite[Proposition 2.16]{Aihara10}, Wei~\cite[Section 6]{Wei11} and Iyama--J{\o}rgensen--Yang~\cite{IyamaJorgensenYang14} in various generalities. The idea is to form an analogue of the Bongartz complement for tilting modules (Theorem~\ref{t:bongartz-complement}).

\begin{proposition}\label{l:partial-silting} Let $\cc$ be a Hom-finite Krull--Schmidt triangulated category with a silting object $M$ and let $N$ be a presilting object of $\cc$. Assume that $N$ is $2$-term with respect to $M$.
Then $N$ can be completed to a silting object. Moreover, $N$ is a silting object if and only if $|N|=|M|$.
\end{proposition}

Assume that $\cc$ is a Hom-finite Krull--Schmidt triangulated category. Let $M=M_1\oplus\ldots\oplus M_n$ be a basic silting object, where $M_1,\ldots,M_n$ are indecomposable. For $1 \le i\le n$,
the \emph{left mutation} of $M$ at the direct summand $M_i$ is the object
$\mu^-_i(M)=\mu^-_{M_i}(M)=M_i^*\oplus\bigoplus_{j\neq i}M_j$ where $M_i^*$ is the cone of
the minimal left $\add(\bigoplus_{j\neq i}M_j)$-approximation of $M_i$
$$\xymatrix{M_i\ar[r]& E.}$$
The object $\mu^-_i(M)$ is again a silting object. Silting mutation was defined and studied by Buan, Reiten and Thomas in~\cite{BuanReitenThomas11} for bounded derived categories of finite-dimensional hereditary algebras and independently by Aihara and Iyama in~\cite{AiharaIyama12} for the general case. It was shown that the Brenner--Butler-tilting module is the left mutation of the free module of rank $1$, see~\cite[Theorem 2.53]{AiharaIyama12} and~\cite[Proposition 7.4]{KoenigYang12}.

For a  non-positive dg algebra $A$ we say a silting object is $2$-term if it is $2$-term with respect to $A$, and we set
\begin{align*}
\twosilt(A) &= \text{the set of isomorphism classes of basic $2$-term silting objects of $\per(A)$}.
\end{align*}

\medskip
\noindent
{\bf Example.}  Let $A$ be the hereditary algebra of type $\mathbb{A}_2$ as in Section~\ref{ss:tilting-module}. There are precisely five isomorphism classes of basic 2-term silting objects in $\ch^b(\proj A)$ with respect to $A$: $S_1\oplus P_2=A$, $S_2\oplus P_2$, $S_2\oplus \Sigma S_1$, $S_1\oplus \Sigma P_2$ and $\Sigma S_1\oplus \Sigma P_2=\Sigma A$. The ordered exchange graph $\twosilt(A)$ is
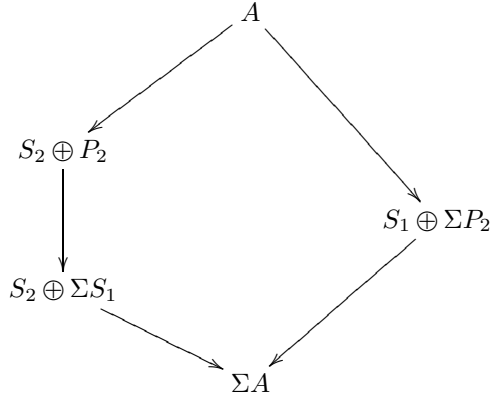
\[\begin{xy} 0;<0.35pt,0pt>:<0pt,-0.35pt>::
(300,0) *+{A} ="0",
(100,150) *+{S_2\oplus P_2}="1",
(100,300) *+{S_2\oplus \Sigma S_1}="2",
(500,225) *+{S_1\oplus \Sigma P_2}="3",
(300,400) *+{\Sigma A}="4",
"0", {\ar "1"}, "0", {\ar "3"},
"1", {\ar "2"},
"2", {\ar "4"},
"3", {\ar "4"},  
\end{xy}\]

\subsection{Simple-minded collections, $\intersmc(A)$} Let $\cc$ be a triangulated category.
A collection $\{X_1,\ldots,X_n\}$ of objects of $\cc$ is said to be
\emph{simple-minded} (cohomologically Schurian
in~\cite{Al-Nofayee09}) if the following conditions hold for
$i,j=1,\ldots,n$
\begin{itemize}
\item[$\cdot$] $\Hom(X_i,\Sigma^p X_j)=0,~~\forall~p<0$,
\item[$\cdot$] $\Hom(X_i,X_j)=\begin{cases} k & \text{if\ }i=j,\\
                                           0 & \text{otherwise},
                                           \end{cases}$
\item[$\cdot$] $X_1,\ldots,X_n$ generate
$\cc$.
\end{itemize}
This notion was first used by Rickard in~\cite{Rickard02} to help constructing derived equivalences of symmetric algebras from stable equivalences. Spherical collections in algebraic geometry~\cite{SeidelThomas01} are examples of simple-minded collections. In representation theory, a typical example of a simple-minded collection is a complete collection of pairwise non-isomorphic simple modules over a finite-dimensional algebra $A$, considered as objects in $\cd^b(\mod A)$.

The following result is a counterpart of Theorem~\ref{t:generators-of-grothendieckgroup-silting-case}.

\begin{proposition} Assume that $\cc$ is Krull--Schmidt and Hom-finite and has a simple-minded collection $\{X_1,\ldots,X_n\}$. Then the Grothendieck group $K_0(\cc)$ of $\cc$ is free of rank $n$. In particular, any two simple-minded collections of $\cc$ have the same cardinality.
\end{proposition}
The idea of the proof is to show that a simple-minded collection is a complete collection of pairwise non-isomorphic simple objects of the heart of a bounded $t$-structure and then use Lemma~\ref{l:grothendieck-gp-and-bbd-t-str} (see for example~\cite[Section 5.5]{KoenigYang12}). 
The construction of the $t$-structure will be given in Section~\ref{ss:silting-smc-tstr-cotstr}.

Assume that $\cc$ is Krull--Schmidt and Hom-finite. 
Let $\{X_1,\ldots,X_n\}$ be a simple-minded collection in $\cc$ and fix
$i=1,\ldots,n$. Let $\cx_i$ denote the extension closure of $X_i$ in
$\cc$. Assume that for any $j$ the object $\Sigma^{-1}X_j$ admits a
minimal left approximation $g_j:\Sigma^{-1}X_j\rightarrow X_{ij}$ in
$\cx_i$. The \emph{left mutation}
$\mu^-_i(X_1,\ldots,X_n)$ of $X_1,\ldots,X_n$ at $X_i$ is a new
collection $\{X_1',\ldots,X_n'\}$ such that $X_i'=\Sigma X_i$ and $X_j'$
($j\neq i$) is the cone of the above left approximation $g_j$.
This was introduced by
Kontsevich and Soibelman for spherical
collections in~\cite[Section 8.1]{KontsevichSoibelman08} and by Koenig and the second-named author for the general case in~\cite{KoenigYang12}. Under suitable conditions on $\cc$ (\eg these conditions are satisfied when $\cc=\cd^b(\mod A)$, where $A$ is a finite-dimensional algebra), the new collection $\mu^-_i(X_1,\ldots,X_n)$ is simple-minded (see \cite[Section 7.2]{KoenigYang12}).

Let $A$ be a non-positive dg algebra such that $H^0(A)$ is finite-dimensional. Let $\{S_1,\ldots,S_n\}$ be a complete collection of pairwise non-isomorphic simple $H^0(A)$-modules. Then $\{S_1,\ldots,S_n\}$ is a simple-minded collection in $\cd_{fd}(A)$, see Appendix~\ref{ss:non-pos-dg}.
A simple-minded collection $\{X_1,\ldots,X_n\}$ is \emph{2-term}  if $H^p(X_i)$ vanishes for any $p\neq 0,-1$ and any $i=1,\ldots,n$, equivalently, $\{X_1,\ldots,X_n\}$ is contained in the extension closure of $\mod ^0(A)$ and $\Sigma\mod H^0(A)$. We put
\begin{align*}
\intersmc(A) &= \text{the set of isoclasses of 2-term simple-minded collections of $\cd_{fd}(A)$}.
\end{align*}

\medskip
\noindent
{\bf Example.} Let $A$ be the hereditary algebra of type $\mathbb{A}_2$ as in Section~\ref{ss:tilting-module}. There are precisely five isomorphism classes of 2-term simple-minded collections in $\cd^b(\mod A)$: $\{S_1,S_2\}$, $\{\Sigma S_1,P_2\}$, $\{S_2,\Sigma P_2\}$, $\{S_1\oplus \Sigma S_2\}$ and $\{\Sigma S_1,\Sigma S_2\}$. The ordered exchange graph of 2-term simple-minded collections in $\cd^b(\mod A)$ is
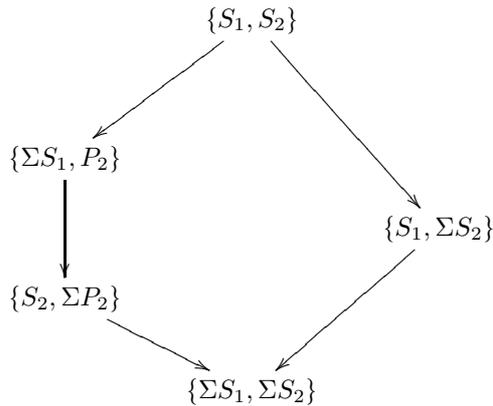
\[\begin{xy} 0;<0.35pt,0pt>:<0pt,-0.35pt>::
(300,0) *+{\{S_1,S_2\}} ="0",
(100,150) *+{\{\Sigma S_1,P_2\}}="1",
(100,300) *+{\{S_2,\Sigma P_2\}}="2",
(500,225) *+{\{S_1,\Sigma S_2\}}="3",
(300,400) *+{\{\Sigma S_1,\Sigma S_2\}}="4",
"0", {\ar "1"}, "0", {\ar "3"},
"1", {\ar "2"},
"2", {\ar "4"},
"3", {\ar "4"},  
\end{xy}\]

\subsection{Cluster-tilting objects, $\cto(\cc)$}\label{ss:cluster-tilting}

Let $\cc$ be a Hom-finite 2-Calabi--Yau Krull--Schmidt triangulated category. A typical example of such a triangulated category is the \emph{cluster category $\cc_Q$} of an acyclic quiver $Q$ (that is, a quiver without oriented cycles), defined as the orbit category $\cd^b(\mod kQ)/\tau^{-1}\circ\Sigma$ (\cite{BuanMarshReinekeReitenTodorov06,Keller05}), where $\tau$ is the Auslander--Reiten translation of $\cd^b(\mod kQ)$. An object $M$ of $\cc$ is called a \emph{rigid object} if $\Hom(M,\Sigma M)=0$ holds, and a \emph{cluster-tilting object} if further the following equality holds
\[\add(M)=\{X\in\cc\mid \Hom(M,\Sigma X)=0\}.\]
For an acyclic quiver $Q$, the free $kQ$-module of rank $1$, considered as an object in $\cc_Q$, is a cluster-tilting object.
A cluster-tilting object $M$ generates $\cc$ in two steps, namely, for any object $X$ of $\cc$, there is a triangle
\[\xymatrix{M^{-1}\ar[r] & M^0\ar[r] & X\ar[r] & \Sigma M^{-1}}\]
with $M^{-1}$ and $M^0$ in $\add(M)$, see \cite[Proposition 2.1 (b)]{KellerReiten07}.

Combining~\cite[Theorem 2.6]{ZhouZhu11} and~\cite[Corollary 3.7.2]{ZhouZhu11}, we obtain the following result.

\begin{proposition}{(\cite{ZhouZhu11})}\label{p:cto-and-number-of-summands} Assume that $\cc$ has a cluster-tilting object $M$.
A rigid object $N$ of $\cc$ is a cluster-tilting object if and only if $|N|=|M|$.
\end{proposition}

The mutation of cluster-tilting objects was introduced by Buan, Marsh, Reineke, Reiten and Todorov~\cite{BuanMarshReinekeReitenTodorov06} for the case when $\cc=\cc_Q$ is the cluster category of an acyclic quiver $Q$ and by Iyama and Yoshino~\cite{IyamaYoshino08} for general $\cc$. Let $M=M_1\oplus\ldots\oplus M_n$ be a cluster-tilting object of $\cc$, where $M_1,\ldots,M_n$ are pairwise non-isomorphic indecomposable objects. Fix $i=1,\ldots,n$. Thanks to the 2-Calabi--Yau property of $\cc$, there is an indecomposable object $M_i^*$, unique up to isomorphism, such that $M_i^*$ is not isomorphic to $M_i$ and $\mu_i(M)=\mu_{M_i}(M)=M_i^*\oplus\bigoplus_{j\neq i}M_j$ is again a cluster-tilting object. $M_i^*$ can be obtained through the following triangles (called \emph{exchange triangles})
\[ \xymatrix{M_i\ar[r]^f & E\ar[r] & M_i^*\ar[r] & \Sigma M_i} \text{ and }\xymatrix{M_i^*\ar[r] & E'\ar[r]^{f'} &M_i\ar[r] & \Sigma M_i^*} \]
where $f$ and $f'$ are respectively left and right $\add(\bigoplus_{j\neq i}M_j)$-approximations of $M_i$. The object $\mu_i(M)$ is the \emph{mutation} of $M$ at the summand $M_i$. Note that this mutation is not naturally oriented. We put
\begin{align*}
\cto(\cc) &= \text{the set of isomorphism classes of basic cluster-tilting objects of $\cc$}.
\end{align*}
This has the structure of an exchange graph by Proposition~\ref{p:cto-and-number-of-summands} and \cite[Theorem 5.3]{IyamaYoshino08}.  The underlying set of this exchange graph consists of indecomposable rigid objects of $\cc$ and the compatibility relation $R$ is given by $(M,N)\in R$ if and only if $\Hom_\cc(M,\Sigma N)=0$.

\medskip
\noindent
{\bf Example.} Let $Q$ be the quiver of type $\mathbb{A}_2$ and $A=kQ$, as in Section~\ref{ss:tilting-module}. The cluster category $\cc_Q$ of $Q$ has precisely five indecomposable objects: $S_1$, $P_2$, $S_2$, $\Sigma S_1$ and $\Sigma P_2$, and it has five basic cluster-tilting objects: $S_1\oplus P_2=A$, $S_2\oplus P_2$, $S_2\oplus\Sigma S_1$, $\Sigma P_2\oplus\Sigma S_1=\Sigma A$ and $\Sigma P_2\oplus S_1$. The exchange graph $\cto(\cc_Q)$ is
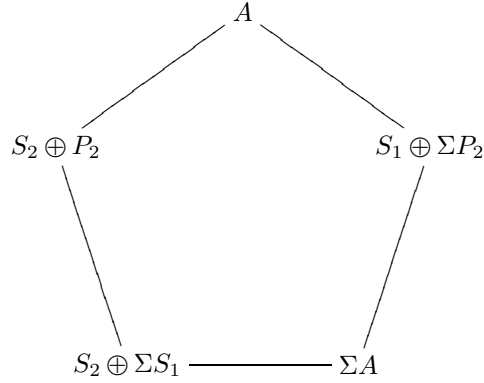
\[\begin{xy} 0;<0.35pt,0pt>:<0pt,-0.35pt>::
(200,0) *+{A} ="0",
(0,145.3) *+{S_2\oplus P_2}="1",
(76.3,380.5) *+{S_2\oplus \Sigma S_1}="2",
(400,145.3) *+{S_1\oplus \Sigma P_2}="3",
(323.7,380.5) *+{\Sigma A}="4",
"0", {\ar@{-} "1"}, "0", {\ar@{-} "3"},
"1", {\ar@{-} "2"},
"2", {\ar@{-} "4"},
"3", {\ar@{-} "4"},  
\end{xy}\]

In \cite{IyamaYoshino08}, Iyama and Yoshino introduce a technique called \emph{2-Calabi--Yau reduction}. Precisely, let $\cc$ be a $2$-Calabi--Yau triangulated category and $M$ be a basic rigid object. Consider the category $(\Sigma^{-1}M)^\perp$, which consists of objects $N$ such that $\Hom(M,\Sigma N)$ vanishes, and consider the full subgraph $\cto_M(\cc)$ of $\cto(\cc)$ whose vertices are the basic cluster-tilting objects containing $M$ as a direct summand.

\begin{theorem}(\cite[Theorems 4.7 and 4.9]{IyamaYoshino08})\label{t:cy-reduction}
Keep the notation as above.
\begin{itemize}
\item[(a)] The additive quotient $\bar{\cc}=(\Sigma^{-1}M)^\perp/\add(M)$ has the natural structure of a $2$-Calabi--Yau triangulated category.
\item[(b)] There is an isomorphism of exchange graphs induced by the canonical quotient functor $(\Sigma^{-1}M)^\perp\rightarrow\bar{\cc}$:
\[
\xymatrix{\cto_M(\cc)\ar[r] & \cto(\bar{\cc})}.
\]
\end{itemize}
\end{theorem}

\subsection{Framed quivers, $\mut(Q)$}\label{ss:mut}
A \emph{cluster quiver} is a finite quiver without loops or oriented two-cycles. Let $(Q,F)$ be an \emph{ice quiver}, \ie $Q$ is a cluster quiver and $F \subset Q_0$ is a (possibly empty) subset of vertices called the \emph{frozen vertices} such that there are no arrows between them. For a non-frozen vertex $i$, the \emph{mutation $\mu_i(Q,F)$ of $(Q,F)$ at $i$} is the new ice quiver $(Q',F)$, where $Q'$ is obtained from $Q$ by applying the following modifications:
			
\begin{enumerate}
\item For any pair of arrows $h \xrightarrow{a} i \xrightarrow{b} j$ in $Q$, add an arrow $h \xrightarrow{[ab]} j$;
\item Any arrow $h \xrightarrow{a} i$ in $Q$ is replaced by an arrow $h \xleftarrow{a^*} i$, and 
any arrow $i \xrightarrow{b} j$ in $Q$ is replaced by an arrow $i \xleftarrow{b^*} j$;
\item A maximal collection of pairwise disjoint oriented 2-cycles is removed;
\item All arrows between frozen vertices are removed.
\end{enumerate}

It is easy to check that $\mu_i(\mu_i(Q,F))=(Q,F)$. A non-frozen vertex $i \in Q_0$ is called \emph{green} if 
$$\ens{j' \in F \ | \ \exists\, i \fl j' \in Q_1} = \emptyset.$$
It is called \emph{red} if 
$$\ens{j' \in F \ | \ \exists\, j' \fl i \in Q_1} = \emptyset.$$
We warn the reader that we have adopted the conventions opposite to those in~\cite{BDP,Keller11b} to keep coherent with mutations of the categorical objects. When we mutate at a green or red vertex, the above step (4) is redundant because in this case no arrows between frozen vertices are produced in steps (1)--(3).

Two ice quivers are called \emph{mutation-equivalent} if one can be obtained from the other by applying a finite number of successive mutations at non-frozen vertices. Since mutations are involutive, this defines an equivalence relation on the set of ice quivers. The equivalence class of an ice quiver $(Q,F)$ is called its \emph{mutation class} and is denoted by $\Mut(Q,F)$.

\smallskip

The \emph{framed quiver} associated to $Q$ is the quiver $\hat{Q}$ obtained from $Q$ by adding an extra vertex $j'$ together with an arrow $j' \to j$ for each vertex $j$ of $Q$.
We identify the set $Q_0$ of vertices of the quiver $Q$ with $\{1, \ldots, n \}$ and write the set $\hat{Q}_0$ of vertices of the quiver $\hat{Q}$ as  $\hat{Q}_0 = Q_0 \cup Q'_0$ where $Q'_0 = \{1', \ldots, n' \}$. 
We consider $\hat{Q}$ as an ice quiver with frozen vertices $F=Q'_0$.
The ice quiver with the same frozen vertices $F=Q'_0$ and an arrow $j \to j'$ for each vertex $j$ of $Q$ is denoted by  ${\check{Q}}$. By definition all vertices of $\hat{Q}$ are green and all vertices of $\check{Q}$ are red.

\begin{proposition}{(\cite[Theorem 1.6]{BDP})}\label{theorem:greenorred}
Let $Q$ be a cluster quiver and $R \in \Mut(\hat Q)$. Then any non-frozen vertex in $R_0$ is either green or red.
\end{proposition}

Two ice quivers $(Q,F)$ and $(\tilde{Q},F)$ sharing the same set of frozen vertices are called \emph{isomorphic as ice quivers} if there is an isomorphism of quivers $\phi: Q \fl \tilde{Q}$ fixing $F$. In this case, we write $(Q,F) \simeq (\tilde{Q},F)$ and we denote by $[(Q,F)]$ the isomorphism class of the ice quiver $(Q,F)$. 

The ordered exchange graph of a cluster quiver $Q$ is the oriented graph $\mut(Q)$ whose vertices are the isomorphism classes $[R]$ of ice quivers $R \in \Mut(\hat{Q})$ and where there is an arrow $[R] \to [R']$ in $\mut (Q)$ if and only if there exists a green vertex $i \in R_0$ such that $\mu_i(R) \simeq R'$.

\medskip
\noindent
{\bf Example.} Let $Q$ be the quiver of type ${\mathbb A}_2$ as in Section~\ref{ss:tilting-module}. There are precisely five  isomorphism classes $[R]$ of ice quivers $R \in \Mut(\hat{Q})$. The ordered exchange graph $\mut(Q)$ is
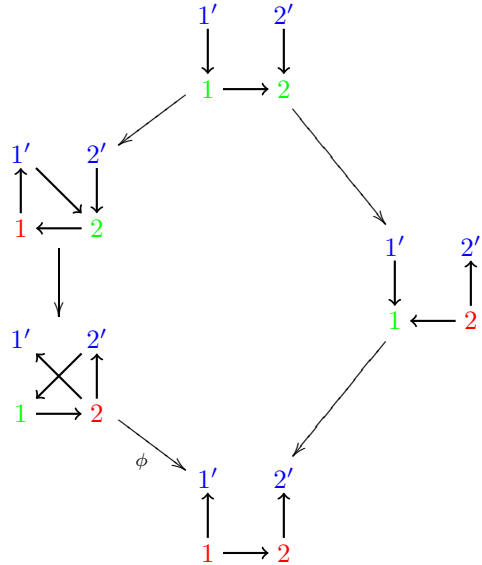
\[\begin{xy} 0;<0.35pt,0pt>:<0pt,-0.35pt>::
(300,0) *+{
\begin{tikzpicture}
[inner sep=0.5mm]
\draw  [thick,->] (0.2,0) -- (0.8,0);
\draw  [thick,->] (0,0.8) -- (0,0.2);
\draw  [thick,->] (1,0.8) -- (1,0.2);
  \node [ blue] at (0,1) {$1'$};
 \node [ blue] at (1,1) {$2'$};
 \node [ green] at (0,0) {$1$};
 \node [ green] at (1,0) {$2$};
 \end{tikzpicture}
} ="0",
(100,150) *+{\begin{tikzpicture}
[inner sep=0.5mm]
\draw  [thick,->] (0.8,0) -- (0.2,0);
\draw  [thick,->] (0,0.2) -- (0,0.8);
\draw  [thick,->] (1,0.8) -- (1,0.2);
\draw  [thick,->] (0.2,0.8) -- (0.8,0.2);
  \node [ blue] at (0,1) {$1'$};
 \node [ blue ] at (1,1) {$2'$};
 \node [ red ] at (0,0) {$1$};
 \node [ green ] at (1,0) {$2$};
 \end{tikzpicture}
}="1",
(100,350) *+{\begin{tikzpicture}
[inner sep=0.5mm]
\draw  [thick,->] (0.2,0) -- (0.8,0);
\draw  [thick,->] (1,0.2) -- (1,0.8);
\draw  [thick,->] (0.8,0.8) -- (0.2,0.2);
\draw  [thick,->] (0.8,0.2) -- (0.2,0.8);
  \node [ blue ] at (0,1) {$1'$};
 \node [ blue ] at (1,1) {$2'$};
 \node [ green ] at (0,0) {$1$};
 \node [ red ] at (1,0) {$2$};
 \end{tikzpicture}
}="2",
(500,250) *+{\begin{tikzpicture}
[inner sep=0.5mm]
\draw  [thick,->] (0.8,0) -- (0.2,0);
\draw  [thick,->] (1,0.2) -- (1,0.8);
\draw  [thick,->] (0,0.8) -- (0,0.2);
  \node [ blue ] at (0,1) {$1'$};
 \node [ blue ] at (1,1) {$2'$};
 \node [ green ] at (0,0) {$1$};
 \node [ red ] at (1,0) {$2$};
 \end{tikzpicture}
}="3",
(300,500) *+{\begin{tikzpicture}
[inner sep=0.5mm]
\draw  [thick,->] (0.2,0) -- (0.8,0);
\draw  [thick,->] (0,0.2) -- (0,0.8);
\draw  [thick,->] (1,0.2) -- (1,0.8);
  \node [ blue ] at (0,1) {$1'$};
 \node [ blue ] at (1,1) {$2'$};
 \node [ red ] at (0,0) {$1$};
 \node [ red ] at (1,0) {$2$};
 \end{tikzpicture}
}="4",
"0", {\ar "1"}, "0", {\ar "3"},
"1", {\ar "2"},
"2", {\ar_{\phi} "4"},
"3", {\ar "4"},  
\end{xy}\]
where the isomorphism of ice quivers $\phi$ interchanges the vertices $1$ and $2$.
We indicated in the diagram above also the colour-coding of green and red vertices.
More examples can easily be computed using Bernhard Keller's java applet \cite{Keller:javaapplet} or  the \verb\Quiver Mutation Explorer\ \cite{DupontPerotin:QME}.

\begin{proposition}{(\cite[Corollary 1.12]{BDP})}\label{corol:uniquesource}
			Let $Q$ be a cluster quiver. Then:
			\begin{enumerate}
				\item $\mut(Q)$ has a unique source, which is $[\hat{Q}]$.
				\item $\mut(Q)$ has a sink if and only if $[{\check Q}]$ is a vertex in $\mut(Q)$ and in this case $[{\check Q}]$ is the unique sink.
			\end{enumerate}
		\end{proposition}
		
Note that $\mut(Q)$ has a sink when the quiver $A$ is acyclic, but for general quivers it is not known when 	$\mut(Q)$ admits a sink. For instance for the following quiver the ordered exchange graph  $\mut(Q)$ admits no sink:
			$$
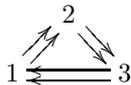
\xymatrix@=1em{
					& 2 \ar@<+2pt>[rd] \ar@<-2pt>[rd] \\
				1 \ar@<+2pt>[ru] \ar@<-2pt>[ru] && 3 \ar@<+2pt>[ll] \ar@<-2pt>[ll]
			}$$
Moreover, there is the following general result.

\begin{proposition}{(\cite[Proposition7.1+Proposition 1.13]{BDP})}
Let $Q$ be a cluster quiver. If there is a non-degenerate potential $W$ on $Q$ such that the Jacobian algebra of $(Q,W)$ is infinite-dimensional, then $\mut(Q)$ does not have a sink.
\end{proposition}

\subsection{Clusters, $\Cl(Q)$}\label{ss:cluster}
Introduced and further investigated by Fomin and Zelevinsky in \cite{cluster1,cluster2, cluster4}, cluster algebras are commutative rings equipped with a distinguished set of generators, the cluster variables which are grouped into overlapping sets of  variables, the clusters.  The cluster variables are defined recursively by iterated mutations of an initial cluster $\{x_1,\ldots, x_n\}$. The dynamics of this mutation process are encoded by the exchange matrices. This is the central notion in cluster theory. The representation theory side of the  theory has been much motivated by attempting to prove various conjectures of Fomin and Zelevinsky on cluster algebras. 	
Due to the separation formulas~\cite[Theorem 3.7, Proposition 3.13 and Corollary 6.3]{cluster4}, the cluster algebras with principal coefficients govern the combinatorics of all cluster algebras (see~\cite{CerulliKellerLabardiniPlamondon12} for a stronger result in the skew-symmetric case). In this article we will only be concerned with skew-symmetric cluster algebras with principal coefficients, which are defined for skew-symmetric matrices with integer entries, or equivalently, for cluster quivers.

Let $n\in\mathbb{Z}$. Consider the field $\mathbb{Q}(x_1,\ldots,x_{2n})$ of rational functions in variables $x_1,\ldots,x_{2n}$. A \emph{seed} is a pair $(\underline{u},(Q,F))$, where $\underline{u}=\{u_1,\ldots,u_{2n}\}$ is a set of elements which freely generate $\mathbb{Q}(x_1,\ldots,x_{2n})$, and $(Q,F)$ is an ice quiver with $Q_0=\{1,\ldots,2n\}$ and $F=\{n+1,\ldots,2n\}$. For $1\leq i\leq n$, the \emph{mutation of $(\underline{u},(Q,F))$ in direction $i$} is defined to be the seed 
$(\underline{u}',\mu_i(Q,F))$
where the mutated cluster is 
\[
\underline{u}'= (\{u_1,\ldots, u_{2n}\} \backslash \{u_i\}) \cup \{u_i'\}
\]
and $u_i'$ is defined by the {\em exchange relation}
\[
u_i' = \frac{1}{u_i} \left(\prod_{\alpha:i\to j} u_j  \quad + \quad \prod_{\beta:l\to i} u_l \right).
\]

Let $Q$ be a cluster quiver with $Q_0=\{1,\ldots,n\}$. Let $\hat{Q}$ be the framed quiver associated to $Q$ defined in \ref{ss:mut} (here we will use the convention $i'=n+i$ for any $1\leq i\leq n$). Fix the \emph{initial seed} $(\{x_1,\ldots,x_{2n}\},\hat{Q})$. A \emph{cluster} of $Q$ is a set $\underline{u}$ appearing in a seed $(\underline{u},R)$ of $Q$, that is, a seed obtained from the initial seed $(\{x_1,\ldots,x_{2n}\},\hat{Q})$ by iterated mutations. Elements of clusters are called \emph{cluster variables}. Notice that the cluster variables $x_{n+1},\ldots,x_{2n}$ are never mutated, so we will call them the \emph{frozen cluster variables}. The \emph{cluster algebra of $Q$ with principal coefficients}, denoted by $\ca_Q$, is the $\mathbb{Z}$-subalgebra of $\mathbb{Q}(x_1,\ldots,x_{2n})$ generated by all cluster variables of $Q$.

The following remarkable result was conjectured by Fomin and Zelevinky~\cite[Conjecture 4.14 (2)]{FominZelevinsky03} (for all cluster algebras) and proved by Cerulli, Keller, Labardini and Plamondon~\cite[Corollary 3.6]{CerulliKellerLabardiniPlamondon12} (for all skew-symmetric cluster algebras) by using additive categorification, see also ~\cite[Theorem 3]{GekhtmanShapiroVainshtein08} and~\cite[Theorem 4.1]{BuanMarshReinekeReitenTodorov06}.

\begin{theorem}\label{t:cluster-determine-seed} Let $Q$ be a cluster quiver.
A seed $(\underline{u},R)$ of $Q$ is determined by its cluster $\underline{u}$.
\end{theorem}

Thanks to this theorem, mutation of seeds induces a well-defined mutation of clusters.
The set of all clusters of $\hat{Q}$, joined by edges corresponding to mutation is called the {\em cluster exchange graph}, denoted $\Cl(Q)$. 

\medskip
\noindent
{\bf Example.} Let $Q$ be the quiver of type ${\mathbb A}_2$ as in Section~\ref{ss:tilting-module}. The cluster algebra $\mathcal A_Q$ with principal coefficients
 of $Q$ has precisely five non-frozen cluster variables: 
 $$x_1, \; x_2, \; \frac{x_2+x_3}{x_1},\; \frac{x_2+x_3+x_1x_3x_4}{x_1x_2}\; \mbox{ and } \;\frac{1+x_1x_4}{x_2}$$ 
 which, together with the frozen cluster variables, are grouped into five clusters: 
\[
\{x_1,x_2,x_3,x_4\},~~\{\frac{x_2+x_3}{x_1},x_2,x_3,x_4\},~~\{\frac{x_2+x_3}{x_1},\frac{x_2+x_3+x_1x_3x_4}{x_1x_2},x_3,x_4\},
\]
\[
\{x_1,\frac{1+x_1x_4}{x_2},x_3,x_4\}~~~~\mbox{and} ~~~~\{\frac{x_2+x_3+x_1x_3x_4}{x_1x_2},\frac{1+x_1x_4}{x_2}, x_3,x_4\}
\] The cluster exchange graph $\Cl(Q)$ is as follows:
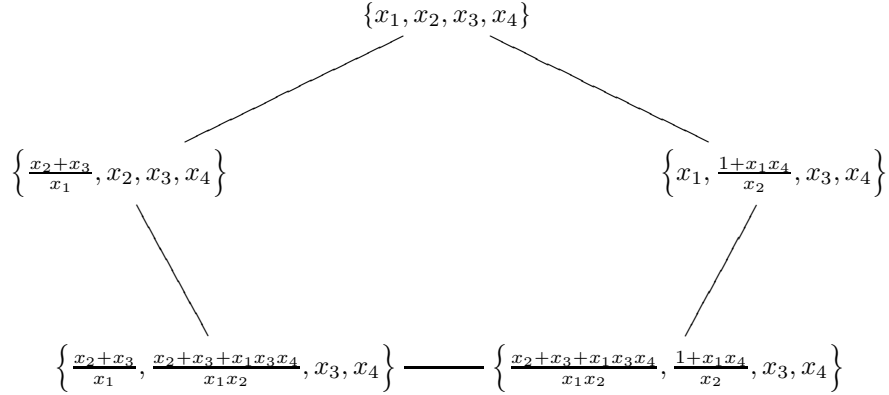
\[\begin{xy} 0;<0.35pt,0pt>:<0pt,-0.35pt>::
(350,0) *+{\left\{x_1,x_2,x_3,x_4\right\} } ="0",
(0,170.6) *+{\left\{\frac{x_2+x_3}{x_1},x_2,x_3,x_4\right\}}="1",
(113.3,380) *+{\left\{\frac{x_2+x_3}{x_1},\frac{x_2+x_3+x_1x_3x_4}{x_1x_2},x_3,x_4\right\}  }="2",
(700,170.6) *+{\left\{x_1,\frac{1+x_1x_4}{x_2},x_3,x_4\right\} }="3",
(586.7,380) *+{\left\{\frac{x_2+x_3+x_1x_3x_4}{x_1x_2},\frac{1+x_1x_4}{x_2},x_3,x_4\right\}  }="4",
"0", {\ar@{-}"1"}, "0", {\ar@{-} "3"},
"1", {\ar@{-} "2"},
"2", {\ar@{-} "4"},
"3", {\ar@{-} "4"},  
\end{xy}\]

One can see from this simple example the following general fact, referred to as the Laurent  phenomenon.

\begin{proposition}\label{Laurent}{(\cite[Proposition 3.6]{cluster4})} Let $Q$ be a cluster quiver. Then each cluster variable of 
${\mathcal A}_Q$ belongs to $\mathbb{Z}[x_1^{\pm 1},\ldots,x_n^{\pm 1},x_{n+1},\ldots,x_{2n}]$.
\end{proposition}

Much of combinatorics of clusters is encoded in $\ccc$-matrices and $\gg$-matrices (see~\cite{cluster4}), which are introduced in the next two subsections.

\subsection{c-matrices, $\cmat(Q)$}\label{ss:cmatrix}

For an ice quiver $(Q,F)$ with $Q_0=\{1,\ldots,n+m\}$ and $F=\{n+1,\ldots,n+m\}$, we define a matrix $B=B(Q,F)=(b_{ij})_{1\leq i\leq n+m,1\leq j\leq n}$ by setting $b_{ij}$ to be the number of arrows from $i$ to $j$ minus the number of arrows from $j$ to $i$. We will call $B$ the \emph{matrix} of $(Q,F)$.

Let $Q$ be a cluster quiver and $\hat{Q}$ be the framed quiver associated to $Q$. For $R\in\Mut(\hat{Q})$,  the \emph{$\ccc$-matrix} $C=(c_{ij})_{1\leq i,j\leq n}$ of $R$ is the lower half of the matrix $B(R)=(b_{ij})_{1\leq i\leq 2n,1\leq i\leq n}$, namely, $c_{ij}=b_{n+i,j}$. Columns of $\ccc$-matrices are called \emph{$\ccc$-vectors}. For instance, we have $\c(\hat Q) = I_n$ and $\c(\check Q) = -I_n$. For more details on $\ccc$-vectors, we refer the reader to \cite{cluster4} where they were introduced and to \cite{NZ:tropicalduality,Najera:cvectors,ST:acyclic,Nagao10,Keller:derivedcluster} where they were studied.

With this terminology, for a quiver $R \in \Mut(\hat Q)$ the vertex $i \in Q_0$ is green if and only if the $i$-th column $\ccc_i(R)$ of the $\ccc$-matrix satisfies $\ccc_i(R) \ge 0$,  and the vertex is red if and only if $\ccc_i(R) \le 0.$
The \emph{sign-coherence for $\ccc$-vectors}, conjectured by Fomin and Zelevinsky in~\cite{cluster4} and established in \cite{DerksenWeymanZelevinsky10}, ensures that each $\ccc$-vector satisfies  either $\ccc_i(R) \ge 0$ or $\ccc_i(R) \le 0$ (this gives a proof of Proposition~\ref{theorem:greenorred}).
In this sense, the $\ccc$-vectors behave like root systems from Lie theory. In fact, it is shown in \cite{Najera:cvectors} that the $\ccc$-vectors are root systems when the quiver $Q$ is acyclic:

\begin{theorem}\label{realroot}{(\cite[Theorem 1]{Najera:cvectors})}\label{t:c-vector-and-root} Let $Q$ be an acyclic quiver. Then for each $R \in \Mut(\hat Q)$, the 	$\ccc$-vectors $\ccc_i(R)$ are real Schur roots for the quiver $Q$, and all real Schur roots occur as a $\ccc$-vector for some  $R \in \Mut(\hat Q)$.
\end{theorem}

Further, for an acyclic quiver, Speyer and Thomas give a combinatorial description of which  collections of roots form $\ccc$-matrices, see~\cite[Theorem 1.4]{ST:acyclic}.

The ordered exchange graph of $\ccc$-matrices, $\cmat(Q)$, of a cluster quiver $Q$ is given by all $\ccc$-matrices (up to column permutations) obtained through iterated mutations starting from the identity matrix $I_n$, with arrows corresponding to mutations at green vertices.
By definition,  $\cmat(Q)$ has $I_n$ as the unique source, and if it has a sink, it is necessarily the matrix $-I_n$.
If $R\to R'$ is an arrow in the ordered exchange graph $\mut(Q)$ corresponding to the mutation at $i$, then the $\ccc$-matrix $C'$ of $R'$ is obtained from the $\ccc$-matrix $C$ of $R$ by the following formula:
\[
c'_{jl}  =  \left\{ \begin{array}{ll} c_{jl} & \mbox{if } l=i \\ c_{jl} + [c_{ji}]_+ [b_{il}]_+ - [-c_{ji}]_+ [-b_{il}]_+ & \mbox{otherwise.} \end{array}  \right. 
\]
Here $B(R)=(b_{jl})_{1\leq j\leq 2n,1\leq l\leq n}$ is the matrix of $R$ and $[x]_+=\max(x,0)$ for an integer $x$.

\medskip
\noindent
{\bf Example.} Let $Q$ be the quiver of type ${\mathbb A}_2$ as in Section~\ref{ss:tilting-module}. The ordered exchange graph $\cmat(Q)$ is (where $\phi$ interchanges the indices $1$ and $2$)
\[\begin{xy} 0;<0.35pt,0pt>:<0pt,-0.35pt>::
(300,0) *+{\left[ \begin{array}{rr} 1 & 0 \\ 0 & 1 \end{array} \right]
} ="0",
(100,150) *+{\left[ \begin{array}{rr} -1 & 1 \\ 0 & 1 \end{array} \right]
}="1",
(100,350) *+{\left[ \begin{array}{rr} 0 & -1 \\ 1 & -1 \end{array} \right]
}="2",
(500,250) *+{\left[ \begin{array}{rr} 1 & 0 \\ 0 & -1 \end{array} \right]
}="3",
(300,500) *+{\left[ \begin{array}{rr} -1 & 0 \\ 0 & -1 \end{array} \right]
}="4",
"0", {\ar "1"}, "0", {\ar "3"},
"1", {\ar "2"},
"2", {\ar_\phi "4"},
"3", {\ar "4"},  
\end{xy}\]

\subsection{$\gg$-matrices, $\gmat(Q)$}\label{ss:gmatrix}

Let $Q$ be a cluster quiver and $\ca_Q$ be the corresponding cluster algebra with principal coefficients defined in Section~\ref{ss:cluster}. 

The ring $\mathbb{Z}[x_1^{\pm 1},\ldots,x_n^{\pm 1},x_{n+1},\ldots,x_{2n}]$ is $\mathbb{Z}^n$-graded with 
\[\deg(x_i)=e_i, \qquad \deg(x_{n+i})=-B(Q)e_i,\]
where $e_1,\ldots,e_n$ is the standard basis of $\mathbb{Z}^n$ and $B(Q)$ is the matrix of $Q$. Recall that the Laurent phenomenon states that each cluster variable is an element of the ring $\mathbb{Z}[x_1^{\pm 1},\ldots,x_n^{\pm 1},x_{n+1},\ldots,x_{2n}]$. Moreover, by \cite[Proposition 6.1]{cluster4}, each non-frozen cluster variable of $\ca_Q$ is homogeneous, and its degree is by definition its \emph{$\gg$-vector}. In this way, with each cluster we can associate the matrix of g-vectors of its non-frozen cluster variables, called its \emph{$\gg$-matrix}. Note that $\gg$-matrices are defined up to column permutations.

Assume that $Q$ is acyclic. For two vertices $i$ and $j$ of $Q$, let $p_{ji}$ be the number of paths from $i$ to $j$. Let $\alpha_1,\ldots,\alpha_n$ be the simple roots of the root system of $Q$. The following result is `dual' to Theorem~\ref{t:c-vector-and-root}.

\begin{theorem}{(\cite[Theorem 5.2]{Keller:derivedcluster})} The linear map taking $e_j$ to $\sum_{i=1}^np_{ji}\alpha_i$, $1\leq j\leq n$, is a bijection from the set of $\gg$-vectors of $Q$ minus the set $\{-e_1,\ldots,-e_n\}$ to the set of positive real  Schur roots.
\end{theorem}

The $\gg$-matrix of the initial cluster $\{x_1,\ldots,x_n,x_{n+1},\ldots,x_{2n}\}$ is the identity matrix $I_n$. The exchange graph of $\gg$-matrices, $\gmat(Q)$, of a cluster quiver $Q$ consists of g-matrices of all clusters of $\ca_Q$ (up to column permutations). It has a unique source $I_n$ and its arrows correspond to mutations. Suppose that $\underline{u}$ and $\underline{u}'$ are clusters of $\ca_Q$ related by a mutation at $i$. Recall that $\underline{u}$ determines a seed $(\underline{u},R)$.The $\gg$-matrices $G=(g_1,\ldots,g_n)$ of $\underline{u}$ and $G'=(g'_1,\ldots,g'_n)$ are related by the following formula (\cite[Proposition 6.6]{cluster4})
\[
g'_j=\left\{\begin{array}{ll} g_j & \text{ if } j\neq i\\
-g_j+\sum_{m=1}^n[b_{mi}]_+g_m-\sum_{l=1}^n[-b_{n+l,i}]_+B(Q)e_l &\text{ if } j=i,
\end{array}\right.
\]
where $(b_{ml})_{1\leq m\leq 2n,1\leq l\leq n}$ is the matrix of $R$.

\medskip
\noindent
{\bf Example.} Let $Q$ be the quiver of type ${\mathbb A}_2$ as in Section~\ref{ss:tilting-module}. The ordered exchange graph $\gmat(Q)$ is (where $\phi$ interchanges the indices $1$ and $2$)
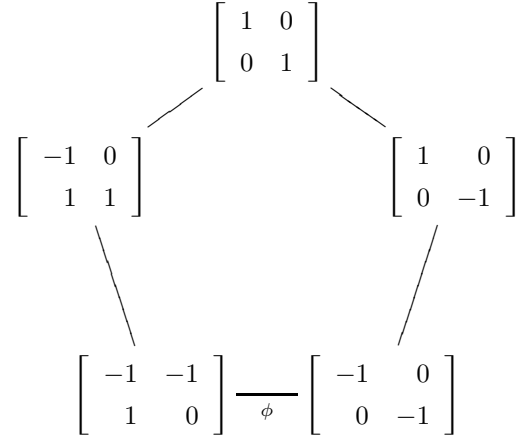
\[\begin{xy} 0;<0.35pt,0pt>:<0pt,-0.35pt>::
(200,0) *+{\left[ \begin{array}{rr} 1 & 0 \\ 0 & 1 \end{array} \right]} ="0",
(0,145.3) *+{\left[ \begin{array}{rr} -1 & 0 \\ 1 & 1 \end{array} \right]}="1",
(76.3,380.5) *+{\left[ \begin{array}{rr} -1 & -1 \\ 1 & 0 \end{array} \right]}="2",
(400,145.3) *+{\left[ \begin{array}{rr} 1 & 0 \\ 0 & -1 \end{array} \right]}="3",
(323.7,380.5) *+{\left[ \begin{array}{rr} -1 & 0 \\ 0 & -1 \end{array} \right]}="4",
"0", {\ar@{-} "1"}, "0", {\ar@{-} "3"},
"1", {\ar@{-} "2"},
"2", {\ar_{\phi}@{-} "4"},
"3", {\ar@{-} "4"},  
\end{xy}\]

\section{The bijections}\label{s:the-bijection}

In this section we define the bijections appearing in Figure~\ref{f:the-map}. We split the diagram into smaller ones.

\subsection{The correspondences between silting objects, simple-minded collections, $t$-structures and co-$t$-structures}
\label{ss:silting-smc-tstr-cotstr}
For a non-positive dg algebra $A$, set
\begin{align*}
\silt(A) &= \text{the set of isoclasses of basic silting objects of $\per(A)$},\\
\smc(A) &= \text{the set of isoclasses of simple-minded collections of $\cd_{fd}(A)$},\\
\tstr(A) &= \text{the set of bounded $t$-structures on $\cd_{fd}(A)$ with length heart},\\
\cotstr(A) &= \text{the set of bounded co-$t$-structures on $\per(A)$}.
\end{align*}

Let $(Q,W)$ be a quiver with potential and let $\Gamma=\hat{\Gamma}(Q,W)$ and $J=\hat{J}(Q,W)$ be respectively the corresponding Ginzburg dg algebra and the Jacobian algebra (see Section~\ref{ss:ginzburg-algebra}). We assume that $(Q,W)$ is Jacobi-finite, \ie $J$ is finite-dimensional. It follows  that $\per(\Gamma)$ and $\cd_{fd}(\Gamma)$ are Hom-finite and Krull--Schmidt, see~\cite[Proposition 2.4]{Amiot09} and~\cite[Lemma 2.17]{KellerYang11}. Moreover, we have $\per(J)\cong \ch^b(\proj J)$ and $\cd_{fd}(J)\cong\cd^b(\mod J)$. 
Recall that $\per(A)$ has a canonical silting object $A_A$, which has a standard decomposition $A=e_1A\oplus\ldots \oplus e_nA$ into a direct sum of indecomposable objects, where $\{1,\ldots,n\}=Q_0$ is the vertex set of $Q$ and $e_1,\ldots,e_n$ are the trivial paths. By Theorem~\ref{t:generators-of-grothendieckgroup-silting-case}, any silting object $M$ of $\per(A)$ satisfies the property $|M|=n$. Similarly, any simple-minded collection of $\cd_{fd}(A)$ has cardinality $n$.

Below we say that a map commutes with mutations (respecitvely, preserves partial orders) if it commutes with existing mutations (respectively, preserves existing partial orders).

\begin{theorem}\label{t:knky}{(\cite{KoenigYang12,KellerNicolas11})} Let $A=\Gamma$ or $J$. Then there is a commutative diagram of bijections which commute with mutations and preserve partial orders
\[
\xymatrix{
\silt(A)\ar[rr]\ar[d]\ar[drr]&&\cotstr(A)\ar[ll]\ar[dll]\\\
\tstr(A)\ar[rr]\ar[urr]&&\smc(A)\ar[ll]\ar[u]
}
\]
\end{theorem}


This theorem holds more generally for homologically smooth non-positive dg algebras (\cite{KellerNicolas11}) and for finite-dimensional algebras (\cite[Theorems 6.1, 7.12 and 7.13]{KoenigYang12}). The study of this diagram was initiated by Keller and Vossieck~\cite{KellerVossieck88}. They showed that the left vertical map is bijective for the case when $A$ is the path algebra of a Dynkin quiver. Parts of this diagram have also appeared in~\cite{KellerNicolas12,RickardRouquier10}.

\begin{remark}
\begin{itemize}
\item[(a)] The naturally defined mutations on the sets $\silt(A)$, $\smc(A)$ and $\tstr(A)$ induce a mutation operation on $\cotstr(A)$.  Let $(\cp_{\geq 0},\cp_{\leq 0})$ be a bounded co-$t$-structure on $\per(A)$ and let $M=M_1\oplus\ldots\oplus M_n$ be the corresponding basic silting object, where $M_1,\ldots,M_n$ are indecomposable. The \emph{left mutation} $\mu_i^-(\cp_{\geq 0},\cp_{\leq 0})$ of $(\cp_{\geq 0},\cp_{\leq 0})$ is the pair $(\cp'_{\geq 0},\cp'_{\leq 0})$, where
$\cp'_{\leq 0}$ is the smallest full subcategory of $\per(A)$ which contains $\Sigma\cp_{\leq 0}$ and $M_j$ and $\cp'_{\geq 0}$ is the left orthogonal category of $\Sigma\cp'_{\leq 0}$ in $\per(A)$.

\item[(b)]
The naturally defined  partial orders on $\tstr(A)$ and $\cotstr(A)$ induce partial orders on $\silt(A)$ and $\smc(A)$. Precisely,
for two silting objects $M$ and $M'$ of $\per(A)$, we have
\[M\geq M'\Leftrightarrow \Hom_{\cc}(M,\Sigma^i M')=0 \text{ for all } i>0,\]
see \cite[Definition 2.10]{AiharaIyama12}; 
for two simple-minded collections $\{X_1,\ldots,X_n\}$ and $\{X'_1,\ldots,X'_n\}$ in $\cd_{fd}(A)$, we have
\[\{X_1,\ldots,X_n\}\geq \{X'_1,\ldots,X'_n\}\Leftrightarrow \Hom_{\cc}(X_i,\Sigma^m X'_j)=0 \text{ for all }i,j\text{ and } m<0,\]
see~\cite[Section 7]{KoenigYang12}.
\end{itemize}
\end{remark}

Let us define the bijections in Theorem~\ref{t:knky} as in \cite{KellerNicolas11,KoenigYang12}. 

$\silt(A)\longrightarrow\smc(A)$: Let $M=M_1\oplus\ldots\oplus M_n$ be a basic silting object in $\per(A)$, where $M_1,\ldots,M_n$ are indecomposable. The corresponding simple minded collection $\{X_1,\ldots,X_n\}$ is the unique collection (up to isomorphism) in $\cd_{fd}(A)$ satisfying for any $i,j=1,\ldots,n$
\begin{align}
\Hom(M_i,\Sigma^m X_j)=\begin{cases} k & \text{ if } i=j \text{ and } m=0,\\ 0 & \text{ otherwise}.\end{cases} \label{eq:silting-smc-are-dual}
\end{align}
More precisely, it follows from Keller's Morita theorem for triangulated categories (\cite[Theorem 4.3]{Keller94}) that there is a non-positive dg algebra $B$ together with a triangle equivalence $\cd(B)\rightarrow \cd(A)$ taking $B$ to $M$. The collection $\{X_1,\ldots,X_n\}$ is the image of a complete collection of pairwise non-isomorphic simple $H^0(B)$-modules under this equivalence.

$\silt(A)\longrightarrow\tstr(A)$: Let $M$ be a silting object in $\per(A)$. The corresponding $t$-structure $(\cd^{\leq 0},\cd^{\geq 0})$ on $\cd_{fd}(A)$ is defined as
\begin{align*}
\cd^{\leq 0}&=\{X\in\cd_{fd}(A)\mid \Hom(M,\Sigma^m X)=0 \text{ for } m>0\},\\
\cd^{\geq 0}&=\{X\in\cd_{fd}(A)\mid\Hom(M,\Sigma^m X)=0\text{ for } m<0\}.
\end{align*}
The object $M$ can be characterised as an additive generator of the category
\[\{N\in\per(A)\mid \Hom(N,\Sigma X)=0 \text{ for any }X\in\cd^{\leq 0}\}.\]
A nice consequence of the bijectivity of this map is that the heart of any element in $\tstr(A)$ is equivalent to the category of finite-dimensional modules over a finite-dimensional algebra, namely, the endomorphism algebra of the corresponding silting object. 

$\silt(A)\longrightarrow\cotstr(A)$: Let $M$ be a silting object in $\per(A)$. The corresponding co-$t$-structure $(\cp_{\geq 0},\cp_{\leq 0})$ on $\per(A)$ is defined as
\begin{align*}
\cp_{\geq 0} & = \text{the smallest full subcategory of $\per(A)$ which contains $\{\Sigma^m M\mid m\leq 0\}$}\\
& \text{and which is closed under taking extensions and direct summands},\\
\cp_{\leq 0} & = \text{the smallest full subcategory of $\per(A)$ which contains $\{\Sigma^m M\mid m\geq 0\}$}\\
& \text{and which is closed under taking extensions and direct summands}.
\end{align*}
We point out that this map is well-defined and bijective in a much more general setting, see~\cite{Bondarko10,AiharaIyama12,KellerNicolas11,MendozaSaenzSantiagoSouto10}.

$\tstr(A)\longrightarrow\smc(A)$: Let $(\cd^{\leq 0},\cd^{\geq 0})$ be a bounded $t$-structure on $\cd_{fd}(A)$ with length heart. The corresponding simple-minded collection is a complete collection of pairwise non-isomorphic simple objects of the heart $\cd^{\leq 0}\cap\cd^{\geq 0}$.

$\tstr(A)\longrightarrow\cotstr(A)$: Let $(\cd^{\leq 0},\cd^{\geq 0})$ be a bounded $t$-structure on $\cd_{fd}(A)$ with length heart. The corresponding co-$t$-structure $(\cp_{\geq 0},\cp_{\leq 0})$ is defined as
\begin{align*}
\cp_{\geq 0}&=\{N\in\per(A)\mid \Hom(N,\Sigma X)=0 \text{ for any }X\in\cd^{\leq 0}\},\\
\cp_{\leq 0}&=\{N\in\per(A)\mid \Hom(N,\Sigma^{-1} X)=0 \text{ for any }X\in\cd^{\geq 0}\}.
\end{align*}


$\smc(A)\longrightarrow\tstr(A)$: Let $\{X_1,\ldots,X_n\}$ be a simple-minded collection of $\cd_{fd}(A)$. The corresponding $t$-structure $(\cd^{\leq 0},\cd^{\geq 0})$ on $\cd_{fd}(A)$ is defined as
\begin{align*}
\cd^{\leq 0} & = \text{the smallest full subcategory of $\cd_{fd}(A)$ which contains $\{\Sigma^m X_i\mid 1\leq i\leq n,$}\\ 
&\text{$~m\geq 0\}$ and which is closed under taking extensions and direct summands},\\
\cd^{\geq 0} & = \text{the smallest full subcategory of $\cd_{fd}(A)$ which contains $\{\Sigma^m X_i\mid 1\leq i\leq n,$}\\
& \text{$~m\leq 0\}$ 
and which is closed under taking extensions and direct summands}.
\end{align*}

$\smc(A)\longrightarrow\cotstr(A)$: Let $\{X_1,\ldots,X_n\}$ be a simple-minded collection of $\cd_{fd}(A)$. The corresponding co-$t$-structure $(\cp_{\geq 0},\cp_{\leq 0})$ on $\per(A)$ is defined as
\begin{align*}
\cp_{\geq 0} & = \{N\in\per(A)\mid \Hom(N,\Sigma^m\bigoplus_{i=1}^n X_i)=0 \text{ for } m>0\},\\
\cp_{\leq 0} & = \{N\in\per(A)\mid \Hom(N,\Sigma^m\bigoplus_{i=1}^n X_i)=0 \text{ for } m<0\}.
\end{align*}
This map was not directly defined in \cite{KoenigYang12}. We take it as the composition of $\smc(A)\rightarrow\tstr(A)$ and $\tstr(A)\rightarrow\cotstr(A)$.


$\cotstr(A)\longrightarrow\silt(A)$: Let $(\cp_{\geq 0},\cp_{\leq 0})$ be a bounded co-$t$-structure on $\per(A)$. The corresponding silting object $M$ of $\per(A)$ is an additive generator of the coheart, \ie $\add(M)=\cp_{\geq 0}\cap\cp_{\leq 0}$.

$\cotstr(A)\longrightarrow\tstr(A)$: Let $(\cp_{\geq 0},\cp_{\leq 0})$ be a bounded co-$t$-structure on $\per(A)$. The corresponding $t$-structure $(\cd^{\leq 0},\cd^{\geq 0})$ is defined as
\begin{align*}
\cd^{\leq 0}&=\{X\in\cd_{fd}(A)\mid \Hom(N,\Sigma X)=0 \text{ for any } N\in\cp_{\geq 0}\},\\
\cd^{\geq 0}&=\{X\in\cd_{fd}(A)\mid\Hom(N,\Sigma^{-1} X)=0\text{ for any } N\in\cp_{\leq 0}\}.
\end{align*}

\medskip

Next, consider the subsets $\twosilt(A)$, $\intersmc(A)$, $\intertstr(A)$ and $\intercotstr(A)$ of $\silt(A)$, $\smc(A)$, $\tstr(A)$ and $\cotstr(A)$.
Set
\begin{align*}
\cf_A&=\{\cone(f)\mid f \text{ is a morphism in }\add_{\per(A)}(A)\}\subseteq \per(A).
\end{align*}
By definition a $2$-term silting object of $\per(A)$ is exactly a silting object belonging to $\cf_A$. 

\begin{corollary}\label{c:knky-for-intermediate}
There is a commutative diagram of bijections which commute with mutations and which preserve partial orders
\[
\xymatrix{
\twosilt(A)\ar[rr]\ar[d]\ar[drr]&&\intercotstr(A)\ar[ll]\ar[dll]\\\
\intertstr(A)\ar[rr]\ar[urr]&&\intersmc(A)\ar[ll]\ar[u]
}
\]
\end{corollary}
\begin{proof} It suffices to check that the following six maps restrict.

$\cotstr(A)\longleftrightarrow \tstr(A)$: These two maps are defined by taking shifts of orthogonal subcategories. It is easy to check that they restrict to intermediate objects.

$\silt(A)\longleftrightarrow \cotstr(A)$: Let $(\cp^\std_{\geq 0},\cp^\std_{\leq 0})$ denote the  standard co-$t$-structure on $\per(A)$. One checks that $\cf_A=\Sigma \cp^\std_{\geq 0}\cap\cp^\std_{\leq 0}$ holds. By Lemma~\ref{l:intermediate-co-t-str}, a co-$t$-structure is intermediate if and only if its co-heart is contained in $\cf_A$, namely, the corresponding silting object is $2$-term. Thus the two maps $\silt(A)\leftrightarrow \cotstr(A)$ restrict to bijections $\twosilt(A)\leftrightarrow\intercotstr(A)$.

$\smc(A)\longleftrightarrow\tstr(A)$: Dual to $\silt(A)\leftrightarrow\cotstr(A)$.
\end{proof}

The proof of Corollary~\ref{c:knky-for-intermediate} works more generally for $A$ being a homologically smooth non-positive dg algebra or a finite-dimensional algebra.

\begin{theorem}\label{t:the-top-two-floor} There is a commutative diagram of bijections which commute with mutations and preserve partial orders
\[\xymatrix@C=0.45pc@R=1.5pc{
\twosilt(\Gamma)\ar[rrr]\ar[dr] \ar[ddd]\ar[drrrr]& &&
\intercotstr(\Gamma)\ar[lll]\ar@{-->}[ddd]
\\
&\intertstr(\Gamma)\ar[rrr]
&&&
\intersmc(\Gamma)\ar[lll]\ar[ul]
\\
\\
\twosilt(J)\ar@{-->}[rrr]\ar[dr] \ar@{-->}[drrrr]& &&
\intercotstr(J)\ar@{-->}[lll]
\\
&\intertstr(J)\ar[rrr]\ar[uuu]&&&
\intersmc(J)\ar[lll]\ar@{-->}[ul]\ar[uuu]
}\]
\end{theorem}
The following proof of this theorem still works if the pair $(\Gamma,J)$ is replaced by $(A,H^0(A))$, where $A$ is a homologically smooth non-positive dg algebra such that $H^0(A)$ is finite-dimensional.

\begin{proof} 
The top square (respectively, the bottom square) is obtained by applying Corollary~\ref{c:knky-for-intermediate} to $\Gamma$ (respectively, $J$).

Next, we will define the vertical maps. We state the bijectivity and compatibility with mutations for some of them: that for others follow from the commutativity of the diagram.

Consider the canonical projection $p:\Gamma\rightarrow J$.
Let $p_*:\per(\Gamma)\rightarrow\per(J)$ (respectively, $p^*:\cd_{fd}(J)\rightarrow\cd_{fd}(\Gamma)$) be the induction (respectively, restriction) along the projection $p$. The functor $p^*$ is extensively studied by King and Qiu in~\cite{KingQiu11} for the case of a Dynkin quiver with trivial potential.

$\twosilt(\Gamma)\longrightarrow\twosilt(J)$: The assignment $M\mapsto p_*(M)$ defines a bijection from $\twosilt(\Gamma)$ to $\twosilt(J)$ which commutes with mutations, by applying Propositions~\ref{p:induction-silting} and~\ref{p:induction-silting-mutation} to the dg algebra $A=\Gamma$. More explicitly, observe that $p_*$ induces an additive equivalence $\add_{\per(\Gamma)}(\Gamma)\rightarrow\add_{\per(J)}(J)=\proj J$. Since $M$ belongs to $\twosilt(\Gamma)$, there is a triangle
\[\xymatrix{P\ar[r]^f & Q\ar[r] & M\ar[r] & \Sigma P}\]
with $P,Q\in\add_{\per(\Gamma)}(\Gamma)$. Then $p_*(M)$ is the cone of the morphism $p_*(f):p_*(P)\rightarrow p_*(Q)$, which is a morphism $\proj J$.

$\intercotstr(\Gamma)\longrightarrow\intercotstr(J)$: Consider the assignment induced by the projection: $(\cp_{\geq 0},\cp_{\leq 0})\mapsto (H_{\geq 0},H_{\leq 0})$, where
\begin{align*}
H_{\geq 0} & = \text{ the smallest full subcategory of $\per(J)$ which contains $p_*(\cp_{\geq 0})$}\\
&\text{and which is closed under extensions and direct summands},\\
H_{\leq 0} & = \text{ the smallest full subcategory of $\per(J)$ which contains $p_*(\cp_{\leq 0})$}\\
&\text{and which is closed under extensions and direct summands},
\end{align*}
which clearly preserves partial orders.

Let $(\cp_{\geq 0},\cp_{\leq 0})$ be an intermediate co-$t$-structure on $\per(\Gamma)$ and let $M$ be the corresponding silting object. Let $(H_{\geq 0},H_{\leq 0})$ be defined as above. Let $(H'_{\geq 0},H'_{\leq 0})$ be the co-$t$-structure on $\per(J)$ corresponding to the silting object $p_*(M)$. Then $H_{\geq 0}\supseteq H'_{\geq 0}$ and $H_{\leq 0}\supseteq H_{\leq 0}'$. It follows that both equalities hold, $(H_{\geq 0},H_{\leq 0})$ is a bounded co-$t$-structure on $\per(J)$ and it corresponds to $p_*(M)$. This shows the well-definedness of the map $\intercotstr(\Gamma)\rightarrow\intercotstr(J)$ as well as the commutativity of the inner square of the diagram.

$\intersmc(J)\longrightarrow\intersmc(\Gamma)$: Consider the assignment induced by the projection: $\{X_1,\ldots,X_n\}\mapsto\{p^*(X_1),\ldots,p^*(X_n)\}$.

Let $\{X_1,\ldots,X_n\}$ be a $2$-term simple-minded collection in $\cd^b(\mod J)$. We want to show that $\{p^*(X_1),\ldots,p^*(X_n)\}$ is a $2$-term simple-minded collection in $\cd_{fd}(\Gamma)$. 

Let $M=M_1\oplus\ldots\oplus M_n$ be the basic $2$-term silting object corresponding to $\{X_1,\ldots,X_n\}$, where $M_1,\ldots,M_n$ are indecomposable. Then as we have shown, there is a $2$-term basic silting object $N=N_1\oplus\ldots\oplus N_n$ such that $p_*(N_i)=M_i$ for all $i$. Then
\begin{align*}
\Hom(N_i,\Sigma^m p^*(X_j))&=\Hom(p_*(N_i),\Sigma^m X_j)\\
&=\Hom(M_i,\Sigma^m X_j)\\
&=\begin{cases} k & \text{ if } i=j \text{ and } m=0,\\ 0 & \text{ otherwise}.\end{cases}
\end{align*}
It follows by (\ref{eq:silting-smc-are-dual}) that $\{p^*(X_1),\ldots,p^*(X_n)\}$ is the simple-minded collection corresponding to the silting object $N$. This shows the well-definedness of the map $\intersmc(J)\rightarrow\intersmc(\Gamma)$ as well as the commutativity of the diagonal square of the diagram.

$\intertstr(J)\longrightarrow\intertstr(\Gamma)$: The assignment $(\cc^{\leq 0},\cc^{\geq 0})\mapsto (\cd^{\leq 0},\cd^{\geq 0})$,
where 
\begin{align*}
\cd^{\leq 0} & = \text{ the smallest full subcategory of $\cd_{fd}(J)$ which contains $p^*(\cc^{\leq 0})$}\\
&\text{and which is closed under extensions and direct summands},\\
\cd^{\geq 0} & = \text{ the smallest full subcategory of $\cd_{fd}(J)$ which contains $p^*(\cc^{\geq 0})$}\\
&\text{and which is closed under extensions and direct summands},
\end{align*}
defines a bijection $\intertstr(J)\longrightarrow\intertstr(\Gamma)$. The proof for the well-definedness of this map and the commutativity of the right square of the diagram is similar to the case for $\intercotstr(\Gamma)\rightarrow\intercotstr(J)$.
\end{proof}

\subsection{The Amiot cluster category}\label{ss:amiot-cluster-category}

Let $(Q,W)$ be a Jacobi-finite quiver with potential, $\Gamma=\hat{\Gamma}(Q,W)$ and $J=\hat{J}(Q,W)$. Recall from Section~\ref{ss:ginzburg-algebra} that $\per(\Gamma)$ contains $\cd_{fd}(\Gamma)$ as a triangulated subcategory. The \emph{Amiot cluster category} of $(Q,W)$ is defined as the triangle quotient 
\[\cc_{(Q,W)}:=\per(\Gamma)/\cd_{fd}(\Gamma).\]
Let $\pi:\per(\Gamma)\rightarrow\cc_{(Q,W)}$ be the canonical projection functor.

\begin{theorem}\label{t:generalised-cluster-category}
The following statements hold
\begin{itemize}
\item[(a)] {{(\cite[Theorem 3.5]{Amiot09})}} The Amiot cluster category $\cc_{(Q,W)}$
is Hom-finite and 2-Calabi--Yau. Moreover, the object $\pi(\Gamma)$ is a cluster-tilting object object in $\cc_{(Q,W)}$. Its endomorphism algebra is isomorphic to $J$.
\item[(b)] {{(\cite[Proposition 2.9]{Amiot09})}} The functor $\pi:\per(\Gamma)\rightarrow\cc_{(Q,W)}$ induces an additive equivalence
\[\cf_\Gamma\stackrel{\sim}{\longrightarrow}\cc_{(Q,W)}.\]
\item[(c)] {{(\cite[Proposition 2.12]{Amiot09})}} For $X,Y\in\cf_\Gamma$, there is a short exact sequence
\[\hspace{-7pt}\xymatrix@C=0.5pc{0\ar[r] &\Hom_{\per(\Gamma)}(X,\Sigma Y)\ar[r] & \Hom_{\cc_{(Q,W)}}(\pi(X),\Sigma \pi(Y))\ar[r] & D\Hom_{\per(\Gamma)}(Y,\Sigma X)\ar[r] &0}.\]
\end{itemize}
\end{theorem}

These results were proved by Amiot~\cite{Amiot09} for non-complete Ginzburg dg algebras and non-complete Jacobian algebras, but her proof works also in the complete setting. When $Q$ is acyclic, the Amiot cluster category $\cc_{(Q,0)}$ is triangle equivalent to the cluster category $\cc_Q$, thanks to the theorem ~\cite[Theorem 2.1]{KellerReiten08} of Keller and Reiten.

\subsection{The correspondences between silting objects, support $\tau$-tilting objects and cluster-tilting objects}\label{ss:silting-sttilting-cto}
Let $(Q,W)$ be a Jacobi-finite quiver with potential, $\Gamma=\hat{\Gamma}(Q,W)$, $J=\hat{J}(Q,W)$ and $p:\Gamma\rightarrow J$ be the canonical projection of dg algebras. Let $\cc_{(Q,W)}$ be the Amiot cluster category of $(Q,W)$ and $\pi:\per(\Gamma)\rightarrow\cc_{(Q,W)}$ be the canonical projection functor. Recall that
$\sttilt(J)$  denotes the set of isomorphism classes of basic support $\tau$-tilting modules over $J$ (Section~\ref{ss:support-tau-tilting}) and
$\cto(\cc_{(Q,W)})$ denotes the set of isomorphism classes of basic cluster-tilting objects in $\cc_{(Q,W)}$ (Section~\ref{ss:cluster-tilting}).

\begin{theorem}\label{t:silt-support-cluster}
There is a commutative diagram of bijections which commute with mutations
{\rm
\[\xymatrix@C=0.45pc@R=1.5pc{
\twosilt(\Gamma)\ar[rr] \ar[ddr]\ar@/_20pt/[rdddd]&&
\twosilt(J)\ar[ldd]
\\
\\
&\sttilt(J)
\\
\\
&\cto(\cc_{(Q,W)})\ar[uu]\ar@/_20pt/[ruuuu]
}
\]
}
\end{theorem}

The map $\twosilt(\Gamma)\rightarrow\twosilt(J)$ was defined in Theorem~\ref{t:the-top-two-floor}.

$\twosilt(\Gamma)\longrightarrow\sttilt(J)$: This map is defined as taking the $0$-th cohomology: $M\mapsto H^0(M)=\Hom_{\per(\Gamma)}(\Gamma,M)$, see~\cite{IyamaJorgensenYang14}.

$\twosilt(J)\longrightarrow\sttilt(J)$: This map is defined as taking the $0$-th cohomology: $M\mapsto H^0(M)=\Hom_{\ch^b(\proj J)}(J,M)$. For the well-definedness, bijectivity and compatibility of mutations of this map, see
\cite[Theorem 3.2 and Corollary 3.9]{AdachiIyamaReiten12}.
The commutativity of the top triangle follows from the fact that the equivalence $p_*:\add_{\per(\Gamma)}(\Gamma)\rightarrow\proj J$ is just $H^0$. 

$\cto(\cc_{(Q,W)})\longrightarrow\sttilt(J)$: Recall from Theorem~\ref{t:generalised-cluster-category} (a) that $\pi(\Gamma)$ is a cluster-tilting object of $\cc_{(Q,W)}$ with endomorphism algebra $J$. So there is a functor $\Hom_{\cc_{(Q,W)}}(\pi(\Gamma),?):\cc_{(Q,W)}\rightarrow \mod J$. The desired map takes a cluster-tilting object $M$ to $\Hom_{\cc_{(Q,W)}}(\pi(\Gamma),M)$, see~\cite[Theorem 4.1]{AdachiIyamaReiten12}.

$\twosilt(\Gamma)\longrightarrow\cto(\cc_{(Q,W)})$: The assignment $M\mapsto \pi(M)$ defines a bijection $\twosilt(\Gamma)\rightarrow\cto(\cc_{(Q,W)})$, which is known to the experts. We give a proof for the convenience of the reader. Recall that a silting object of $\per(\Gamma)$ is $2$-term if and only if it belongs to $\cf_\Gamma$.

\begin{proposition}
The equivalence $\pi:\cf_{\Gamma}\stackrel{\sim}{\longrightarrow}\cc_{(Q,W)}$ induces a bijection from $\twosilt(\Gamma)$ to $\cto(\cc_{(Q,W)})$. Moreover, this bijection commutes with mutations, that is, if $M$ is a basic 2-term silting object in $\per(\Gamma)$ and $N$ is an indecomposable direct summand of $M$ such that $\mu^-_N(M)$ is 2-term, then $\mu_{\pi(N)}(\pi(M))=\pi(\mu^-_N(M))$.

\end{proposition}
\begin{proof}
Let $X$ be an object of $\cf_\Gamma$. Looking at the short exact sequence in Theorem~\ref{t:generalised-cluster-category} (c) with $Y=X$, we obtain that $\Hom_{\cc_{(Q,W)}}(\pi(X),\Sigma \pi(X))=0$ if and only if $\Hom_{\per(\Gamma)}(X,\Sigma X)=0$. Thus we have
\begin{align*}
\lefteqn{\text{$\pi(X)$ is a cluster tilting object in $\cc_A$}}\\
&\Leftrightarrow \text{$\Hom_{\cc_{(Q,W)}}(\pi(X),\Sigma \pi(X))=0$ and $|\pi(X)|=|\pi(\Gamma)|$} \qquad \text{by Proposition~\ref{p:cto-and-number-of-summands}}\\
&\Leftrightarrow \text{$\Hom_{\per(\Gamma)}(X,\Sigma X)=0$ and $|X|=|\Gamma|$}\\
&\Leftrightarrow \text{$X$ is a presilting object and $|X|=|\Gamma|$}\qquad\text{by Lemma~\ref{l:2-term-partial-silting-and-vanishing-of-ext1}}\\
&\Leftrightarrow \text{$X$ is a silting object}  \qquad \text{by Proposition~\ref{l:partial-silting}}.
\end{align*}
This proves the first statement.

The second statement holds because $f$ is a minimal approximation in $\cf_\Gamma$ if and only if $\pi(f)$ is a minimal approximation in $\cc_{(Q,W)}$.  
\end{proof}

The commutativity of the left triangle of the diagram follows from the definitions of the three maps and the fact that $\pi:\cf_\Gamma\rightarrow\cc_{(Q,W)}$ is an equivalence.

$\cto(\cc_{(Q,W)})\longrightarrow\twosilt(J)$: The functor $\Hom_{\cc_{(Q,W)}}(\pi(\Gamma),?):\cc_{(Q,W)}\rightarrow \mod J$ restricts to an additive equivalence $\add_{\cc_{(Q,W)}}(\pi(\Gamma))\rightarrow \proj J$. Let $M$ be a cluster-tilting object of $\cc_{(Q,W)}$. Then there is a triangle
\begin{align}
\xymatrix{Q^{-1}\ar[r]^f & Q^0\ar[r] & M\ar[r] &\Sigma Q^{-1}}\label{eq:triangle-cto->silting}
\end{align}
with $Q^{-1}$ and $Q^0$ in $\add_{\cc_{(Q,W)}}(\pi(\Gamma))$. The desired map takes $M$ to the cone of the morphism $\Hom_{\cc_{(Q,W)}}(\pi(\Gamma),f)$ in $\ch^b(\proj J)$, see~\cite[Theorem 4.7]{AdachiIyamaReiten12}. By applying $\Hom_{\cc_{(Q,W)}}(\pi(\Gamma),?)$ to the triangle (\ref{eq:triangle-cto->silting}), we see that $\Hom_{\cc_{(Q,W)}}(\pi(\Gamma),M)$ is the cokernel of $\Hom_{\cc_{(Q,W)}}(\pi(\Gamma),f)$, \ie the $0$-th cohomology of its cone.  The commutativity of the right triangle of the diagram follows.

\begin{remark}
Reduction techniques analogous to $2$-Calabi--Yau reduction (Section~\ref{ss:cluster-tilting}), \emph{silting reduction} and \emph{$\tau$-tilting reduction}, are respectively introduced by Aihara and Iyama in \cite{AiharaIyama12} and by Jasso in \cite{Jasso13}. The compatibility of these reductions are studied in \cite{Keller11,IyamaYang13,Jasso13}.
\end{remark}


\subsection{The correspondences between silting objects, $t$-structures, torsion pairs and support $\tau$-tilting modules}

Let $(Q,W)$ be a Jacobi-finite quiver with potential, $\Gamma=\hat{\Gamma}(Q,W)$ and $J=\hat{J}(Q,W)$.

\begin{theorem}\label{t:torsion-t-str}
There is a commutative diagram of bijections which commute with mutations and preserve partial orders
{\rm
\[\xymatrix@C=2pc@R=1.5pc{
\twosilt(\Gamma)\ar[rr] \ar@{-->}[rdd]\ar[dddd]
& &
\intertstr(\Gamma)
\ar[ll]
\\
\\
&\twosilt(J)\ar@{-->}[rr]\ar@{-->}[ldd] & &
\intertstr(J)\ar@{-->}[ll]\ar[uul]
\\
\\
\sttilt(J)\ar[rr] && \fftor(J) \ar[uur]\ar[uuuu]
}
\]
}
\end{theorem}

The maps in the upper square were defined in Theorem~\ref{t:the-top-two-floor} and the maps in the left triangle were defined in Theorem~\ref{t:silt-support-cluster}.

$\sttilt(J)\longrightarrow\fftor(J)$: 
The assignment $T\mapsto \Fac(T)$ defines a bijection from $\sttilt(J)$ to $\fftor(J)$, see~\cite[Theorem 2.6]{AdachiIyamaReiten12}. The module $T$ can be characterised as an additive generator of the subcategory
\[\{N\in\mod J\mid \Ext^1_J(N,M)=0\text{ for all } M\in\Fac(T)\}.\]

$\fftor(J)\longrightarrow \intertstr(J)$: This map is defined as the Happel--Reiten--Smal{\o} tilt (see Section~\ref{ss:t-str}). 
There is the following general result.
\begin{theorem}\label{t:torsion-pair-and-int-t-str}{(\cite[Theorem 3.1]{BeligiannisReiten07} and \cite[Proposition 2.1]{Woolf10})} Let $(\cc^{\leq 0},\cc^{\geq 0})$ be a bounded $t$-structure on $\cc$ with heart $\ca$. The Happel--Reiten--Smal{\o} tilt induces a bijective map from the set of torsion pairs of $\ca$ to intermediate $t$-structures with respect to $(\cc^{\leq 0},\cc^{\geq 0})$. Its inverse takes $(\cc'^{\leq 0},\cc'^{\geq 0})$ to $(\cc'^{\leq 0}\cap\ca,(\Sigma^{-1}\cc'^{\geq 0})\cap\ca)$.
\end{theorem}

However, it is difficult to directly prove that this map is well-defined, namely, a functorially finite torsion class is taken to a bounded $t$-structure with length heart. We use the commutativity of the lower square of the diagram. To show the commutativity, let $M$ be a $2$-term silting object and let $(\cd^{\leq 0},\cd^{\geq 0})$ denote the corresponding $t$-structure on $\cd^b(\mod J)$. We need to prove the equality $\Fac(H^0(M))=\cd^{\leq 0}\cap \mod J$.

Let $X\in \Fac(H^0(M))$. Then there is a positive integer $m$ and a surjection $H^0(M)^{\oplus m}\rightarrow X$. Since $M$ is $2$-term, there is a morphism $M\rightarrow H^0(M)$ such that taking $H^0$ we obtain the identity map. Composing the above two maps, we obtain a map $f:M^{\oplus m}\rightarrow X$. Form the triangle
\begin{align}
\xymatrix{
M^{\oplus m}\ar[r]^f & X\ar[r] & N\ar[r] &\Sigma M^{\oplus m}
}\label{eq:triangle-torsion->t-str}
\end{align}
By taking cohomologies we see that $H^i(N)=0$ for $i\geq 0$, and hence $N\in\Sigma\cd_\std^{\leq 0}$, where $(\cd^{\leq 0}_\std,\cd^{\geq 0}_\std)$ denotes the standard $t$-structure on $\cd^b(\mod J)$. Since $(\cd^{\leq 0},\cd^{\geq 0})$ is intermediate, it follows that $N$ belongs to $\cd^{\leq 0}$.
Applying $\Hom(M,?)$ to the above triangle (\ref{eq:triangle-torsion->t-str}), we obtain a long exact sequence
\[
\xymatrix{
\ldots\ar[r]&\Hom(M,\Sigma^i M^{\oplus n})\ar[r] & \Hom(M,\Sigma^i X)\ar[r] & \Hom(M,\Sigma^i N)\ar[r]&\ldots
}
\]
The two outer terms vanish for $i>0$, so the middle term also vanishes for $i>0$, \ie $X\in\cd^{\leq 0}$. This shows the inclusion $\Fac(H^0(M))\subseteq \cd^{\leq 0}\cap \mod J$.

Let $X$ be a non-zero object in $\cd^{\leq 0}\cap \mod J$. We claim that $\Hom(M,X)\neq 0$. Indeed, if $\Hom(M,X)=0$, then $\Hom(M,\Sigma^i X)=0$ for $i\geq 0$, so $\Sigma^{-1}X\in\cd^{\leq 0}$. Thus  $X\in\Sigma \cd^{\leq 0}\subseteq \Sigma\cd_\std^{\leq 0}$ since $(\cd^{\leq 0},\cd^{\geq 0})$ is intermediate. This implies that $X=0$ because $\Sigma\cd_\std^{\leq 0}\cap \mod J=0$. Now take a basis $f_1,\ldots,f_m$ of $\Hom(M,X)$ and form the following triangle
\begin{align}
\xymatrix@C=2.5pc{
N\ar[r] & M^{\oplus m}\ar[r]^{(f_1,\ldots,f_m)} & X\ar[r] &\Sigma N
}. \label{eq:triangle-torsion->t-str-2}
\end{align}
Applying $\Hom(M,?)$ to this triangle we obtain a long exact sequence
\[
\xymatrix@C=0.9pc{
\Hom(M,\Sigma^i M^{\oplus m})\ar[r] & \Hom(M,\Sigma^i X)\ar[r] & \Hom(M,\Sigma^{i+1}N)\ar[r] & \Hom(M,\Sigma^{i+1}M^{\oplus m}).
}
\]
If $i>0$, then $\Hom(M,\Sigma^{i+1}N)$ vanishes because its two neighbours vanish. If $i=0$, then the leftmost map is clearly surjective, and hence $\Hom(M,\Sigma N)=0$. Therefore $N\in\cd^{\leq 0}\subseteq \cd_\std^{\leq 0}$, in particular, $H^1(N)=0$. Now taking cohomologies of the triangle (\ref{eq:triangle-torsion->t-str-2}) gives us an exact sequence
\[\xymatrix{H^0(M)^{\oplus n}\ar[r] & X\ar[r] & H^1(N)=0}.\]
So $X\in \Fac(H^0(M))$. This shows the inclusion $\cd^{\leq 0}\cap \mod J\subseteq \Fac(H^0(M))$.

$\fftor(J)\longrightarrow \intertstr(\Gamma)$: This map is defined as the Happel--Reiten--Smal{\o} tilt, similar to the map $\fftor(J)\rightarrow \intertstr(J)$.

\begin{remark}\label{r:smc-are-stalk}
Let $\{X_1,\ldots,X_n\}$ be a $2$-term simple-minded collection of $\cd_{fd}(\Gamma)$ (respectively, $\cd^b(\mod J)$). In Section~\ref{ss:silting-smc-tstr-cotstr} we constructed a bounded $t$-structure $(\cd^{\leq 0},\cd^{\geq 0})$ whose heart $\ca$ is a length category with simple objects $X_1,\ldots,X_n$. By Theorem~\ref{t:torsion-pair-and-int-t-str}, there is a torsion pair $(\ct,\cf)$ of $\mod J$ such that $(\cd^{\leq 0},\cd^{\geq 0})$ is the Happel--Reiten--Smal{\o} tilt of the standard $t$-structure at $(\ct,\cf)$. In particular, $\ca$ has a torsion pair $(\Sigma\cf,\ct)$. Therefore, for any $i=1,\ldots,n$, the object $X_i$ belongs to either $\Sigma\cf$ or $\ct$.
\end{remark}

\subsection{Ordered exchange graphs and reachable objects}\label{ss:reachable-objects}

Let $(Q,W)$ be a Jacobi-finite quiver with potential, $\Gamma=\hat{\Gamma}(Q,W)$ and $J=\hat{J}(Q,W)$.
Gluing the three diagrams in Theorems~\ref{t:the-top-two-floor}, \ref{t:silt-support-cluster} and~\ref{t:torsion-t-str} we obtain the diagram in Figure~\ref{f:diagram-cat}.

\begin{figure}
\[
\begin{xy} 0;<0.35pt,0pt>:<0pt,-0.35pt>::
(50,0) *+{\twosilt(\Gamma)}="0",
(500,0) *+{\intercotstr(\Gamma)}="1",
(180,100) *+{\intertstr(\Gamma)}="2",
(630,100) *+{\intersmc(\Gamma)}="3",
(50,250) *+{\twosilt(J)}="4",
(500,250) *+{\intercotstr(J)}="5",
(180,350) *+{\intertstr(J)}="6",
(630,350) *+{\intersmc(J)}="7",
(200,490) *+{\sttilt(J)}="8",
(500,490) *+{\fftor(J)}="9",
(400,630) *+{\cto(\cc_{(Q,W)})}="10",
(382,350) *+{}="14",
"0", {\ar "1"}, {\ar "2"}, {\ar "3"}, {\ar "4"}, {\ar@/_90pt/ "10"}, 
"1", {\ar "0"},  {\ar@{-->} "5"}, 
"2", {\ar "3"}, 
"3", {\ar "1"}, {\ar "2"},
"4", {\ar@{-->} "5"}, {\ar "6"}, {\ar@{-->} "7"}, {\ar@/_20pt/ "8"},
"5", {\ar@{-->} "4"},
"6",  {\ar "7"}, {\ar "2"}, 
"7", {\ar@{-->} "5"}, {\ar "6"}, {\ar "3"},
"8", {\ar "9"}, 
"9", {\ar "6"}, {\ar@{-} "14"}, "14", {\ar@{-->} "2"},
"10", {\ar "8"},
\end{xy}
\]
\caption{}
\label{f:diagram-cat}
\end{figure}

Via the bijections each of the sets in the diagram is equipped with a mutation operation and a partial order. Thanks to Theorem~\ref{t:support-tau-tilting-is-an-ordered-exchange-graph}, they all have the structure of ordered exchange graphs and in this way the diagram becomes a commutative diagram of isomorphisms of ordered exchange graphs.  

It is known that these graphs are connected when they have a finite connected component (Proposition~\ref{p:finiteness-and-connectedness-for-support-tau-tilting}) and when $Q$ is acyclic and $W=0$ (\cite[Proposition 3.5]{BuanMarshReinekeReitenTodorov06}). In general they are not connected, for example, when $Q$ is the quiver
$$
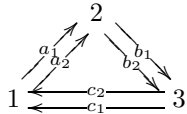
\xymatrix@=2em{
					& 2 \ar@<+3pt>[rd]|{b_1} \ar@<-3pt>[rd]|{b_2} \\
				1 \ar@<+3pt>[ru]|{a_1} \ar@<-3pt>[ru]|{a_2} && 3 \ar@<+3pt>[ll]|{c_1} \ar@<-3pt>[ll]|{c_2}
			}$$
and $W=c_1b_1a_1+c_2b_2a_2-c_1b_2a_1c_2b_1a_2$, see~\cite[Example 4.3]{Plamondon11c}. Demonet and Iyama informed us that in this example the graph has precisely two connected components.

The distinguished objects 
\[
\begin{array}{l}
\Gamma_\Gamma\in\twosilt(\Gamma),\\[1mm]
\{p^*(S_1),\ldots,p^*(S_n)\}\in\intersmc(\Gamma),\\[1mm]
(\cd_{\std,\Gamma}^{\leq 0},\cd_{\std,\Gamma}^{\geq 0})\in\intertstr(\Gamma),\\[1mm]
(\cp^{\std,\Gamma}_{\geq 0},\cp^{\std,\Gamma}_{\leq 0})\in\intercotstr(\Gamma),\\[1mm]
J_J\in\twosilt(J),\\[1mm]
\{S_1,\ldots,S_n\}\in\intersmc(J),\\[1mm]
(\cd_{\std,J}^{\leq 0},\cd_{\std,J}^{\geq 0})\in\intertstr(J),\\[1mm]
(\cp^{\std,J}_{\geq 0},\cp^{\std,J}_{\leq 0})\in\intercotstr(J),\\[1mm]
J_J\in\sttilt(J),\\[1mm]
(\mod J,0)\fftor(J),\\[1mm]
\pi(\Gamma)\in\cto(\cc_{(Q,W)})
\end{array}
\]
correspond to each other under the above bijections and they are the unique sources, where $\{S_1,\ldots,S_n\}$ is a complete collection of non-isomorphisc simple $J$-modules, $p:\Gamma\rightarrow J$ is the canonical projection homomorphism and $\pi:\per(\Gamma)\rightarrow\cc_{(Q,W)}$ is the canonical projection functor. The objects in the same connected component as these distinguished objects in the ordered exchange graphs are called \emph{reachable} objects. Adding the condition `reachable' to the above sets we obtain subsets $\rtwosilt(\Gamma)$, $\rintersmc(\Gamma)$, $\rintertstr(\Gamma)$, $\rintercotstr(\Gamma)$, $\rtwosilt(J)$, $\rintersmc(J)$, $\rintertstr(J)$, $\rintercotstr(J)$, $\rsttilt(J)$, $\rfftor(J)$ and $\rcto(\cc_{(Q,W)})$ and a commutative diagram of isomorphisms of ordered exchange graphs (Figure~\ref{f:diagram-cat-reachable}), which, by definition, are all connected.

\begin{figure}
\[
\begin{xy} 0;<0.35pt,0pt>:<0pt,-0.35pt>::
(50,0) *+{\rtwosilt(\Gamma)}="0",
(500,0) *+{\rintercotstr(\Gamma)}="1",
(180,100) *+{\rintertstr(\Gamma)}="2",
(630,100) *+{\rintersmc(\Gamma)}="3",
(50,250) *+{\rtwosilt(J)}="4",
(500,250) *+{\rintercotstr(J)}="5",
(180,350) *+{\rintertstr(J)}="6",
(630,350) *+{\rintersmc(J)}="7",
(200,490) *+{\rsttilt(J)}="8",
(500,490) *+{\rfftor(J)}="9",
(400,630) *+{\rcto(\cc_{(Q,W)})}="10",
(382,350) *+{}="14",
"0", {\ar "1"}, {\ar "2"}, {\ar "3"}, {\ar "4"}, {\ar@/_90pt/ "10"}, 
"1", {\ar "0"},  {\ar@{-->} "5"}, 
"2", {\ar "3"}, 
"3", {\ar "1"}, {\ar "2"},
"4", {\ar@{-->} "5"}, {\ar "6"}, {\ar@{-->} "7"}, {\ar@/_20pt/ "8"},
"5", {\ar@{-->} "4"},
"6", {\ar "7"}, {\ar "2"}, 
"7", {\ar@{-->} "5"}, {\ar "6"}, {\ar "3"},
"8", {\ar "9"}, 
"9", {\ar "6"}, {\ar@{-} "14"}, "14", {\ar@{-->} "2"},
"10", {\ar "8"},
\end{xy}
\]
\caption{}
\label{f:diagram-cat-reachable}
\end{figure}

For an acyclic quiver with trivial potential, it follows from~\cite[Theorem 1.1]{Ladkani07} that Bernstein--Gelfand--Ponomarev reflections (\ie mutations at sinks/sources) induce flip-flops of these ordered exchange graphs. 

\subsection{From cluster-tilting objects to clusters}\label{ss:from-cto-to-cluster}

Let $(Q,W)$ be a Jacobi-finite quiver with potential such that $Q$ is a cluster quiver and $W$ is non-degenerate, $\Gamma=\hat{\Gamma}(Q,W)$, $J=\hat{J}(Q,W)$ and let $\cc_{(Q,W)}$ denote the Amiot cluster category of $(Q,W)$ and $\pi:\per(\Gamma)\rightarrow\cc_{(Q,W)}$ denote the canonical projection functor.
Assume $Q_0=\{1,\ldots,n\}$. 

Let $M$ be an object of $\cc_{(Q,W)}$. Recall that there is a triangle
\[
\xymatrix{P^{-1}\ar[r] & P^0\ar[r] & M\ar[r] &\Sigma P^{-1}}
\]
with $P^{-1}$ and $P^0$ in $\add(\pi(\Gamma))$. Define the \emph{index} of $M$ as 
\[\ind(M):=[P^0]-[P^{-1}]\in K_0^{\rm split}(\add(\pi(\Gamma))).\]
The object $\pi(\Gamma)$ has a canonical decomposition $\pi(\Gamma)=\bigoplus_{i=1}^n \pi(\Gamma_i)$, where $\Gamma_i=e_i\Gamma$. Thus $[\pi(\Gamma_1)],\ldots,[\pi(\Gamma_n)]$ form a (standard) basis for $K_0^{\rm split}(\add(\pi(\Gamma)))$. Via this basis we identity $K_0^{\rm split}(\add(\pi(\Gamma)))$ with $\mathbb{Z}^n$ and identify the index $\ind(M)$ of $M$ with the corresponding $n$-tuple of integers $(\ind_1(M),\ldots,\ind_n(M))$. 

\begin{theorem}{(\cite[Theorem 2.3]{DehyKeller08} and \cite[Proposition 3.1]{Plamondon11b})}\label{t:index-determine-rigid-object} Let $M$ and $N$ be two rigid objects of $\cc_{(Q,W)}$. Then $M$ is isomorphic to $N$ if and only if $\ind(M)=\ind(N)$.
\end{theorem}

Define the  \emph{F-polynomial} of $M$ as
\[F_M(y_1,\ldots,y_n):=\sum_{d}\chi(\mathrm{Gr}_d(\Hom(\pi(\Gamma),\Sigma M)))y_1^{d_1}\cdots y_n^{d_n},\]
where $d$ runs over all $n$-tuples of non-negative integers,  $\mathrm{Gr}_d$ denotes the variety of submodules with dimension vector $d$ and $\chi$ is the Euler--Poincar\'e characteristic.

Now we define a map, the \emph{Caldero--Chapoton map}:
\[\xymatrix@R=0.5pc@C=0.2pc{\mathrm{CC}&:&\cc_{(Q,W)}\ar[rrr]&&&\mathbb{Z}[x_1^{\pm 1},\ldots,x_n^{\pm 1},x_{n+1},\ldots,x_{2n}]\\
&&M\ar@{|->}[rrr]& && x_1^{\ind_1(M)}\cdots x_n^{\ind_n(M)}F_M(\hat{y}_1,\ldots,\hat{y}_n)}\]
where $\hat{y}_i=x_{n+i}\prod_{j=1}^n x_j^{b_{ji}}$ for $i=1,\ldots,n$, and $B(Q)=(b_{ji})_{1\leq j,i\leq n}$ is the matrix of $Q$.
This map plays a central role in the theory of additive categorification of cluster algebras, which has proved powerful in understanding cluster algebras, for example in proving a number of Fomin and Zelevinsky's conjectures. It was originally defined by Caldero and Chapoton~\cite{CalderoChapoton06} to use quiver representations to categorify cluster algebras (without coefficients) with defining quiver being of Dynkin type.  This work was generalised to all acyclic quivers by Caldero and Keller~\cite{CalderoKeller06,CalderoKeller08} (see also Hubery~\cite{Hubery06}) and further to 2-Calabi--Yau triangulated categories by Fu and Keller~\cite{FuKeller10} and by Palu~\cite{Palu08},  and to Amiot cluster categories of (not necessarily Jacobi-finite) quivers with potential by Plamondon~\cite{Plamondon11,Plamondon11b}. In parallel, instead of objects in $\cc_{(Q,W)}$, Derksen, Weyman and Zelevinsky constructed in~\cite{DerksenWeymanZelevinsky10} the Caldero--Chapoton map for decorated representations over the Jacobian algebra $J$ and gave the first complete proof of some of Fomin and Zelevinsky's conjectures in the skew-symmetric case; Nagao constructed in~\cite{Nagao10} the Caldero--Chapoton map for certain objects of $\per(\Gamma)$ and related it to Donaldson--Thomas invariants. Geiss, Leclerc and Schr\"oer took a different approach for stably 2-Calabi--Yau Frobenius categories arising from preprojective algebras in \cite{GeissLeclercSchroer05,GeissLeclercSchroer06,GeissLeclercSchroer11} and later they proved in~\cite{GeissLeclercSchroer10} that the two approaches are closely related.

Let $\Cl(Q)$ denote the set of clusters of the cluster algebra $\ca_Q$ with principal coefficients (Section~\ref{ss:cluster}). An important feature of the Caldero--Chapoton map is

\begin{theorem}{(\cite{Plamondon11b})}\label{t:from-cto-to-clusters}
The map $\mathrm{CC}$ induces a bijection
\[\xymatrix@R=0.5pc{\rcto(\cc_{(Q,W)})\ar[r]&\Cl(Q)\\
M\ar@{|->}[r] & \{\mathrm{CC}(M_1),\ldots,\mathrm{CC}(M_n),x_{n+1},\ldots,x_{2n}\},}\]
which commutes with mutations.
Here $M=M_1\oplus\ldots\oplus M_n$ is a decomposition of $M$ into the direct sum of indecomposable objects.
\end{theorem}

\begin{remark}
\begin{itemize}
\item[(a)] With $\cc_{(Q,W)}$ being replaced by a suitable subcategory, Theorem~\ref{t:from-cto-to-clusters} holds for any cluster quiver, see \cite{Plamondon11b}.
\item[(b)] Thanks to~\cite[Corollary 5.5]{CerulliKellerLabardiniPlamondon12}, with suitable modification of the CC map we can replace in the above theorem the cluster algebra $\ca_Q$  by a cluster algebra with arbitrary coefficients with defining quiver $Q$.
\end{itemize}
\end{remark}

Theorem~\ref{t:from-cto-to-clusters} is a consequence of the following Theorem~\ref{t:from-cto-to-cluster-with-ice-quiver}, isomorphism (\ref{eq:cy-reduction}) and Proposition~\ref{p:ind-and-Fpoly-in-cy-reduction}.

A priori one has to work with the quiver with potential $(Q',W)$, where $Q'=\hat{Q}$ is the framed quiver associated to $Q$ with frozen vertices $\{n+1,\ldots,2n\}$ (Section~\ref{ss:mut}). Let $\Gamma'=\hat{\Gamma}(Q',W)$ and $J'=\hat{J}(Q',W)$. They are related to $\Gamma$ and $J$ by $\Gamma=\Gamma'/(e_{n+1}\oplus\ldots\oplus e_{2n})$ and $J=J'/(e_{n+1}\oplus\ldots\oplus e_{2n})$.
  Let $\cc_{(Q',W)}$ be the Amiot cluster category of $(\hat{Q},W)$ and $\pi':\per(\Gamma')\rightarrow\cc_{(Q',W)}$ be the canonical projection functor. Then $\Gamma'=\bigoplus_{i=1}^{2n} \Gamma'_i$ with $\Gamma'_i=e_i\Gamma'$. 
Consider the full subcategory $\cu$ of $\cc_{(Q',W)}$ which consists of objects $M$ such that $\Hom(\bigoplus_{i=1}^n\pi'(\Gamma'_{n+i}),\Sigma M)$ vanishes. Denote by $\rcto_{\bigoplus_{i=1}^n \pi'(\Gamma_{n+i})}(\cc_{(Q',W)})$ the set of isomorphism classes of reachable basic cluster-tilting objects of $\cc_{(Q',W)}$ which have $\bigoplus_{i=1}^n \pi'(\Gamma'_{n+i})$ as a direct summand.

\begin{theorem}(\cite[Theorem 3.7 and its proof]{Plamondon11b})\label{t:from-cto-to-cluster-with-ice-quiver}
The map $\mathrm{CC}$ induces a bijection
\[\xymatrix@R=0.5pc{\rcto_{\bigoplus_{i=1}^n\pi'(\Gamma'_{n+i})}(\cc_{(Q',W)})\ar[r]&\Cl(Q)\\
M\ar@{|->}[r] & \{\mathrm{CC}(M_1),\ldots,\mathrm{CC}(M_n),x_{n+1},\ldots,x_{2n}\},}\]
which commutes with mutations. Here $M=M_1\oplus\ldots\oplus M_n\oplus\pi'(\Gamma'_{n+1})\oplus\ldots\oplus \pi'(\Gamma'_{2n})$ is a decomposition of $M$ into the direct sum of indecomposable objects.
\end{theorem}

It follows from Theorem~\ref{t:cy-reduction} that $\cu/\add(\bigoplus_{i=1}^n\pi'(\Gamma'_{n+i}))$ is naturally a $2$-Calabi--Yau triangulated category and the additive quotient functor 
\[
\xymatrix{\cu\ar[r] & \cu/\add(\bigoplus_{i=1}^n\pi'(\Gamma'_{n+i}))}
\] 
induces an isomorphism of exchange graphs 
\[
\xymatrix{\cto_{\bigoplus_{i=1}^n\pi'(\Gamma'_{n+i})}(\cc_{(Q',W)})\ar[r]&\cto(\cu/\add(\bigoplus_{i=1}^n\pi'(\Gamma'_{n+i}))}.
\] 
In particular, $\pi'(\Gamma')$ is a cluster-tilting object of $\cu/\add(\bigoplus_{i=1}^n\pi'(\Gamma'_{n+i}))$. In fact, Keller shows that $\cu/\add(\bigoplus_{i=1}^n\pi'(\Gamma'_{n+i}))$ is triangle equivalent to $\cc_{(Q,W)}$ and the image of $\pi'(\Gamma')$ is $\Gamma$, see \cite[Theorem 7.4]{Keller11}. So we obtain an additive quotient functor
\[
\Phi:\xymatrix{\cu\ar[r] & \cc_{(Q,W)}}
\]
which induces an isomorphism of exchange graphs
\begin{align}
\xymatrix{\rcto_{\bigoplus_{i=1}^n\pi'(\Gamma'_{n+i})}(\cc_{(Q',W)})\ar[r] &\rcto_{(Q,W)}}.\label{eq:cy-reduction}
\end{align}
Recall that for an object in $\cc_{(Q',W)}$, the index is an element in $K_0(\add(\pi'(\Gamma')))$ which is a free abelian group of rank $2n$ and the F-polynomial is a polynomial in $2n$ variables $y_1,\ldots,y_{2n}$. Part (a) of the following result is a combination of the proof of \cite[Theorem 3.13]{Plamondon11b} and \cite[Proposition 3.14]{Plamondon11c}.

\begin{proposition}\label{p:ind-and-Fpoly-in-cy-reduction}
Let $M$ be an indecomposable object of $\cu$ which is not isomorphic to $\pi'(\Gamma'_{n+i})$ for any $i=1,\ldots,n$.
\begin{itemize}
\item[(a)] For any $i=1,\ldots,n$, the coefficient of $[\Gamma'_{n+i}]$ in $\ind(M)$ is trivial. Moreover, if we identify $[\Gamma'_i]$ with $[\Gamma_i]$ for $1\leq i\leq n$, then $\ind(M)=\ind(\Phi M)$.
\item[(b)] The polynomial $F_M(y_1,\ldots,y_{n},y_{n+1},\ldots,y_{2n})$ is constant in $y_{n+1},\ldots,y_{2n}$. Moreover, $F_M(y_1,\ldots,y_n,y_{n+1},\ldots,y_{2n})=F_{\Phi M}(y_1,\ldots,y_n)$.
\end{itemize}
\end{proposition}
\begin{proof} (b) This proof is due to Plamondon. It suffices to show: (b1) As a $J'$-module, $\Hom_{\cc_{(Q',W)}}(\pi'(\Gamma'),\Sigma M)$ is supported on $J$; (b2) $\Hom_{\cc_{(Q',W)}}(\pi'(\Gamma'),\Sigma M)$, as a $J$-module, is isomorphic to $\Hom_{\cc_{(Q,W)}}(\pi(\Gamma),\Sigma\Phi M)$.

(b1) is clear because by the definition of $\cu$ we have $\Hom_{\cc_{(Q',W)}}(\pi'(\Gamma'_{n+i}),\Sigma M)=0$ for any $i=1,\ldots,n$.

(b2) follows from \cite[Lemma 4.8]{IyamaYoshino08} since the functor $\Phi$ is a Calabi--Yau reduction and we have $\Phi(\pi'(\Gamma'))=\pi(\Gamma)$.
\end{proof}

\subsection{The correspondences between silting objects, simple-minded collections, cluster-tilting objects, clusters, ice quivers, g-matrices and c-matrices}

Let $(Q,W)$ be a Jacobi-finite quiver with potential such that $Q$ is a cluster quiver and $W$ is non-degenerate. Let $\Gamma=\hat{\Gamma}(Q,W)$, $J=\hat{J}(Q,W)$ and let $\cc_{(Q,W)}$ denote the Amiot cluster category and $\Cl(Q)$ denote the clusters of the cluster algebra $\ca_Q$ with principal coefficients (Section~\ref{ss:cluster}).

\begin{theorem}\label{t:silt-smc-cto-cluster-mut-gmat-cmat}
There is a commutative diagram of bijections which commute with mutations
\[
\xymatrix{
\rtwosilt(\Gamma)\ar[rrr]\ar[ddd]\ar[dr] & & & \rintersmc(\Gamma)\ar[lll]\ar[ddd]\\
&\rcto(\cc_{(Q,W)})\ar[ddl]\ar[d]&&\\
&\Cl(Q)\ar[ld]\ar[r]&\mut(Q)\ar[dr]\\
\gmat(Q)\ar[rrr]&&&\cmat(Q)\ar[lll]
}
\]
\end{theorem}

The four maps $\rtwosilt(\Gamma)\leftrightarrow\rintersmc(\Gamma)$, $\rtwosilt(\Gamma)\rightarrow\rcto(\cc_{(Q,W)})$ and $\rcto(\cc_{(Q,W)})\rightarrow\Cl(Q)$ were defined in Sections~\ref{ss:silting-smc-tstr-cotstr},~\ref{ss:silting-sttilting-cto} and~\ref{ss:from-cto-to-cluster}, respectively. We will define the other maps and check the commutativity of the diagram. Again we will state the bijectivity only for some maps.

$\Cl(Q)\longrightarrow\gmat(Q)$:
This map is part of the definition of $\gg$-matrices, see Section~\ref{ss:gmatrix}. By definition it commutes with mutations.

$\rtwosilt(\Gamma)\longrightarrow\gmat(Q)$: Assume that $Q_0=\{1,\ldots,n\}$. Put $\Gamma_i=e_i\Gamma$. Then each $\Gamma_i$ is indecomposable in $\per(\Gamma)$ and $\Gamma=\Gamma_1\oplus\ldots\oplus\Gamma_n$ is a silting object in $\per(\Gamma)$. Hence, by Theorem~\ref{t:generators-of-grothendieckgroup-silting-case}, the Grothendieck group $K_0(\per(\Gamma))$ is free of rank $n$ and the isomorphism classes $[\Gamma_1],\ldots,[\Gamma_n]$ form a basis of $K_0(\per(\Gamma))$. 

Let $M=M_1\oplus\ldots\oplus M_n$ be a basic 2-term silting object of $\per(\Gamma)$ with $M_i$ indecomposable for each $i$. Then the isomorphism classes $[M_1],\ldots,[M_n]$ form a basis of $K_0(\per(\Gamma))$ for the same reason as above. Let $T_{M,\Gamma}$ be the invertible matrix representing a change of basis from $\{[\Gamma_1],\ldots,[\Gamma_n]\}$ to $\{[M_1],\ldots,[M_n]\}$, \ie
\[
([M_1],\ldots,[M_n])=([\Gamma_1],\ldots,[\Gamma_n])T_{M,\Gamma}.
\]

\begin{theorem}{(\cite[Theorem 6.18]{Nagao10})} The assignment $M\mapsto T_{M,\Gamma}$ defines a map $\rtwosilt(\Gamma)\rightarrow\gmat(Q)$, which commutes with mutations.
\end{theorem}

$\rcto(\cc_{(Q,W)})\longrightarrow\gmat(Q)$:  Let $M=M_1\oplus \ldots\oplus M_n$ be a basic cluster-tilting object of $\cc_{(Q,W)}$ with $M_i$ indecomposable for each $i$. Applying the map $\ind$ of taking indices to each $M_i$, we obtain a matrix $T_M$ such that 
\[
(\ind(M_1),\ldots,\ind(M_n))=([\pi(\Gamma_1)],\ldots,[\pi(\Gamma_n)])T_M.
\]
By \cite[Proposition 3.6 and the proof of Theorem 3.7]{Plamondon11b} and Proposition~\ref{p:ind-and-Fpoly-in-cy-reduction}, the assignment $M\mapsto T_M$ defines a map $\rcto(\cc_{(Q,W)})\rightarrow\gmat(Q)$, which is the composition
\[\rcto(\cc_{(Q,W)})\rightarrow\Cl(Q)\rightarrow\gmat(Q).\]
This map is bijective by Theorem~\ref{t:index-determine-rigid-object}. 
That  the composition
\[\rtwosilt(\Gamma)\rightarrow\rcto(\cc_{(Q,W)})\rightarrow\gmat(Q)\]
is $\rtwosilt(\Gamma)\rightarrow\gmat(Q)$
follows immediately from the definitions of these three maps. 

$\Cl(Q)\longrightarrow \mut(Q)$: Let $\underline{u}$ be a cluster of $\ca_Q$. According to Theorem~\ref{t:cluster-determine-seed}, there is a unique seed $(\underline{u},R)$. Then the assignment $\underline{u}\mapsto R$ defines a map from $\Cl(Q)$ to $\mut(Q)$ which commutes with mutations.

$\mut(Q)\longrightarrow\cmat(Q)$: This map is part of the definition of $\ccc$-matrices, see Section~\ref{ss:cmatrix}. By definition it commutes with mutations.

$\gmat(Q)\longrightarrow\cmat(Q)$: According to \cite[Theorem 1.2]{NZ:tropicalduality}, taking inverse transpose $T\mapsto T^{-\mathrm{tr}}$ defines a bijection between the set of $\gg$-matrices and the set of $\ccc$-matrices which commutes with mutations. One checks the commutativity of the lower square of the desired diagram by using the fact that all the four maps commute with mutations. The bijectivity of $\Cl(Q)\rightarrow \mut(Q)$ and $\mut(Q)\rightarrow\cmat(Q)$ is a consequence of the fact that these two maps are surjective and the other two maps in the square are bijective.

$\rintersmc(\Gamma)\longrightarrow\cmat(Q)$: Let $X$ be an object of $\cd_{fd}(\Gamma)$. Then the cohomologies of $X$ are finite-dimensional modules over $J=H^0(\Gamma)$. Define the dimension vector $\underline{\dim}(X)$ of $X$ as
\[\underline{\dim}(X):=\sum_{i\in\mathbb{Z}}(-1)^i\underline{\dim}(H^i(X)).\]
This yields a map
\[
\xymatrix@C=1pc@R=0.5pc{\underline{\dim}\hspace{5pt}:\hspace{-7pt}&\rintersmc(\Gamma)\ar[rr]&&\cm_n(\mathbb{Z})\\
&\{X_1,\ldots,X_n\}\ar@{|->}[rr]&&(\underline{\dim}(X_1),\ldots,\underline{\dim}(X_n))}
\]
where $\cm_n(\mathbb{Z})$ is the set of $n\times n$ matrices with integer entries. This map has range $\cmat(Q)$ and commutes with mutations, see~\cite[Theorem 7.9 and Section 7.10]{Keller:derivedcluster}. In conjunction with Remark~\ref{r:smc-are-stalk} this implies the sign-coherence of $c$-vectors, as observed by Keller in~\cite[Theorem 7.9]{Keller:derivedcluster}. The commutativity of the outer square of the desired diagram follows from the fact that all four maps commute with mutations. But next we give an explanation from a different point of view.

Let $\{S_1,\ldots,S_n\}$ be a complete set of pairwise non-isomorphic simple $J$-modules, viewed as a collection of objects in $\cd_{fd}(\Gamma)$ via the restriction functor $p^*$. Recall that both $\{[S_1],\ldots,[S_n]\}$ and $\{[X_1],\ldots,[X_n]\}$ are bases of the Grothendieck group $K_0(\cd_{fd}(\Gamma))$. In fact, the matrix $\underline{\dim}(X_1,\ldots,X_n)$ is precisely the matrix representing a change of basis from $\{[S_1],\ldots,[S_n]\}$ to $\{[X_1],\ldots,[X_n]\}$.

Observe that there is the \emph{Euler form}
\[
\xymatrix@C=1.5pc@R=0.5pc{
K_0(\per(\Gamma))\times K_0(\cd_{fd}(\Gamma))\ar[r]&\mathbb{Z}\\
(M,X)\ar@{|->}[r]&\sum_{i\in\mathbb{Z}}(-1)^i\dim \Hom(M,\Sigma^i X)
}
\]
which is non-degenerate.

Let $M=M_1\oplus\ldots\oplus M_n$ be a basic 2-term silting object in $\per(\Gamma)$ with $M_i$ indecomposable for each $i$ and let $\{X_1,\ldots,X_n\}$ be the corresponding 2-term simple-minded collection of $\cd_{fd}(\Gamma)$. Thanks to (\ref{eq:silting-smc-are-dual}), the sets $\{[M_1],\ldots,[M_n]\}$ and $\{[X_1],\ldots,[X_n]\}$ form dual bases of $K_0(\per(\Gamma))$ and $K_0(\cd_{fd}(\Gamma))$ with respect to the Euler form. As a consequence, the matrices representing changes of bases from $\{[\Gamma_1],\ldots,[\Gamma_n]\}$ to $\{[M_1],\ldots,[M_n]\}$ and from $\{[S_1],\ldots,[S_n]\}$ to $\{[X_1],\ldots,[X_n]\}$ are related by taking the inverse of the transpose. Namely, $\underline{\dim}([X_1],\ldots,[X_n])=T_{M,\Gamma}^{-\mathrm{tr}}$, so the outer square of the diagram is commutative.

\begin{remark} 
\begin{itemize}
\item[(a)]
Recall from Section~\ref{ss:silting-smc-tstr-cotstr} that the bijection $\rintersmc(J)\rightarrow\rintersmc(\Gamma)$ is induced by the restriction $p^*$ along the morphism $p:\Gamma\rightarrow J$, and hence preserves dimension vectors. Thus, combining it with the bijection $\rintersmc(\Gamma)\rightarrow\cmat(Q)$, we obtain a bijection
\[\underline{\dim}: \rintersmc(J)\longrightarrow\cmat(Q).\]
This map is expected to help to understand $\cmat(Q)$ using representation theory over the finite-dimensional algebra $J$.

\item[(b)]
Let $C=(c_{ij})_{1\leq i,j\leq n}$ be a $\ccc$-matrix and $M=M_1\oplus\ldots\oplus M_n$ be the corresponding reachable basic cluster-tilting object of $\cc_{(Q,W)}$. Then there is a triangle
\[
\xymatrix{p(\Gamma_i)\ar[r]& \bigoplus_{j:c_{ij}>0}M_j^{\oplus c_{ij}}\ar[r]&\bigoplus_{j:c_{ij}<0}M_j^{\oplus (-c_{ij})}\ar[r] &\Sigma p(\Gamma_i)}\]
for each $1\leq i\leq n$.

\item[(c)] As a consequence of Theorem~\ref{t:silt-smc-cto-cluster-mut-gmat-cmat}, the graph $\twosilt(\Gamma)$ (respectively, $\intersmc(\Gamma)$, $\intertstr(\Gamma)$, $\intercotstr(\Gamma)$, $\twosilt(J)$, $\intersmc(J)$, $\intertstr(J)$, $\intercotstr(J)$, $\sttilt(J)$, $\fftor(J)$ and $\cto(\cc_{(Q,W)})$) is connected if and only if it is isomorphic to $\mut(Q)$. In particular, if $\twosilt(\Gamma)$ is connected, then $\mut(Q)$ has a sink.
\end{itemize}
\end{remark}

\bigskip

\appendix

\section{Non-positive dg algebras}\label{a:non-positive-dg-alg}

Let $k$ be an algebraically closed field.



 \subsection{Non-positive dg algebras}\label{ss:non-pos-dg}
 A dg algebra $A$ is said to be \emph{non-positive} if the degree $i$ component $A^i$ vanishes for $i>0$.
 
 Let $A$ be a non-positive dg algebra.  
 Then $A$ is a silting object of $\per(A)$ because
\[\Hom_{\per(A)}(A,\Sigma^i A)=H^i(A)\]
vanishes for $i>0$.
Conversely, let $\cc$ be an idempotent complete algebraic triangulated category which has a silting object $P$. Here a triangulated category is said to be \emph{algebraic} if it is triangle equivalent to the stable category of a Frobenius category. Then by Keller's Morita theorem for triangulated categories (\cite[Theorem 3.8 (b)]{Keller06d}), there is a non-positive dg algebra $A$ such that there is a triangle equivalence $\cc\rightarrow \per(A)$ which takes $P$ to $A$ (see for example~\cite[Lemma 4.1]{KoenigYang12} for a detailed proof). Part (a) of the following theorem is obtained by combining Proposition 6.2.1 and Proposition 5.2.2 of \cite{Bondarko10}, part (b) implicitly appears in \cite[Section 2.1]{Amiot09} (see for example~\cite[Seciton 2.4]{KalckYang12} for a detailed proof) and part (c) is easy to prove by using (b).

\begin{theorem} Let $A$ be a non-positive dg algebra.
\begin{itemize}
\item[(a)] The pair $(\cp^\std_{\geq 0},\cp^\std_{\leq 0})$ is a bounded co-$t$-structure on $\per(A)$ with co-heart $\add_{\per(A)}(A)$, where $\cp^\std_{\geq 0}$ (respecitively, $\cp^\std_{\leq 0}$) is the smallest full subcategory of $\per(A)$ which contains $\Sigma^i A$ for $i\leq 0$ (respectively, $i\geq 0$) and which is closed under extensions and direct summands. We will refer to this co-$t$-structure as the \emph{standard} co-$t$-structure on $\per(A)$.
\item[(b)] The pair $(\cd_{\std}^{\leq 0},\cd_{\std}^{\geq 0})$ is a bounded $t$-structure on $\cd_{fd}(A)$ with heart being equivalent to $\mod H^0(A)$, where $\cd_{\std}^{\leq 0}$ (respectively, $\cd_{\std}^{\geq 0}$) is the full subcategory of $\cd_{fd}(A)$ containing those dg $A$-modules whose cohomologies are concentrated in non-positive degrees (respectively, non-negative degrees). We will refer to this $t$-structure as the \emph{standard} $t$-structure on $\cd_{fd}(A)$.
\item[(c)] Suppose that $H^0(A)$ is finite-dimensional. Then a complete collection of pairwise non-isomorphic simple $H^0(A)$ modules, considered as a collection of objects in $\cd_{fd}(A)$, is simple-minded.
\end{itemize}
\end{theorem}

Let $A$ be a non-positive dg algebra. Set
\begin{align*}
\cf_A&=\{\cone(f)\mid f \text{ is a morphism in }\add_{\per(A)}(A)\}\subseteq \per(A).
\end{align*}
Objects of $\cf_A$ will be called \emph{2-term} objects of $\per(A)$ for obvious reasons. The following is an easy observation.

\begin{lemma}\label{l:2-term-partial-silting-and-vanishing-of-ext1}
Let $X$ be an object of $\cf_A$. Then $\Hom(X,\Sigma^p X)=0$ for $p\geq 2$. In particular $X$ is a presilting object if and only if $\Hom(X,\Sigma X)=0$.
\end{lemma}



\subsection{The induction functor}
Let $A$ be a non-positive dg algebra. Let $\bar{A}=H^0(A)$. Denote by $p$ the canonical projection $A\rightarrow\bar{A}$. Then the induction functor $p_*:\per(A)\rightarrow\per(\bar{A})$ restricts to an additive equivalence $\add_{\per(A)}(A)\stackrel{\sim}{\rightarrow}\add_{\per(\bar{A})}(\bar{A})$ and induces a canonical isomorphism $K_0(\per(A))\rightarrow K_0(\per(\bar{A}))$ of Grothendieck groups. 

\subsubsection{The bijection}

\begin{proposition}\label{p:induction-silting}
The induction functor $p_*:\per(A)\rightarrow\per(\bar{A})$ induces a bijection between the sets of isomorphism classes of $2$-term silting objects.
\end{proposition}

We need the following two $4$-lemmas.
\begin{lemma}\label{l:four-lemma}
Consider the following commutative diagram of abelian groups with exact rows
\[
\begin{xy}
\SelectTips{cm}{10}\xymatrix{M^{-1}\ar[r]\ar[d]^{f} & M^0 \ar[r]\ar[d]^g & M^1 \ar[r]\ar[d]^h & M^2 \ar[d]^i\\
N^{-1} \ar[r] & N^0 \ar[r] & N^1 \ar[r] & N^2}
\end{xy}\]
\begin{itemize}
\item[(a)] If $f$ is surjective, $g$ is injective and $i$ is injective, then $h$ is injective.
\item[(b)] If $f$ is surjective, $h$ is surjective and $i$ is injective, then $g$ is surjective.
\end{itemize}
\end{lemma}

\begin{proof}[Proof of Proposition~\ref{p:induction-silting}]
We first show that if $X$ is an object of $\cf_A$ then $X$ is a silting object if and only if $p_*(X)$ is a silting object. There is a triangle
\[\xymatrix{X'\ar[r] & X\ar[r] & X''\ar[r] &\Sigma X'}\]
with $X'\in\add_{\per(A)}(A)$ and $X''\in\Sigma\add_{\per(A)}(A)$.
Applying the two functors $\Hom_{\per(A)}(?,X'')$ (respectively, $\Hom_{\per(\bar{A})}(?, p_*(X''))$) and $\Hom_{\per(A)}(?,X')$ (respectively, $\Hom_{\per(\bar{A})}(?,p_*(X'))$) to this triangle (respectively, its image under $p_*$), we obtain two commutative diagrams
\[{\small 
\begin{xy}
\SelectTips{cm}{10}
\xymatrix@C=0.5pc{(X',\Sigma^{-1}X'')\ar[r]\ar[d]^{f_1} & (X'',X'')\ar[r]\ar[d]^{f_2} & (X,X'')\ar[r]\ar[d]^{f_3} & (X',X'')=0\ar[d]\\
(p_*(X'),\Sigma^{-1}p_*(X''))\ar[r] & (p_*(X''),p_*(X''))\ar[r] & (p_*(X),p_*(X''))\ar[r] & (p_*(X'),p_*(X''))=0}
\end{xy}}\]
\[{\small 
\begin{xy}
\SelectTips{cm}{10}
\xymatrix@C=0.5pc{(X',X')\ar[r]\ar[d]^{g_1} & (X'',\Sigma X')\ar[r]\ar[d]^{g_2} & (X,\Sigma X')\ar[r]\ar[d]^{g_3} & (X',\Sigma X')=0\ar[d]\\
(p_*(X'),p_*(X'))\ar[r] & (p_*(X''),\Sigma p_*(X'))\ar[r] & (p_*(X),\Sigma p_*(X'))\ar[r] & (p_*(X'),\Sigma p_*(X'))=0}
\end{xy}}\]
Since $p_*:\add_{\per(A)}(A)\rightarrow\add_{\per(\bar{A})}(\bar{A})$ is an equivalence, it follows that $f_1$, $f_2$, $g_1$ and $g_2$ are bijective. Therefore by Lemma~\ref{l:four-lemma}, $f_3$ and $g_3$ are bijective. Applying $\Hom_{\per(A)}(X,?)$ (respectively, $\Hom_{\per(\bar{A})}(p_*(X),?)$) to the above triangle (respectively, its image under $p_*$), we obtain the following commutative diagram
\[{\small 
\begin{xy}
\SelectTips{cm}{10}
\xymatrix@C=0.5pc{(X,X'')\ar[r]\ar[d]^{f_3} & (X,\Sigma X')\ar[r]\ar[d]^{g_3} & (X,\Sigma X)\ar[r]\ar[d]^h & (X,\Sigma X'')=0\ar[d]\\
(p_*(X),p_*(X''))\ar[r] & (p_*(X),\Sigma p_*(X'))\ar[r] & (p_*(X),\Sigma p_*(X))\ar[r] & (p_*(X),\Sigma p_*(X''))=0}
\end{xy}}\]
By Lemma~\ref{l:four-lemma}, $h$ is bijective. Therefore $X$ is a presilting object if and only if $p_*(X)$ is a presilting object\footnote{In the published version the rest of this proof is missing except the sentence starting with `The bijectivity'.}. If $X$ is a silting object, then $X$ is presilting and $A\in\thick(X)$. This implies that $p_*(X)$ is presilting and $\bar{A}\in\thick(p_*(X))$, so $p_*(X)$ is a silting object. Conversely, if $p_*(X)$ is a silting object, then $X$ is presilting, and by \cite[Corollary 2.4]{IyamaJorgensenYang14} there is a triangle in $\per(\bar{A})$:
\[
\xymatrix{
\bar{Y}_1\ar[r] & \bar{Y}_0\ar[r] & \Sigma \bar{A}\ar[r]^{\bar{f}} & \Sigma \bar{Y}_1,
}
\]
where $\bar{Y}_1,\bar{Y}_0\in\add_{\per(\bar{A})}(p_*(X))$. Let $Y_1\in\add_{\per(A)}(X)$ be such that $p_*(Y_1)\cong \bar{Y}_1$. We claim that the functor $p_*$ induces an isomorphism $\Hom_{\per(A)}(A,Y_1)\cong\Hom_{\per(\bar{A})}(\bar{A},p_*(Y_1))$. Then there exists a morphism $f\colon\Sigma A\to\Sigma Y_1$ such that $p_*(f)$ is isomorphic to $\bar{f}$.
Complete this morphism to a triangle in $\per(A)$:
\[
\xymatrix{
Y_1\ar[r] & Y_0 \ar[r] & \Sigma A\ar[r]^f & \Sigma Y_1.
}
\]
Then $Y_0\in\cf_A$ and the image under $p_*$ of this triangle is isomorphic to the above triangle in $\per(\bar{A})$, in particular, $p_*(Y_0)\cong\bar{Y}_0$. By Proposition~\ref{prop:kernel-of-push-out} below, $Y_0\in\add_{\per(A)}(X)$ and hence $A\in\thick_{\per(A)}(X)$, so $X$ is a silting object. 
This shows that $X$ is a silting object if and only if $p_*(X)$ is a silting object.
The bijectivity is a consequence of the following Proposition~\ref{prop:kernel-of-push-out}.

To prove the claim, we take a triangle 
\[
\xymatrix{
P^{-1}\ar[r] & P^0\ar[r] & Y_1\ar[r] & \Sigma P^{-1}
}
\]
with $P^{-1},P^0\in\add_{\per(A)}(A)$. Applying $\Hom_{\per(A)}(A,?)$ to this triangle and $\Hom_{\per(\bar{A})}(\bar{A},?)$ to its image under $p_*$, we obtain a commutative diagram
\[{\small 
\begin{xy}
\SelectTips{cm}{10}
\xymatrix@C=0.5pc{(A,P^{-1})\ar[r]\ar[d]^{\cong} & (A,P^0)\ar[r]\ar[d]^{\cong} & (A,Y^{-1})\ar[r]\ar[d] & (A,\Sigma P^{-1})=0\ar[d]\\
(\bar{A},p_*(P^{-1}))\ar[r] & (\bar{A},p_*(P^0))\ar[r] & (\bar{A},p_*(Y_1))\ar[r] & (\bar{A},\Sigma p_*(P^{-1}))=0}
\end{xy}}\]
The claim follows from Lemma~\ref{l:four-lemma}. 
\end{proof}

\begin{proposition}\label{prop:kernel-of-push-out}
Let $\ci$ be the ideal of $\cf_A$ consisting of morphisms factoring through morphisms $\Sigma X\rightarrow Y$ with $X,Y\in\add_{\per(A)}(A)$. Then $\ci^2=0$ and $p_*$ induces an equivalence $\cf_A/\ci\rightarrow\cf_{\bar{A}}$. In particular, $p_*$ is full, detects isomorphisms, preserves indecomposability and induces a bijection between isomorphism classes of objects of $\cf_A$ and those of $\cf_{\bar{A}}$.
\end{proposition}

\begin{proof} 
That $\ci^2=0$ holds is because $\Hom(A,\Sigma A)$ vanishes. Next we show that $p_*$ induces an equivalence $\cf_A/\ci\rightarrow\cf_{\bar{A}}$. Then the last statement follows immediately.

Let $Y\in\cf_{\bar{A}}$. Then $Y\cong\cone(g)$ for a morphism $g$ in $\add_{\per(\bar{A})}(\bar{A})$. Since $p_*:\add_{\per(A)}(A)\rightarrow\add_{\per(\bar{A})}(\bar{A})$ is an equivalence, it follows that there is a morphism $f$ in $\add_{\per(A)}(A)$ such that $g=p_*(f)$. Take $X=\cone(f)$. Then $p_*(X)\cong \cone(p_*(f))=\cone(g)=Y$. This shows that $p_*:\cf_A\rightarrow\cf_{\bar{A}}$ is dense.

Since there is no non-trivial morphism from $\Sigma\bar{A}$ to $\bar{A}$ in $\per(\bar{A})$, it follows that the image of a morphism in $\ci$ under $p_*$ is zero.

Let $X,Y\in\cf_A$. There are triangles in $\per(A)$
\[\xymatrix@R=0.5pc{P_X^{-1}\ar[r]^u & P_X^0\ar[r]^v &X\ar[r]^w &\Sigma P_X^{-1}\\
P_Y^{-1}\ar[r]^{u'} & P_Y^0\ar[r]^{v'} &Y\ar[r]^{w'} &\Sigma P_Y^{-1}}\]
with $P_X^{-1},P_X^0,P_Y^{-1},P_Y^0\in\add_{\per(A)}(A)$. 
Applying $\Hom_{\per(A)}(P_X^0,?)$ respectively $\Hom_{\per(\bar{A})}(p_*(P_X^0),?)$ to the triangle containing $Y$ respectively its image under $p_*$, we obtain the following commutative diagram with exact rows
\[{\small \hspace{-10pt}
\begin{xy}
\SelectTips{cm}{10}
\xymatrix@C=0.5pc{(P_X^0,P_Y^{-1})\ar[r]\ar[d]^{f_1} & (P_X^0,P_Y^0)\ar[r]\ar[d]^{f_2} & (P_X^0,Y)\ar[r]\ar[d]^{f_3} & (P_X^0,\Sigma Y^{-1})=0\ar[d]\\
(p_*(P_X^0),p_*(P_Y^{-1})) \ar[r] & (p_*(P_X^0),p_*(P_Y^0)) \ar[r] & (p_*(P_X^0),p_*(Y))\ar[r] & (p_*(P_X^0),\Sigma p_*(P_Y^{-1}))=0}
\end{xy}}\]
The maps $f_1$ and $f_2$ are bijective, implying that $f_3$ is bijective too.

Applying $\Hom_{\per(A)}(\Sigma P_X^{-1},?)$ respectively $\Hom_{\per(\bar{A})}(p_*(\Sigma P_X^{-1}),?)$ to the triangle containing $Y$ respectively its image under $p_*$, we obtain the following commutative diagram with exact rows
\[{\fontsize{8pt}{5pt}\hspace{-3pt}
\begin{xy}
\SelectTips{cm}{10}
\xymatrix@C=0.5pc{ (\Sigma P_X^{-1},P_Y^0)\ar[r]^{g_4}\ar[d] & (\Sigma P_X^{-1},Y)\ar[r]\ar[d]^{g_1} & (\Sigma P_X^{-1},\Sigma P_Y^{-1})\ar[r]\ar[d]^{g_2} & (\Sigma P_X^{0},\Sigma P_Y^{-1})\ar[d]^{g_3}\\
 (\Sigma p_*( P_X^{-1}),p_*(P_Y^0)) \ar[r] & (\Sigma p_*( P_X^{-1}),p_*(Y))\ar[r] & (\Sigma p_*(P_X^{-1}),\Sigma p_*(P_Y^{-1}))\ar[r] & (\Sigma p_*(P_X^{0}),\Sigma p_*(P_Y^{-1}))}
\end{xy}}\]
The vector space $(\Sigma p_*( P_X^{-1}),p_*(P_Y^0))$ is trivial and the maps $g_2$ and $g_3$ are bijective. By Lemma~\ref{l:four-lemma} (b), $g_1$ is surjective. A straightforward diagram chasing shows that $\Ker(g_1)=\Im(g_4)$.

Applying $\Hom_{\per(A)}(?,Y)$ respectively $\Hom_{\per(\bar{A})}(?,p_*(Y))$ to this triangle respectively its image under $p_*$, we obtain the following commutative diagram with exact rows
\[{\fontsize{7pt}{5pt}\hspace{-3pt}
\begin{xy}
\SelectTips{cm}{10}
\xymatrix@C=0.45pc{(\Sigma P_X^{0},Y)\ar[r]^\alpha\ar[d]^{h_1} & (\Sigma P_X^{-1},Y)\ar[r]^(0.55)\beta\ar[d]^{g_1} & (X,Y)\ar[r]^\gamma\ar[d]^{h_3} & (P_X^0,Y)\ar[r]^\delta \ar[d]^{f_3} & (P_X^{-1},Y)\ar[d]^{h_5}\\
(\Sigma p_*(P_X^0),p_*(Y))\ar[r]^{\alpha'} & (\Sigma p_*(P_X^{-1}),p_*(Y))\ar[r]^(0.55){\beta'} & (p_*(X),p_*(Y))\ar[r]^{\gamma'} & (p_*(P_X^0),p_*(Y))\ar[r]^{\delta'} & (p_*(P_X^{-1}),p_*(Y))}
\end{xy}}\]
Recall that $g_1$ is surjective and $f_3$ is bijective, and similarly one shows that $h_1$ is surjective and  $h_5$ is bijective. It follows from Lemma~\ref{l:four-lemma} (b) that $h_3$ is surjective.

We claim that $\Ker(h_3)=\Im(\beta\circ g_4)$. Notice that for $\psi\in \Hom_{\per(A)}(\Sigma P_X^{-1},P_Y^0)$ we have that $\beta\circ g_4(\psi)=v'\circ \psi\circ w$ belongs to $\ci$. So the proof is complete.

To prove the claim, take $\varphi\in \Ker(h_3)$. Then $h_4\circ\gamma(\varphi)=\gamma'\circ h_3(\varphi)=0$, implying that $\gamma(\varphi)=0$. So there is $\varphi_1\in \Hom_{\per(A)}(P_X^{-1},Y)$ such that $\varphi=\beta(\varphi_1)$. Then $\beta\circ h_2(\varphi_1)=h_3\circ\beta(\varphi_1)=0$. So there is $\varphi_2\in \Hom_{\per(\bar{A})}(\Sigma p_*(P_X^0),p_*(Y))$ such that $g_1(\varphi_1)=\alpha'(\varphi_2)$. Since $h_1$ is surjective, there is $\varphi_3\in \Hom_{\per(A)}(\Sigma P_X^0,Y)$ such that $\varphi_2=h_1(\varphi_3)$. Let $\varphi'_1=\varphi_1-\alpha(\varphi_3)$. Then
$g_1(\varphi'_1)=0$ and $\beta(\varphi'_1)=\varphi$. Because $\Ker(g_1)=\Im(g_4)$, this finishes the proof of the claim.
\end{proof}

\subsubsection{Compatibility with mutations}

Assume further that $\per(A)$ is Hom-finite.

\begin{proposition}\label{p:induction-silting-mutation}
Let $M$ be a $2$-term silting object of $\per(A)$ and $N$ be an indecomposable direct summand of $M$ such that $\mu^-_N(M)$ is again $2$-term. Then $\mu^-_{p_*(N)}(p_*(M))=p_*(\mu^-_N(M))$.
\end{proposition}
\begin{proof} By definition, $\mu_N^-(M)=N^-\oplus L$, where $N^-$ is given by the triangle in $\per(A)$
\[\xymatrix{N\ar[r]^f & E\ar[r] & N^-\ar[r] & \Sigma N},\]
where the morphism $f$ is a minimal left $\add(L)$-approximation. Applying the triangle functor $p_*$ we obtain a triangle in $\per(\bar{A})$
\[\xymatrix{p_*(N)\ar[r]^{p_*(f)} & p_*(E)\ar[r] & p_*(N^-)\ar[r] & \Sigma p_*(N)},\]
where $p_*(f)$ is a left $\add(p_*(L))$-approximation.
On the other hand, $p_*(N)$ is indecomposable and $p_*(M)=p_*(N)\oplus p_*(L)$. By definition, $\mu_{p_*(N)}(p_*(M))=p_*(N)^-\oplus p_*(L)$, where $p_*(N)^-$ is given by the triangle in $\per(\bar{A})$
\[\xymatrix{p_*(N)\ar[r]^g & E'\ar[r] & p_*(N)^-\ar[r]\ar[r] & \Sigma p_*(N)},\]
where the morphism $g$ is a minimal left $\add(p_*(L))$-approximation. Therefore $p_*(f)$ factors through $g$ and we have the following octahedron
\[{\scriptsize
\begin{xy} 0;<0.35pt,0pt>:<0pt,-0.35pt>::
(300,0) *+{p_*(N^-)} ="0",
(0,160) *+{p_*(N)^-}="1",
(420,160) *+{E''}="2",
(180,240) *+{p_*(N)}="3",
(600,240) *+{p_*(E)}="4",
(300,400) *+{E'}="5",
"0", {\ar@{->>} "3"}, "0", {\ar@{.>} "2"},
"1", {\ar "0"},
"1", {\ar@{->>} "3"}, "2", {\ar@{.>>} "1"},
"2", {\ar@{.>>} "5"}, "3", {\ar|g "5"},
"3", {\ar|{p_*(f)} "4"}, "4", {\ar "0"},
"4", {\ar@{.>} "2"}, "5", {\ar "1"},
"5", {\ar|h "4"},
\end{xy}}
\]
The morphism $h$ splits, and hence $E''$ belongs to $\add(p_*(L))$. Consider the triangle 
\[\xymatrix{p_*(N)^-\ar[r] & p_*(N^-)\ar[r] & E''\ar[r] & \Sigma p_*(N)^-}.\]
Since $p_*(L)\oplus p_*(N)^-$ is a silting object, the last morphism in the above triangle vanishes. Therefore the triangle splits. Because both $p_*(N)^-$ and $p_*(N^-)$ are indecomposable, they must be isomorphic and the desired result follows.
\end{proof}

\subsection{A summary in terms of abstract categories}

We summarise the main results of this appendix in terms of abstract categories.

Let $\cc$ be an idempotent complete algebraic triangulated category and let $T$ be a silting object of $\cc$ with endomorphism algebra $E=\End_\cc(T)$.
Put 
\begin{align*}
\cf_T&=\{\cone(f)\mid f \text{ is a morphism in }\add_{\cc}(T)\}\subseteq \cc,\\
\cf_E&=\{\cone(f)\mid f \text{ is a morphism in }\proj E\}\subseteq K^b(\proj E).
\end{align*}

\begin{theorem}
\begin{itemize}
\item[(a)] There is a triangle functor $Q:\cc\rightarrow K^b(\proj E)$ which takes $T$ to $E$ and which takes the aisle (respectively, co-aisle) of the co-$t$-structure of $\cc$ associated to $T$ to the aisle (respectively, co-aisle) of the standard co-$t$-structure of $K^b(\proj E)$.

\item[(b)]
Let $\ci$ be the ideal of $\cf_T$ consisting of morphisms factoring through morphisms $\Sigma X\rightarrow Y$ with $X,Y\in\add_{\cc}(T)$. Then $\ci^2=0$ and $Q$ induces an equivalence $\cf_T/\ci\rightarrow\cf_{E}$. In particular, $Q|_{\cf_T}:\cf_T\to\cf_E$ is full, detects isomorphisms, preserves indecomposability and induces a bijection between isomorphism classes of objects of $\cf_T$ and those of $\cf_{E}$.
\item[(c)]
The triangle functor $Q$ induces a bijection from the set of $2$-term silting objects in $\cc$ with respect to $T$ and the set of $2$-term silting objects in $K^b(\proj E)$.
If $\cc$ is Hom-finite, this bijection commutes with mutations.
\end{itemize}
\end{theorem}


\begin{thebibliography}{10}

\bibitem{AbeHoshino06}
H. Abe and M. Hoshino, 
On derived equivalences for selfinjective algebras.
\emph{Comm. Algebra} \textbf{34} (2006), no. 12, 4441--4452. 

\bibitem{AdachiIyamaReiten12}
T. Adachi, O. Iyama, and I. Reiten, $\tau$-tilting theory. \emph{Compos. Math.} \textbf{150} (2014), no,~3, 415--452.
  
\bibitem{Aihara10} 
T. Aihara,
Tilting-connected symmetric algebras. \emph{Algebr. Represent. Theory} \textbf{16} (2013), no. 3, 873--894.

\bibitem{AiharaIyama12}
T. Aihara and O. Iyama, Silting mutation in triangulated
  categories. \emph{J. London Math. Soc.} \textbf{85} (2012), no. 3, 633--668.

\bibitem{Al-Nofayee09}
S. Al-Nofayee, Simple objects in the heart of a {$t$}-structure. \emph{J.
  Pure Appl. Algebra} \textbf{213} (2009), no.~1, 54--59.

\bibitem{ACCERV:BPS}
M. Alim, S. Cecotti, C. Cordova, S. Espahbodi, A. Rastogi, and C. Vafa,
{N}=2 quantum field theories and their {BPS} quivers. \emph{Adv. Theor. Math. Phys.} \textbf{18} (2014), no.~1, 27--127.

\bibitem{Amiot09}
C. Amiot, Cluster categories for algebras of global dimension 2 and
  quivers with potential. \emph{Ann. Inst. Fourier (Grenoble)} \textbf{59} (2009),
  no.~6, 2525--2590.


\bibitem{AssemSimsonSkowronski06}
I. Assem, D. Simson and A. Skowro\'nski, 
\emph{Elements of the representation theory of associative algebras Vol. 1: 
Techniques of representation theory}. London Mathematical Society Student Texts 65, Cambridge University Press, Cambridge, 2006.




\bibitem{AuslanderSmalo81}
M. Auslander and S. O. Smal{\o}, Almost split sequences in
  subcategories, \emph{J. Algebra} \textbf{69} (1981), no.~2, 426--454.


\bibitem{BeilinsonBernsteinDeligne82}
A. A. Beilinson, J. Bernstein, and P. Deligne, \emph{Faisceaux pervers}. Ast{\'e}risque, vol. 100, Soc. Math.
  France, 1982.

\bibitem{BeligiannisReiten07}
A. Beligiannis and I. Reiten, Homological and homotopical
  aspects of torsion theories. \emph{Mem. Amer. Math. Soc.} \textbf{188} (2007),
  no.~883, viii+207.

\bibitem{Bondarko10}
M. V. Bondarko, Weight structures vs. {$t$}-structures; weight
  filtrations, spectral sequences, and complexes (for motives and in general).
  \emph{J. K-Theory} \textbf{6} (2010), no.~3, 387--504.


\bibitem{Bongartz81}
K. Bongartz, Tilted algebras. \emph{Representations of algebras (Puebla, 1980)}, pp. 26--38.
Lecture Notes in Math. 903, Springer, Berlin-New York, 1981. 

\bibitem{Bridgeland05}
T. Bridgeland, t-structures on some local {C}alabi-{Y}au varieties. \emph{J.
  Algebra} \textbf{289} (2005), no.~2, 453--483.

\bibitem{Bridgeland07}
\bysame, Stability conditions on triangulated categories, \emph{Ann. of Math.
  (2)} \textbf{166} (2007), no.~2, 317--345.
  
\bibitem{BDP}
T. Br\"{u}stle, G. Dupont, and M. P\'erotin,
On Maximal Green Sequences. \emph{Int. Math.
Res. Notices}, to appear; preprint (2012),
arXiv:1205.2050.


\bibitem{BuanMarshReinekeReitenTodorov06}
A. B. Buan, R. J. Marsh, M. Reineke, I. Reiten, and G.
  Todorov, Tilting theory and cluster combinatorics. \emph{Adv. 
  Math.} \textbf{204} (2006), no.~2, 572--618.


\bibitem{BuanReitenThomas11}
A. B. Buan, I. Reiten, and H. Thomas, Three kinds of mutation.
  \emph{J. Algebra} \textbf{339} (2011), 97--113.

\bibitem{CalderoChapoton06}
P. Caldero and F. Chapoton, Cluster algebras as
  {H}all algebras of quiver representations. \emph{Comment. Math. Helv.} \textbf{81}
  (2006), no.~3, 595--616.


\bibitem{CalderoKeller08}
P. Caldero and B. Keller, From triangulated categories to cluster algebras. \emph{Invent. Math.}
  \textbf{172} (2008), 169--211.

\bibitem{CalderoKeller06}
\bysame, From triangulated categories to
  cluster algebras. {II}. \emph{Ann. Sci. {\'E}cole Norm. Sup. (4)} \textbf{39}
  (2006), no.~6, 983--1009.


\bibitem{CCV}
S. Cecotti, C. Cordova, and C. Vafa, {B}raids, {W}alls, and
  {M}irrors. Preprint (2011), arXiv:1110.2115 [hep-th].
  

\bibitem{CerulliKellerLabardiniPlamondon12}
G. Cerulli Irelli, B. Keller, D. Labardini-Fragoso, and
  P.-G. Plamondon, Linear independence of cluster monomials for
  skew-symmetric cluster algebras. \emph{Compos. Math.} \textbf{149} (2013), no.~10, 1753--1764.
  
\bibitem{ClineParshallScott86}
E. Cline, B. Parshall, and L. Scott, 
Derived categories and Morita theory. 
\emph{J. Algebra} \textbf{104} (1986), no. 2, 397--409. 
  
\bibitem{DehyKeller08}  R. Dehy and B. Keller,
On the combinatorics of rigid objects in 2-Calabi-Yau categories.
\emph{Int. Math. Res. Not.} IMRN 2008, no. 11, Art. ID rnn029, 17 pp. 

\bibitem{DerksenFei09}
H. Derksen and J. Fei, General presentations of algebras. Preprint (2009),
  arXiv:0911.4913v2.

\bibitem{DerksenWeymanZelevinsky08}
H. Derksen, J. Weyman, and A. Zelevinsky, Quivers with
  potentials and their representations. {I}. {M}utations. \emph{Selecta Math. (N.S.)}
  \textbf{14} (2008), no.~1, 59--119.

\bibitem{DerksenWeymanZelevinsky10}
\bysame, Quivers with
  potentials and their representations {II}: applications to cluster algebras.
  \emph{J. Amer. Math. Soc.} \textbf{23} (2010), no.~3, 749--790.

\bibitem{Dickson66}
S. E. Dickson, A torsion theory for abelian categories. \emph{Trans. Amer.
  Math. Soc.} \textbf{121} (1966), 223--235.
  
  \bibitem{DupontPerotin:QME}
G. Dupont and M. P{\'e}rotin.
Quiver {M}utation {E}xplorer.
Available at https://github.com/mp-bull/qme-ng, 2012.

\bibitem{Fomin10}
S. Fomin,
Total positivity and cluster algebras. \emph{Proceedings of the International Congress of Mathematicians}. Volume II, 125--145, Hindustan Book Agency, New Delhi, 2010. 


\bibitem{cluster1}
S. Fomin and A. Zelevinsky,
Cluster algebras {I}: Foundations.
\emph{J. Amer. Math. Soc.} \textbf{15} (2002), 497--529.

\bibitem{cluster2}
\bysame,
Cluster algebras {II}: Finite type classification.
\emph{Invent. Math.} \textbf{154} (2003), 63--121.

\bibitem{FominZelevinsky03}
\bysame, Cluster algebras: notes for the CDM-03 conference. \emph{Current developments in mathematics}, 2003, 1--34, Int. Press, Somerville, MA, 2003. 

\bibitem{cluster4}
\bysame,
Cluster algebras {IV}: Coefficients.
\emph{Compos. Math.} \textbf{143} (2007), no.~1, 112--164.



\bibitem{FuKeller10}
C. Fu and B. Keller, On cluster algebras with coefficients
  and 2-{C}alabi-{Y}au categories. \emph{Trans. Amer. Math. Soc.} \textbf{362}
  (2010), no.~2, 859--895.

\bibitem{GMN:WKB}
D. Gaiotto, G. W. Moore, and A. Neitzke,
Wall-crossing, {H}itchin systems, and the {WKB} approximation. \emph{Adv. Math.} \textbf{234} (2013), 239--403.

\bibitem{GeissLeclercSchroer05}
C. Geiss, B. Leclerc, and J. Schr\"oer, Semicanonical bases and preprojective algebras. 
\emph{Ann. Sci. \'Ecole Norm. Sup. (4)} \textbf{38} (2005), no. 2, 193--253. 

\bibitem{GeissLeclercSchroer06}
\bysame, Rigid modules over preprojective algebras.
\emph{Invent. Math.} \textbf{165} (2006), no. 3, 589--632. 

\bibitem{GeissLeclercSchroer11}
\bysame, Kac-Moody groups and cluster algebras.
\emph{Adv. Math.} \textbf{228} (2011), no. 1, 329--433. 

\bibitem{GeissLeclercSchroer10}
\bysame, Generic bases for cluster algebras and the Chamber Ansatz. \emph{J. Amer. Math. Soc.} \textbf{25} (2012), no. 1, 21--76.

\bibitem{GekhtmanShapiroVainshtein08}
M. Gekhtman, M. Shapiro and A. Vainshtein,
On the properties of the exchange graph of a cluster algebra. 
\emph{Math. Res. Lett.} \textbf{15} (2008), no. 2, 321--330. 

\bibitem{Ginzburg06}
V. Ginzburg, {Calabi-Yau} algebras. Preprint (2006), arXiv:math/0612139v3 [math.AG].


\bibitem{Happel87}
D. Happel,
On the derived category of a finite-dimensional algebra. 
\emph{Comment. Math. Helv.} \textbf{62} (1987), no. 3, 339--389. 

\bibitem{HappelReitenSmaloe96}
D. Happel, I. Reiten, and S. O. Smal{\o}, Tilting in abelian
  categories and quasitilted algebras. \emph{Mem. Amer. Math. Soc.} \textbf{120}
  (1996), no.~575, viii+88.



\bibitem{HappelUnger89}
D. Happel and L. Unger,
Almost complete tilting modules,
\emph{Proc. Amer. Math. Soc.} \textbf{107} (1989), no. 3, 603--610.


\bibitem{HappelUnger1}
\bysame,
On the quiver of tilting modules, \emph{J. Algebra} \textbf{284} (2005), no. 2, 857--868.

\bibitem{HappelUnger2}
\bysame, On the set of tilting objects in hereditary categories. \emph{Fields Institute Communications} \textbf{45} (2005), 141--159.


\bibitem{HappelUnger05}
\bysame,
On a partial order of tilting modules.
\emph{Algebr. Represent. Theory} \textbf{8} (2005), no. 2, 147--156. 


\bibitem{Hubery06}
A Hubery, Acyclic cluster algebras via {R}ingel-{H}all algebras.
  Preprint available at http://www1.maths.leeds.ac.uk/{\~{}}ahubery/Cluster.pdf.
  
  
\bibitem{Huybrechts11}
D. Huybrechts, Introduction to stability conditions. \emph{Moduli spaces}, 179--229, London Math. Soc. Lecture Note Ser., \textbf{411}, Cambridge Univ. Press, Cambridge, 2014. See also arXiv:1111.1745.

\bibitem{IngallsThomas09}
C. Ingalls and H. Thomas, 
Noncrossing partitions and representations of quivers. 
\emph{Compos. Math.} 145 (2009), no. 6, 1533--1562. 

\bibitem{IyamaJorgensenYang14}
O. Iyama, P. J{\o}rgensen, and D. Yang, Intermediate co-t-structures, two-term silting objects, $\tau$-tilting modules, and torsion classes. \emph{Algebra and Number Theory} \textbf{8} (2014), no. 10, 2413--2431.

\bibitem{IyamaYang13}
O. Iyama and D. Yang, Silting reduction and Calabi--Yau reduction of triangulated categories. Preprint (2014), arXiv:1408.2678.


\bibitem{IyamaYoshino08}
O. Iyama and Y. Yoshino, Mutations in triangulated categories and rigid Cohen-Macaulay modules. \emph{Invent. Math.} \textbf{172} (2008), 117--168.

\bibitem{Jasso13}
G. Jasso, Reduction of $\tau$-tilting modules and torsion pairs. Preprint (2013), arXiv:1302.2709.

\bibitem{JorgensenPauksztello11}
P. J{\o}rgensen and D. Pauksztello, The co-stability manifold of a
  triangulated category. \emph{Glasg. Math. J.} \textbf{55} (2013), 161--175.
  
  
\bibitem{KalckYang12}
M. Kalck and D. Yang, Relative singularity categories I: Auslander resolutions. Preprint (2012), arXiv:1205.1008v2.


\bibitem{Keller94}
B. Keller, Deriving {D}{G} categories. \emph{Ann. Sci. {\'E}cole Norm.
  Sup. (4)} \textbf{27} (1994), no.~1, 63--102.

\bibitem{Keller05}
\bysame, {On triangulated orbit categories}, \emph{Doc. Math.} \textbf{10}
  (2005), 551--581.

\bibitem{Keller06d}
\bysame, On differential graded categories. \emph{International Congress of
  Mathematicians}. Vol. II, Eur. Math. Soc., Z{\"u}rich, 2006, pp.~151--190.


\bibitem{Keller10}
\bysame, Calabi-Yau triangulated categories. \emph{Trends in representation theory of algebras and related topics}, 467--489, 
EMS Ser. Congr. Rep., Eur. Math. Soc., Z\"urich, 2008. 

\bibitem{Keller10c}
\bysame,
Cluster algebras, quiver representations and triangulated categories. \emph{Triangulated categories}, 76--160, 
London Math. Soc. Lecture Note Ser. 375, Cambridge Univ. Press, Cambridge, 2010. 


\bibitem{Keller11}
\bysame, Deformed {C}alabi-{Y}au completions. \emph{J. Reine Angew. Math.}
  \textbf{654} (2011), 125--180, With an appendix by Michel Van den Bergh.
  
  \bibitem{Keller11b}
\bysame, On cluster theory and quantum dilogarithm identities. \emph{Representations of algebras and related topics}, 85--116, 
EMS Ser. Congr. Rep., Eur. Math. Soc., Z\"urich, 2011. 
  
  \bibitem{Keller:derivedcluster}
\bysame, 
Cluster algebras and derived categories. 
In \emph{Derived categories in algebraic 
geometry}, EMS Ser. Congr. Rep., Eur. Math. Soc., Z\"urich 2012, 123--183.


  \bibitem{Keller:javaapplet}
\bysame, Quiver mutation in {J}ava.
Available at http://www.math.jussieu.fr/{\~{}}keller/quivermutation/.

\bibitem{KellerNicolas11}
B. Keller and P. Nicol{\'a}s, Cluster hearts and cluster tilting
  objects. Work in preparation.

\bibitem{KellerNicolas12}
\bysame, Weight structures and simple dg modules for positive dg
  algebras. \emph{Int. Math. Res. Notices} \textbf{2013} (2013), 1028--1078.

\bibitem{KellerReiten07}
B. Keller and I. Reiten,
Cluster-tilted algebras are Gorenstein and stably Calabi-Yau. 
\emph{Adv. Math.} \textbf{211} (2007), no. 1, 123--151. 

\bibitem{KellerReiten08}
B. Keller and I. Reiten, Acyclic {C}alabi-{Y}au categories.
  \emph{Compos. Math.} \textbf{144} (2008), no.~5, 1332--1348, With an appendix by
  Michel Van den Bergh.

\bibitem{KellerVossieck88}
B. Keller and D. Vossieck, Aisles in derived categories. \emph{Bull.
  Soc. Math. Belg. S{\'e}r. A} \textbf{40} (1988), no.~2, 239--253.

\bibitem{KellerYang11}
B. Keller and D. Yang, Derived equivalences from mutations of
  quivers with potential. \emph{Adv. Math.} \textbf{226} (2011), no.~3, 2118--2168.

\bibitem{KingQiu11}
A. King and Y. Qiu, Oriented exchange graphs of acyclic Calabi-Yau
  categories. Preprint (2011), arXiv:1109.2924.

\bibitem{KoenigYang12}
S. Koenig and D. Yang, Silting objects, simple-minded collections,
  $t$-structures and co-$t$-structures for finite-dimensional algebras. \emph{Doc. Math.} \textbf{19} (2014), 408--438.

\bibitem{KontsevichSoibelman08}
M. Kontsevich and Y. Soibelman, Stability structures, motivic
  {D}onaldson-{T}homas invariants and cluster transformations. Preprint (2008),
  arXiv:0811.2435.
  

\bibitem{Ladkani07}
S. Ladkani, Universal derived
equivalences of posets of cluster tilting objects. Preprint (2007), arXiv:0710.2860.
  
\bibitem{Leclerc10} 
B. Leclerc, 
Cluster algebras and representation theory. \emph{Proceedings of the International Congress of Mathematicians}. Volume IV, 2471--2488, Hindustan Book Agency, New Delhi, 2010. 

\bibitem{LiuVitoriaYang12}
Q. Liu, J. Vit{\'o}ria, and D. Yang, Gluing silting objects. \emph{Nagoya Math. J.} \textbf{216} (2014), 117--151.

\bibitem{MendozaSaenzSantiagoSouto10}
O. Mendoza, E. C. S\'aenz, V. Santiago, and M.
  J. Souto Salorio, Auslander-buchweitz context and
  co-$t$-structures. Appl. Categor. Struct. \textbf{21} (2013), 417--440.

\bibitem{Mizuno13}
Y. Mizuno, Classifying $\tau$-tilting modules over the preprojective algebra of Dynkin quivers. \emph{Math. Z.} \textbf{277} (2014), no. 3-4, 665--690.

\bibitem{Nagao10}
K. Nagao, Donaldson-{T}homas theory and cluster algebras. \emph{Duke Math. J.} \textbf{162} (2013), no. 7, 1313--1367.

\bibitem{Najera:cvectors}
A. N{\'a}jera~Ch{\'a}vez,
On the c-vectors of an acyclic cluster algebra. Preprint (2012)
 arXiv:1203.1415v1 [math.RT].

\bibitem{NZ:tropicalduality}
T. Nakanishi and A. Zelevinsky,
On tropical dualities in cluster algebras.
\emph{Contemp. Math.} \textbf{565} (2012), 217--226.

\bibitem{Palu08}
Y. Palu, Cluster characters for $2$-{C}alabi-{Y}au triangulated
  categories. \emph{Ann. Inst. Fourier (Grenoble)} \textbf{58} (2008), no.~6,
  2221--2248.

\bibitem{Pauksztello08}
D. Pauksztello, Compact corigid objects in triangulated categories and
  co-{$t$}-structures. \emph{Cent. Eur. J. Math.} \textbf{6} (2008), no.~1, 25--42.

\bibitem{Plamondon11b}
P.-G. Plamondon, Cluster algebras via cluster categories with
  infinite-dimensional morphism spaces. \emph{Compos. Math.} \textbf{147} (2011),
  no.~6, 1921--1934.

\bibitem{Plamondon11}
\bysame, Cluster characters for cluster categories with
  infinite-dimensional morphism spaces. \emph{Adv. Math.} \textbf{227} (2011), no.~1,
  1--39.
  
  \bibitem{Plamondon11c}
\bysame, Generic bases for cluster algebras from the cluster category. \emph{Int. Math. Res. Not.} IMRN 2013, no. 10, 2368--2420.
  
\bibitem{Reiten10}
I. Reiten,
Cluster categories. \emph{Proceedings of the International Congress of Mathematicians}. Volume I, 558--594, Hindustan Book Agency, New Delhi, 2010. 

\bibitem{Rickard89}
J. Rickard, Morita theory for derived categories.
\emph{J. London Math. Soc. (2)} \textbf{39} (1989), no. 3, 436--456.

\bibitem{Rickard02}
\bysame, Equivalences of derived categories for symmetric
  algebras. \emph{J. Algebra} \textbf{257} (2002), no.~2, 460--481.

\bibitem{RickardRouquier10}
J. Rickard and R. Rouquier, Stable categories and
  reconstruction. Preprint (2010), arXiv:1008.1976.

\bibitem{RickardSchofield89}
J. Rickard and A. Schofield, Cocovers and tilting modules. \emph{Math.
  Proc. Cambridge Philos. Soc.} \textbf{106} (1989), no.~1, 1--5.


\bibitem{RiedtmannSchofield91}
C. Riedtmann and A. Schofield,
On a simplicial complex associated with tilting modules.
\emph{Comment. Math. Helv.} \textbf{66} (1991), no. 1, 70--78. 

\bibitem{Ringel07}
C. M. Ringel,
Appendix: Some remarks concerning tilting modules and tilted algebras. Origin. Relevance. Future. \emph{Handbook of tilting theory}, 413--472, 
London Math. Soc. Lecture Note Ser. 332, Cambridge Univ. Press, Cambridge, 2007.



\bibitem{SeidelThomas01}
P. Seidel and R. Thomas, Braid group actions on derived categories
  of coherent sheaves. \emph{Duke Math. J.} \textbf{108} (2001), no.~1, 37--108.

\bibitem{ST:acyclic}
D. Speyer and H. Thomas,
Acyclic cluster algebras revisited. 
In \emph{Algebras, quivers and representations}, Abel Symposia 8, Springer, Berlin 2013, 275--298.
 
 \bibitem{Unger96a}
L. Unger, 
The partial order of tilting modules for three-point-quiver algebras. \emph{Representation theory of algebras (Cocoyoc, 1994)}, 671--679, 
CMS Conf. Proc. 18, Amer. Math. Soc., Providence, RI, 1996. 

\bibitem{Unger96b}
\bysame, The simplicial complex of tilting modules over quiver algebras. 
\emph{Proc. Lond. Math. Soc. (3)} \textbf{73} (1996), no. 1, 27--46. 

\bibitem{Unger07}
\bysame, Combinatorial aspects of the set of tilting modules. \emph{Handbook of tilting theory}, 259--277, 
London Math. Soc. Lecture Note Ser. 332, Cambridge Univ. Press, Cambridge, 2007. 


\bibitem{Wei11}
J. Wei, Semi-tilting complexes. \emph{Israel J. Math.}
\textbf{194} (2013), no. 2, 871--893.

\bibitem{Woolf10}
J. Woolf, Stability conditions, torsion theories and tilting. \emph{J.
  London Math. Soc. (2)} \textbf{82} (2010), no.~3, 663--682.

\bibitem{Xie}
D. Xie, BPS spectrum, wall crossing
and quantum dilogarithm identity. Preprint (2012), arXiv:1211.7071.


\bibitem{Zelevinsky07}
A. Zelevinsky, What is ... a cluster algebra?
\emph{Notices Amer. Math. Soc.} \textbf{54} (2007), no. 11, 1494--1495. 

\bibitem{ZhouZhu11b}
Y. Zhou and B. Zhu, Mutation of torsion pairs in triangulated categories
  and its geometric realization. Preprint (2011), arXiv:1105.3521.

\bibitem{ZhouZhu11}
\bysame, Maximal rigid subcategories in 2-{C}alabi-{Y}au
  triangulated categories.
 \\ \emph{J. Algebra} \textbf{348} (2011), 49--60.
  



\end{thebibliography}

\def\cprime{$'$}
\providecommand{\bysame}{\leavevmode\hbox to3em{\hrulefill}\thinspace}
\providecommand{\MR}{\relax\ifhmode\unskip\space\fi MR }
\providecommand{\MRhref}[2]{%
  \href{http://www.ams.org/mathscinet-getitem?mr=#1}{#2}
}
\providecommand{\href}[2]{#2}

\end{document}